\documentclass[10pt]{amsart}
\usepackage{graphicx}

\usepackage{latexsym,amsmath}
\usepackage{amsmath,amsthm}
\usepackage{amsfonts}
\usepackage[psamsfonts]{amssymb}
\usepackage{enumerate}
\usepackage{url}
\usepackage{enumerate}
\usepackage{pstricks}

\usepackage[backref=page,
colorlinks=true,linkcolor=red,citecolor=green,urlcolor=black]{hyperref}

\theoremstyle{plain} 
\newtheorem{theorem}{Theorem}[section]
\newtheorem{corollary}[theorem]{Corollary}

\newtheorem{lemma}[theorem]{Lemma}  
\newtheorem{proposition}[theorem]{Proposition}
\newtheorem{constr}{Construction}

\theoremstyle{definition} 

\theoremstyle{definition} 
\newtheorem{example}[theorem]{Example}
\newtheorem{remark}[theorem]{Remark}
\newtheorem*{remark*}{Remark}
\numberwithin{equation}{section}

\newcommand{\FF}[1]{\mathcal{F}^{#1}}
\renewcommand{\H}[1]{\mathcal{H}_+^{#1}}
\renewcommand{\r}{\mathsf{r}}
\newcommand{\sss}{\mathsf{s}}
\newcommand{\aaa}{\mathsf{a}}
\newcommand{\w}{\mathsf{w}}

\newcommand{\M}{\mathfrak{M}}

\newcommand{\sign}{\operatorname{sign}}
\newcommand{\mes}{\operatorname{mes}}
\newcommand{\supp}{\operatorname{supp}}
\newcommand{\lc}{\mathsf{L\!C}}

\renewcommand{\le}{\leqslant}
\renewcommand{\ge}{\geqslant}

\newcommand{\lp}{\left(}
\newcommand{\rp}{\right)}

\newcommand{\al}{\alpha}
\newcommand{\ga}{\gamma}
\newcommand{\si}{\sigma}
\newcommand{\la}{\lambda}
\newcommand{\ka}{\kappa}
\newcommand{\de}{\delta}
\newcommand{\vpi}{\varphi}

\renewcommand{\P}{\operatorname{\mathsf{P}}} 

\newcommand{\E}{\operatorname{\mathsf{E}}}

\newcommand{\R}{\mathbb{R}}

\newcommand{\EE}{\mathcal{E}}
\newcommand{\I}{\mathcal{I}}

\newcommand{\B}{\mathcal{B}}

\newcommand{\D}{\overset{\mathrm{D}}=}

\newcommand{\tG}{\tilde{G}}
\newcommand{\tF}{\tilde{F}}

\newcommand{\tX}{\tilde{X}}
\newcommand{\tx}{\tilde{x}}
\newcommand{\ty}{\tilde{y}}
\newcommand{\tr}{\tilde{r}}
\newcommand{\trr}{\tilde{\r}}
\newcommand{\tmu}{\tilde{\mu}}
\newcommand{\tnu}{\tilde{\nu}}

\newcommand{\ts}{\tilde{\sss}}

\renewcommand{\d}{\mathrm{d}}

\newcommand{\as}{\stackrel{\mathrm{a.s.}}{=}}

\newcommand{\vp}{\varepsilon}

\newcommand{\iii}{\operatorname{I}}
\newcommand{\ii}[1]{\iii\{#1\}}

\begin{document}


\begin{abstract}
For any continuous zero-mean random variable (r.v.) $X$, a \emph{reciprocating} function $\r$ is constructed, based only on the distribution of $X$, such that the conditional distribution of $X$ given the  (at-most-)two-point set $\{X,\r(X)\}$ is the zero-mean distribution on this set;   
in fact, a more general construction without the continuity assumption is given in this paper, as well as  
a large variety of other related results, including  characterizations of the reciprocating function and modeling distribution asymmetry patterns.  
The mentioned disintegration of zero-mean r.v.'s implies, in particular, that an arbitrary zero-mean distribution is represented as the mixture of two-point zero-mean distributions; moreover, 
this mixture representation is most symmetric in a variety of senses. 
Somewhat similar representations -- of any probability distribution as the mixture of two-point distributions with the same skewness coefficient (but possibly with different means) -- go back to Kolmogorov; very recently, Aizenman \emph{et al.} further developed such representations and applied them to (anti-)concentration inequalities for functions of independent random variables and to spectral localization for random Schroedinger operators. 
One kind of application given in the present paper is to construct certain statistical 
tests for asymmetry patterns and for location without symmetry conditions. Exact inequalities implying conservative properties of such tests are presented. These developments extend results established earlier by Efron, Eaton, and Pinelis under a symmetry condition.

\end{abstract}

\title[Disintegration of random variables]{{\Large 
Optimal two-value zero-mean disintegration of 
zero-mean random variables}}
\date{\today; \emph{file}: \jobname.tex}
\author{Iosif Pinelis}
\address{ Department of Mathematical Sciences\\
Michigan Technological University\\
Hough\-ton, Michigan 49931 }
\email{ipinelis@mtu.edu}
\keywords{Disintegration of measures, Wasserstein metric, Kantorovich-Rubinstein theorem, transportation of measures, optimal matching, most symmetric, hypothesis testing, confidence regions, Student's $t$-test, asymmetry, exact inequalities, conservative properties}
\subjclass[2000]{Primary: 28A50, 60E05, 60E15, 62G10, 62G15, 62F03, 62F25. Secondary: 49K30, 49K45, 49N15, 60G50, 62G35, 62G09, 90C08, 90C46}

\vspace*{-.9cm}

\maketitle

\vspace*{-.9cm}

\tableofcontents



\makeatletter
\makeatother

\section{Introduction} \label{sec:intro}


Efron \cite{efron} considered the so-called self-normalized sum
\begin{equation}\label{eq:S}
S:=\frac{X_1+\dots+X_n}{\sqrt{X_1^2+\dots+X_n^2}},
\end{equation}
assuming that the $X_i$'s are any random variables (r.v.'s) satisfying the orthant symmetry condition:
the joint distribution of $\eta_1X_1,\dots,\eta_n X_n$ is the same for any choice of signs $\eta_1,\dots,\eta_n$ in the set $\{1,-1\}$, so that, in particular, each $X_i$ is symmetric(ally distributed). It suffices that the $X_i$'s be independent and symmetrically (but not necessarily identically) distributed. 
On the event $\{X_1=\dots=X_n=0\}$, $S:=0$.


Following Efron \cite{efron}, note that the conditional distribution of any symmetric r.v.\ $X$ given $|X|$ is the symmetric distribution on the (at-most-)two-point set $\{|X|,-|X|\}$.
Therefore, under the orthant symmetry condition, the distribution of $S$
is a mixture of the distributions of the normalized Khinchin-Rademacher sums 
$\vp_1a_1+\dots+\vp_na_n$, where the $\vp_i$'s are 
independent Rademacher r.v.'s, with $\P(\vp_i=1)=\P(\vp_i=-1)=\frac12$ for all $i$, which are also
independent of the $X_i$'s, and $a_i=X_i/(X_1^2+\dots+X_n^2)^{\frac12}$, so that $a_1^2+\dots+a_n^2=1$ (except on the event $\{X_1=\dots=X_n=0\}$, where $a_1=\dots=a_n=0$).

Here and in what follows, let $Z$ stand for a standard normal r.v. 
Let now $a_1,\dots,a_n$ be any real numbers such that 
$a_1^2+\dots+a_n^2=1.$ The sharp form, 
\begin{equation}\label{eq:khin}
\E f\lp\vp_1a_1+\dots+\vp_n a_n\rp\le\E f(Z),
\end{equation}
of Khinchin's inequality \cite{kh} for $f(x)\equiv |x|^p$ 
was proved by
Whittle (1960) \cite{whittle} for $p\ge3$ and Haagerup (1982) \cite{haag} for $p\ge2$. 
For $f(x)\equiv e^{\la x}$ ($\la\ge0$), inequality \eqref{eq:khin} follows from Hoeffding (1963) \cite{hoeff}, whence
\begin{equation}\label{eq:exp}
\P\lp\vp_1a_1+\dots+\vp_n a_n\ge x\rp\le 
\inf_{\la\ge0}\frac{\E e^{\la Z}}{e^{\la x}}
=e^{-x^2/2}\quad \forall x\ge0.
\end{equation}

As noted by Efron \cite{efron}, inequalities \eqref{eq:khin} and \eqref{eq:exp} together with the mentioned mixture representation imply
\begin{equation}\label{eq:khin-V}
\E e^{\la S}\le\E e^{\la Z}\quad\forall\la\ge0
\end{equation}
and
\begin{equation}\label{eq:exp-V}
\P\lp S\ge x\rp\le e^{-x^2/2}\quad \forall x\ge0.
\end{equation}
These results can be easily restated in terms of Student's statistic $T$, which is a monotonic function of $S$, as noted by Efron: $T=\sqrt{\frac{n-1}n}\,S/\sqrt{1-S^2/n}$.

Eaton (1970) \cite{eaton1} proved the Khinchin-Whittle-Haagerup inequality \eqref{eq:khin} for a rich class of moment functions, which essentially coincides with the class $\FF3$ of all convex functions $f$ with a convex second derivative $f''$; see \cite[Proposition~A.1]{pin94} and also \cite{asymm}.
Based on this extension of \eqref{eq:khin}, inequality \eqref{eq:exp} was improved in \cite{eaton1,eaton2,pin94}. In particular,
Pinelis (1994) \cite{pin94} obtained the following improvement of a conjecture by Eaton (1974) \cite{eaton2}: 
\begin{equation*}
\P\lp\vp_1a_1+\dots+\vp_n a_n\ge x\rp\le\tfrac{2e^3}9\,\P(Z\ge x)\quad
\forall x\in\R. 
\end{equation*} 

Thus, inequalities \eqref{eq:khin-V} and \eqref{eq:exp-V} can be improved as follows: 
\begin{equation}\label{eq:khin-V-pin94}
\E f(S)\le\E f(Z)\quad\forall f\in\FF3
\end{equation} 
and
\begin{equation}\label{eq:exp-V-pin94}
\P\lp S\ge x\rp\le\min\big(\tfrac{2e^3}9\,\P(Z\ge x),e^{-x^2/2}\big)\quad
\forall x\ge0.
\end{equation}
Multivariate extensions of these results, which can be expressed in terms of Hotel\-ling's statistic in place of Student's, were also obtained in \cite{pin94}. 

It was pointed out in \cite[Theorem~2.8]{pin94} that, since the normal tail decreases fast, inequality \eqref{eq:exp-V-pin94} implies that relevant quantiles of $S$ may exceed the corresponding standard normal quantiles only by a relatively small amount, so that one can use \eqref{eq:exp-V-pin94} rather efficiently to test symmetry even for non-i.i.d.\ observations. 

Here we shall present extensions of inequalities \eqref{eq:khin-V-pin94} and \eqref{eq:exp-V-pin94} to the case when the $X_i$'s are not symmetric. 
This paper is an improvement of preprint \cite{stud}: the results are now much more numerous and comprehensive, and also somewhat more general, while the proof of the basic result  (done here using a completely different method) is significantly shorter. A brief account of results of \cite{stud} (without proofs) was presented in \cite{asymm}. 


Our basic idea is to represent any zero-mean, possibly asymmetric,  distribution as an appropriate mixture of two-point zero-mean distributions.  
Let us assume at this point that a zero-mean r.v.\ $X$ has an everywhere continuous and strictly increasing distribution function (d.f.). Consider the truncated r.v.\ $\tilde X_{a,b}:=X\ii{a\le X\le b}$. 
\big(Here and in what follows  
$\ii{A}$ stands, as usual, for the indicator of a given assertion $A$, so that $\ii{A}=1$ if $A$ is true and $\ii{A}=0$ if $A$ is false.\big)
Then, for every fixed $a\in(-\infty,0]$, the function $b\mapsto\E \tilde X_{a,b}$ is continuous and increasing on the interval $[0,\infty)$ from $\E \tilde X_{a,0}\le0$ to 
$\E \tilde X_{a,\infty}>0$. Hence, for each $a\in(-\infty,0]$, there exists a unique value $b\in[0,\infty)$ such that $\E \tilde X_{a,b}=0$. Similarly, for each $b\in[0,\infty)$, there exists a unique value $a\in(-\infty,0]$ such that 
$\E \tilde X_{a,b}=0$. That is, one has a one-to-one correspondence between $a\in(-\infty,0]$ and $b\in[0,\infty)$ such that $\E \tilde X_{a,b}=0$. Denote by 
$\r=\r_X$ the {\em reciprocating} function defined on $\R$ and carrying this correspondence, so that 
$$\E X\ii{\text{$X$ is between $x$ and $\r(x)$} }=0\quad\forall x\in\R;$$
the function $\r$ is decreasing on $\R$ and such that $\r(\r(x))=x$ $\forall x\in\R$; moreover, $\r(0)=0$. 
(Clearly, $\r(x)=-x$ for all real $x$ if the r.v.\ $X$ is also  symmetric.)
Thus, the set 
$\big\{\,\{x,\r(x)\}\colon x\in\R\,\big\}$ of two-point sets constitutes a partition of $\R$. 
One can see that the conditional distribution of the zero-mean r.v.\ $X$ given the random two-point set $\{X,\r(X)\}$ %
is the uniquely determined zero-mean distribution on the set $\{X,\r(X)\}$.

It follows that the distribution of the zero-mean r.v.\ $X$ with a continuous strictly increasing d.f.\ is represented as a mixture of two-point zero-mean distributions. 
A somewhat similar representation -- of any probability distribution as the mixture of two-point distributions with the same skewness coefficient $\frac{q-p}{\sqrt{pq}}$ (but possibly with different means) -- goes back to Kolmogorov; very recently Aizenman \emph{et al.} \cite{aizen} further developed this representation and applied it to (anti-)concentration inequalities for functions of independent random variables and to spectral localization for random Schroedinger operators.

In accordance with their purposes, instead of r.v.'s $\tilde X_{a,b}=X\ii{a\le X\le b}$ Aizenman \emph{et al.} \cite{aizen} (who refer to $a$ and $b$ as markers) essentially deal with r.v.'s (i)  $\ii{X\le a}-\ii{X>b}$ (in a case of markers moving in opposite directions) and with (ii) $\ii{X\le a}-\ii{q_{1-p}<X\le b}$ (in a case of markers moving in the same direction, where $q_{1-p}$ is a $(1-p)$-quantile of the distribution of $X$). 
The construction described above in terms of $\tilde X_{a,b}=X\ii{a\le X\le b}$ corresponds, clearly, to the case of opposite-moving markers. 

While an analogous same-direction zero-mean disintegration is possible, we shall not deal with it in this paper. 
For a zero-mean distribution, the advantage of an opposite-directions construction is that the resulting two-point zero-mean distributions are less asymmetric than those obtained by using a same-direction method (in fact, we shall show that our opposite-directions disintegration is most symmetric, in a variety of senses). On the other hand, the same-direction method will produce two-point zero-mean distributions that are more similar to one another in width. 
Thus, in our main applications -- to self-normalized sums, the advantages of opposite-directions appear to be more important, since the distribution of a self-normalized sum is much more sensitive to the asymmetry than to the inhomogeneity of the constituent two-point distributions in width; this appears to matter more in the setting of Corollary~\ref{cor:stud-asymm} than in the one of Corollary~\ref{cor:student-normal}. 


These mixture representations of a distribution are similar to the  representations of the points of a convex compact set as mixtures of the extreme points of the set; 
the existence of such representations is provided by the celebrated Krein-Milman-Choquet-Bishop-de Leeuw 
(KMCBdL) theorem;  
concerning ``non-compact'' versions of this theorem see  e.g.\ \cite{oates}. 
In our case, the convex set would be the set of all zero-mean distributions on $\R$.  
However, in contrast with the KMCBdL-type pure-existence theorems, the representations given in \cite{stud}, \cite{aizen}, and this paper are constructive, specific, and, as shown here, optimal, in a variety of senses.

Moreover, in a certain sense \cite{stud} and this paper provide disintegration of r.v.'s rather than that of their distributions, as the two-point set $\{x,r(x)\}$ is a function of the observed value $x$ of the r.v.\ $X$.     
This makes it convenient to construct statistical 
tests for asymmetry patterns and for location without symmetry conditions. Exact inequalities implying conservative properties of such tests will be given in this paper. These developments extend the mentioned results established earlier by Efron, Eaton, and Pinelis under the orthant symmetry condition. 

More specifically, one can construct generalized versions of the self-normalized sum \eqref{eq:S}, which require -- instead of the symmetry of independent r.v.'s $X_i$ -- only that the $X_i$'s be zero-mean:
\begin{equation*} 
S_W:=\frac{X_1+\dots+X_n}{\frac12\sqrt{W_1^2+\dots+W_n^2}}\quad\text{and}\quad
S_{Y,\la}:=\frac{X_1+\dots+X_n}{(Y_1^\la+\dots+Y_n^\la)^{\frac1{2\la}}},
\end{equation*}
where $\la>0$,
$W_i:=|X_i-\r_i(X_i)|$ and $Y_i:=|X_i\,\r_i(X_i)|$, 
and the reciprocating function $\r_i:=\r_{X_i}$ is constructed as above, based on the distribution of $X_i$, for each $i$, so that the $\r_i$'s may be different from one another if the $X_i$'s are not identically distributed.
Note that $S_W=S_{Y,1}=S$ (recall here \eqref{eq:S}) when the $X_i$'s are symmetric.
Logan \emph{et al} \cite{logan} and Shao \cite{shao} obtained limit theorems for the ``symmetric'' version of $S_{Y,\la}$ (with $X_i^2$ in place of $Y_i$), whereas the $X_i$'s were not assumed to be symmetric. 

Corollaries~\ref{cor:student-normal} and \ref{cor:stud-asymm} in Subsection~\ref{subsec:several} of this paper suggest that statistical tests based on the ``corrected for asymmetry'' statistics $S_W$ and $S_Y$ have desirable conservativeness and similarity properties, which could result in greater power; further studies are needed here. \big(Recall that a test is referred to as (approximately) similar if the type I error probabilities are (approximately) the same for all distributions corresponding to the null hypothesis.
\big)

Actually, in this paper we provide two-point zero-mean disintegration of any zero-mean r.v.\ $X$, with a d.f.\ not necessarily continuous or strictly increasing. Toward that end, randomization \big(by means of a r.v.\ uniformly distributed in interval $(0,1)$\big) is used to deal with the atoms of the distribution of r.v.\ $X$, and generalized inverse functions to deal with the intervals on which the d.f.\ of $X$ is constant.  

Note that the reciprocating function $\r$ depends on the usually unknown in statistics distribution of the underlying r.v.\ $X$. However, if e.g.\ the $X_i$'s constitute an i.i.d.\ sample, then the function $G$ defined in the next section by \eqref{eq:G(x)} can be estimated based on the sample, so that one can estimate the reciprocating function $\r$. Thus, replacing $X_1+\dots+X_n$ in the numerators of $S_W$ and $S_{Y,\la}$ by $X_1+\dots+X_n-n\theta$,
one obtains approximate pivots to be used to construct confidence intervals or, equivalently, tests for an unknown mean $\theta$. 
One can also use bootstrap to estimate the distributions of such approximate pivots.

\section{Statements of main 
results on disintegration} \label{sec:results}
\subsection{Two-value zero-mean disintegration of one zero-mean r.v.} \label{subsec:one}
Let $\nu$ be any (nonnegative finite) measure defined on $\B(\R)$, where $\B(E)$ stands for the set of all Borel subsets of a given set $E$. 
Sometimes it will be convenient to consider such a measure $\nu$ extended to $\B([-\infty,\infty])$ so that, naturally, $\nu(\{-\infty\})=\nu(\{\infty\})=0$. 
Consider the function $G=G_\nu$ with values in $[0,\infty]$ defined by the formula 
\begin{equation}\label{eq:G(x)}
	G(x):=G_\nu(x):=
\begin{cases}
\int_{(0,x]}z\,\nu(\d z) & \text{ if }z\in[0,\infty], \\
\int_{[x,0)}(-z)\,\nu(\d z) & \text{ if }z\in[-\infty,0]. 
\end{cases}
\end{equation}
Note that 
\begin{equation}\label{eq:G properties}
\begin{aligned}
	G(0)=0;\ & \text{$G$ is non-decreasing on $[0,\infty]$ and right-continuous on $[0,\infty)$;} \\ 
	\text{and } & \text{$G$ is non-increasing on $[-\infty,0]$ and left-continuous on $(-\infty,0]$;  }
	\end{aligned}
\end{equation}
in particular, $G$ is continuous at $0$.

Define next the positive and negative generalized inverses $x_+$ and $x_-$ of the function $G$:
\begin{align}
x_+(h) & :=x_{+,\nu}(h):= \inf\{x\in[0,\infty]\colon G_\nu(x)\ge h\}, \label{eq:x+} \\
x_-(h)& :=x_{-,\nu}(h):= \sup\{x\in[-\infty,0]\colon G_\nu(x)\ge h\}, \label{eq:x-}
\end{align}
for any $h\in[-\infty,\infty]$; here, as usual, $\inf\emptyset:=\infty$ and $\sup\emptyset:=-\infty$.

Introduce also a ``randomized'' version of $G$: 
\begin{equation} \label{eq:H}
\tG(x,u):=\tG_\nu(x,u):=
\begin{cases}
G_\nu(x-)+(G_\nu(x)-G_\nu(x-))\,u & \text{ if }x\in[0,\infty], \\
G_\nu(x+)+(G_\nu(x)-G_\nu(x+))\,u & \text{ if }x\in[-\infty,0] 
\end{cases}
\end{equation}
and what we shall refer to as the \emph{reciprocating} function $\r=\r_\nu$ for the measure $\nu$: 
\begin{equation} \label{eq:r}
\r(x,u):=\r_\nu(x,u):=
\begin{cases}
x_{-,\nu}(\tG_\nu(x,u)) & \text{ if }x\in[0,\infty], \\
x_{+,\nu}(\tG_\nu(x,u)) & \text{ if }x\in[-\infty,0], \\
\end{cases}
\end{equation}
for all $u\in[0,1]$. 

\begin{remark}\label{rem:cont} 
\begin{enumerate}[(i)]
	\item \label{Borel}
	The function $\tG$ is Borel(-measurable), since each of the functions $G$, $G(\cdot\,+)$, $G(\cdot\,-)$ is monotonic on $[0,\infty)$ and $(-\infty,0]$ and hence Borel. Therefore and  
by property \eqref{x+-non-decr} of Proposition~\ref{lem:left-cont}, stated in the next section, the reciprocating function $\r$ is Borel, too.
		\item \label{r for cont}
Also, $\tG(x,u)$ and hence $\r(x,u)$ depend on $u$ for a given value of $x$ only if $\nu(\{x\})\ne0$.
Therefore, let us write simply $\r(x)$ in place of $\r(x,u)$ in the case when the measure $\nu$ is non-atomic.
\end{enumerate}
\end{remark}

If $\nu$ is the measure $\mu=\mu_X$ that is the distribution of a  r.v.\ $X$, then we may use subscript ${}_X$ with $G$, $\tG$, $\r$, $x_\pm$ in place of subscript ${}_\mu$ (or no subscript at all).  

In what follows, $X$ will by default denote an arbitrary   \emph{zero-mean} real-valued r.v., which will be usually thought of as fixed.
Then, for $G=G_X$,  
\begin{equation} \label{eq:m} 
G(\infty) = G(-\infty) =G(\infty-) = G\big((-\infty)+\big) =\tfrac12 \E |X|=: m<\infty. 
\end{equation}

Let $U$ 
stand for any r.v.\ which is independent of $X$ and uniformly distributed on the unit interval $[0,1]$.

For any $a$ and $b$ in $\R$ such that $ab\le0$, let $X_{a,b}$ denote any \emph{zero-mean} r.v.\ 
with values in 
the two-point set $\{a,b\}$; note that such a r.v.\ $X_{a,b}$ exists and, moreover, its distribution is uniquely determined: 
\begin{equation}\label{eq:Xab}
\P(X_{a,b}=a)=\tfrac b{b-a}\quad\text{and}\quad
\P(X_{a,b}=b)=\tfrac a{a-b}	
\end{equation}
if $a\ne b$, and $X_{a,b}=0$ almost surely (a.s.) if $a=b(=0)$; then in fact $X_{a,b}=0$ a.s.\ whenever $ab=0$. 
Along with the r.v.\ $X_{a,b}$, consider 
\begin{equation}\label{eq:r_ab}
	R_{a,b}:=\r_{a,b}(X_{a,b},U)
\end{equation}
provided that $U$ does not depend on $X_{a,b}$, where $\r_{a,b}:=\r_{X_{a,b}}$, the reciprocal function for $X_{a,b}$. Note that, 
if $ab=0$, then $R_{a,b}=0=X_{a,b}$ a.s. 
If $ab<0$, then $R_{a,b}=b$ a.s.\ on the event $\{X_{a,b}=a\}$, and $R_{a,b}=a$ 
a.s.\ on the event $\{X_{a,b}=b\}$, so that the random set $\{X_{a,b},R_{a,b}\}$ coincides a.s.\ with the nonrandom set $\{a,b\}$. However, $R_{a,b}$ equals in distribution to $X_{a,b}$ only if $a+b=0$, that is, only if $X_{a,b}$ is symmetric; moreover, in contrast with $X_{a,b}$, the r.v.\ $R_{a,b}$ is zero-mean only if $a+b=0$. 
Clearly, $(X_{a,b},R_{a,b})\D(X_{b,a},R_{b,a})$ whenever $ab\le0$. 

We shall prove that the conditional distribution of $X$ given the two-point random set $\{X,\r(X,U)\}$ is the zero-mean distribution on this set:
\begin{equation}\label{eq:cond}
	\big(X\,\big|\,\{X,\r(X,U)\}
	=\{a,b\}\big)\D X_{a,b}.
\end{equation}
In fact, we shall prove a more general result: 
that the conditional distribution of the ordered pair $\big(X,\r(X,U)\big)$ given that 
$\{X,\r(X,U)\}=\{a,b\}$ is the distribution of the ordered pair $\big(X_{a,b},R_{a,b}\big)$:
\begin{equation}\label{eq:cond-pair}
	\Big(\big(X,\r(X,U)\big)\,\Big|\,\{X,\r(X,U)\}
	=\{a,b\}\Big)\D \big(X_{a,b},R_{a,b}\big).
\end{equation}
Formally, this basic result of the paper is expressed as

\begin{theorem}\label{th:main}
Let $g\colon\R^2\to\R$ be any Borel function bounded from below (or from above). Then
\begin{equation}\label{eq:main}
	\E g\big(X,\r(X,U))
	=\int_{\R\times[0,1]} 
	\E g\big(X_{x,\r(x,u)},R_{x,\r(x,u)}\big)\,\P(X\in \d x)\,\d u.
\end{equation}
Instead of the condition that $g$ be bounded from below or above, it is enough to require only that $g(x,r)-cx$ be so for some real constant $c$ over all real $x,r$. 
\end{theorem}

The proofs (whenever necessary) are deferred to Section~\ref{sec:proofs}. 

As one can see, 
Theorem~\ref{th:main} provides a complete description of the distribution of the ordered random pair $\big(X,\r(X,U)\big)$ -- as a mixture of two-point distributions on $\R^2$; each of these two-point distributions is supported by a two-point subset of $\R^2$ of the form $\{(a,b),(b,a)\}$ with $ab\le0$, and at that the mean of the projection of this two-point distribution onto the first coordinate axis is zero. As special cases, Theorem~\ref{th:main} contains  descriptions of the individual distributions of the r.v.'s $X$ and $\r(X,U)$ as mixtures of two-point distributions on $\R$:
for any Borel function $g\colon\R\to\R$ bounded from below (or from above) one has 
\begin{align}
	\E g(X)=\int_{\R\times[0,1]} 
	\E g\big(X_{x,\r(x,u)}\big)\,\P(X\in \d x)\,\d u
	\label{eq:Eg(X)};\\
	\E g\big(\r(X,U)\big)=\int_{\R\times[0,1]} 
	\E g\big(R_{x,\r(x,u)}\big)\,\P(X\in \d x)\,\d u.
	\notag
\end{align}	
This 
is illustrated by 

\begin{example}\label{ex:discrete}
Let $X$ have the discrete distribution
$\frac5{10}\,\de_{-1}+\frac1{10}\,\de_0+\frac3{10}\,\de_1+\frac1{10}\,\de_2$ on the finite set $\{-1,0,1,2\}$, where $\de_a$ denotes the (Dirac) probability distribution on the singleton set $\{a\}$. Then $m=\frac5{10}$ and, for $x\in\R$, $u\in[0,1]$, and $h\in[0,m]$,
\begin{gather*}
G(x)=\tfrac5{10}\ii{x\le-1}+\tfrac3{10}\ii{1\le x<2}+\tfrac5{10}\ii{2\le x},\\
x_+(h)=\ii{0<h\le\tfrac3{10}}+2\ii{\tfrac3{10}<h},\quad
x_-(h)=-\ii{0<h},\\
\tG(-1,u)=\tfrac5{10}\,u,\quad
\tG(0,u)=0,\quad
\tG(1,u)=\tfrac3{10}\,u,\quad
\tG(2,u)=\tfrac3{10}+\tfrac2{10}\,u,\\
\r(-1,u)=\ii{u\le\tfrac35}+2\ii{u>\tfrac35},\ \r(0,u)=0,\ \r(1,u)=-1,\ \r(2,u)=-1.
\end{gather*}
Therefore, the distribution of the random set $\{X,\r(X,U)\}$ is $\frac6{10}\,\de_{\{-1,1\}}+\frac3{10}\,\de_{\{-1,2\}}+\frac1{10}\,\de_{\{0\}}$, and the conditional distributions of $X$ given $\{X,\r(X,U)\}=\{-1,1\}$, \break
$\{X,\r(X,U)\}=\{-1,2\}$, and $\{X,\r(X,U)\}=\{0\}$ are the zero-mean distributions $\frac12\,\de_{-1}+\frac12\,\de_1$, $\frac23\,\de_{-1}+\frac13\,\de_2$, and $\de_0$, respectively.
Thus, the zero-mean distribution of $X$ is represented as a mixture of these two-point zero-mean distributions:
$$
\tfrac5{10}\,\de_{-1}+\tfrac1{10}\,\de_0+\tfrac3{10}\,\de_1+\tfrac1{10}\,\de_2
=
\tfrac6{10}\,(\tfrac12\,\de_{-1}+\tfrac12\,\de_1)
+\tfrac3{10}\,(\tfrac23\,\de_{-1}+\tfrac13\,\de_2)
+\tfrac1{10}\,\de_0.
$$ 
\end{example}

\subsection{Two-value zero-mean disintegration of several independent zero-mean r.v.'s and applications to self-normalized sums} \label{subsec:several}
Suppose here that 
$X_1,\dots,X_n$ are independent zero-mean r.v.'s
and 
$U_1,\dots,U_n$ are independent r.v.'s uniformly distributed on $[0,1]$, which are also independent of $X_1,\dots,X_n$.
For each $j=1,\dots,n$, let 
$R_j:=\r_j(X_j,U_j)$, where 
$\r_j$ denotes the reciprocating function for 
r.v.\ $X_j$. 
For any real $a_1,b_1,\dots,a_n,b_n$ such that $a_jb_j\le0$ for all $j$, let 
$$X_{1;a_1,b_1},\dots,X_{n;a_n,b_n}$$
be independent r.v.'s such that, for each $j\in\{1,\dots,n\}$, 
the r.v.\ 
$X_{j;a_j,b_j}$ is zero-mean 
and takes on its values in the two-point set $\{a_j,b_j\}$. For all $j$, let 
$$R_{j;a_j,b_j}:=a_j\,b_j/X_{j;a_j,b_j}$$
if $a_jb_j<0$ and $R_{j;a_j,b_j}:=0$ if $a_jb_j=0$. 

\begin{theorem} \label{th:F}
Let $g\colon\R^{2n}\to\R$ be any Borel function bounded from below (or from above). 
Then identity \eqref{eq:main} can be generalized as follows:
\begin{equation*}
\E g(X_1,R_1,\dots,X_n,R_n)
=\int_{(\R\times[0,1])^n}
\E g(X_{1;p_1},R_{1;p_1},\dots,X_{n;p_n},R_{n;p_n})\;\d p_1\cdots\d p_n,
\end{equation*}
where $p_j$ and $\d p_j$ stand, respectively, for $x_j,\r_j(x_j,u_j)$  
and $\P(X_j\in\d x_j)\,\d u_j$. 
Instead of the condition that $g$ be bounded from below or above, it is enough to require only that $g(x_1,r_1,\dots,x_n,r_n)-c_1x_1-\dots-c_nx_n$ be so for some real constants $c_1,\dots,c_n$ over all real $x_1,r_1,\dots,x_n,r_n$. 
\end{theorem}

For every natural $\al$, let 
$\H\al$ denote
the class 
of all functions $f\colon\R\to\R$ such that $f$ has finite derivatives $f^{(0)}:=f,f^{(1)}:=f',\dots,f^{(\al-1)}$ on $\R$, $f^{(\al-1)}$ is convex on $\R$, and $f^{(j)}(-\infty+)=0$ for $j=0,1,\dots,\al-1$. 

Applying Theorem~\ref{th:F} along with results of \cite{normal,asymm} to the mentioned asymmetry-corrected versions of self-normalized sums, one can obtain the following results.

\begin{corollary} \label{cor:student-normal}
Consider the self-normalized sum 
\begin{equation*} 
S_W:=\frac{X_1+\dots+X_n}{\frac12\sqrt{W_1^2+\dots+W_n^2}},
\end{equation*}
where $W_i:=|X_i-\r_i(X_i,U_i)|$; here, $\frac00:=0$. 
Then
\begin{align} 
\E f(S_W) &\le \E f(Z)\quad\forall f\in\H5\quad\text{and} \label{eq:stud-f(Z)} 
\\
\P(S_W\ge x)&\le c_{5,0}\P(Z\ge x)\quad\forall x\in\R, \label{eq:stud-P(Z>x)}
\end{align}
where $c_{5,0}=5!(e/5)^5=5.699\dots$ and, as before, $Z$ denotes a standard normal r.v. 
\end{corollary}  


\begin{corollary} \label{cor:stud-asymm}
Consider the self-normalized sum 
\begin{equation*} \label{eq:stud-y} 
S_{Y,\la}:=\frac{X_1+\dots+X_n}
{(Y_1^\la+\dots+Y_n^\la)^{\frac1{2\la}}},
\end{equation*}
where $Y_i:=|X_i\,\r_i(X_i,U_i)|$.
Suppose that for some $p\in(0,1)$ and all $i\in\{1,\dots,n\}$
\begin{equation}\label{eq:bounded-asymm}
\frac{X_i}{|\r_i(X_i,U_i)|}\ii{X_i>0} \le \frac{1-p}p\ \text{a.s.}	
\end{equation}
Then for all 
\begin{equation}\label{eq:la_*}
\la\ge \la_*(p):=
\begin{cases}
\dfrac{1 + p + 2\,p^2}
   {2{\big( {\sqrt{p - p^2}} + 
       2\,p^2 \big) }} \quad & \text{if}\quad 0<p\le\frac{1}{2}, \\
1 \quad & \text{if}\quad \frac{1}{2}\le p<1,     
\end{cases}
\end{equation}
one has
\begin{align*} 
\E f(V_{Y,\la}) &\le \E f(T_n)\quad\forall f\in\H3\quad\text{and} 
\\
\P(V_{Y,\la}\ge x)&\le c_{3,0}\P^\lc(T_n\ge x)\quad\forall x\in\R, 
\end{align*}
where $T_n:=(Z_1+\dots+Z_n)/n^{1/(2\la)}$; 
$Z_1,\dots,Z_n$ are independent r.v.'s each having the standardized Bernoulli distribution with parameter $p$; 
the function $x\mapsto\P^\lc(T_n\ge x)$ is the least log-concave majorant of the function $x\mapsto\P(T_n\ge x)$ on $\R$; $c_{3,0}=2e^3/9=4.4634\ldots$. 
The upper bound $c_{3,0}\P^\lc(T_n\ge x)$ can be replaced by somewhat better ones, in accordance with \cite[Theorem~2.3]{binom} or \cite[Corollary~4]{asymm}.
The lower bound $\la_*(p)$ on $\la$ given by \eqref{eq:la_*} is the best possible one, for each $p$.
\end{corollary} 

The bounded-asymmetry condition \eqref{eq:bounded-asymm} is likely to hold when the $X_i$'s are bounded i.i.d.\ r.v.'s.
For instance, \eqref{eq:bounded-asymm} holds with $p=\frac13$ for r.v.\ $X$ in Example~\ref{ex:discrete} in place of $X_i$.


\section{Statements of related results, with discussion} 
\label{discuss}
We begin this section with a number of propositions, collected in Subsections~\ref{r,x+-}. These propositions describe general properties of the reciprocating function $\r$ and the associated functions $x_+$ and $x_-$, and thus play a dual role. On the one hand, these properties of $\r$ and $x_\pm$ may  be of independent interest, each to its own extent. On the other hand, they will be used in the proofs of the basic Theorem~\ref{th:main} and related results to be stated and discussed in Subsections~\ref{vars}--\ref{model}. 

In Subsection~\ref{vars}, a generalization and various specializations of the mentioned two-point zero-mean disintegration are presented; methods of proofs are discussed and numerous relations of these results between themselves and with the mentioned result by Aizenman \emph{et al.} \cite{aizen} are also given.  
In Subsection~\ref{opt}, which exploits some of the results of Subsection~\ref{vars}, the disintegration based on the reciprocating function is shown to be optimal 
-- most symmetric, but also most inhomogeneous in the widths. 
In Subsection~\ref{charact}, various characterizations of the reciprocating function $\r$ (as well as of the functions $x_\pm$) are given. These characterizations are perhaps the most difficult results in this paper to obtain. They are then used in Subsection~\ref{model} for modeling. 

In all these results, the case when $X=0$ a.s.\ is trivial. So, henceforth let us assume by default that 
$\P(X=0)<1$.	
Also, unless specified otherwise, $\mu$ will stand for the distribution $\mu_X$ of $X$. 

\subsection{General properties of the functions $x_\pm$ and $\r$} \label{r,x+-}

Let us begin this subsection by stating, for easy reference, some elementary properties of the functions $x_\pm$ defined by \eqref{eq:x+} and \eqref{eq:x-}.

\begin{proposition}\label{lem:left-cont}
Take any $h\in[0,m]$ and $x\in[-\infty,\infty]$. Then 
\begin{align}
	x\ge x_+(h)\quad &\iff\quad x\ge0\ \&\ G(x)\ge h;\label{eq:L1}\\
	x\le x_-(h)\quad &\iff\quad x\le0\ \&\ G(x)\ge h.\label{eq:L1-}
\end{align}
It follows that 
\begin{align}
G(x)<h\text{ for all }x\in[0,x_+(h));\label{eq:less+}\\
G(x_+(h)-)\le h\le G(x_+(h));\label{eq:bet+}\\
G(x)<h\text{ for all }x\in(x_-(h),0];\label{eq:less-}\\
G(x_-(h)+)\le h\le G(x_-(h)).\label{eq:bet-}
\end{align}
Moreover, for any $h_1$, $h_2$, and $x$ one has the following implications: 
\begin{align}
& \big(0\le h_1<h_2\ \&\ x_+(h_1)=x_+(h_2)=x\big)\implies\big(\mu(\{x\})>0\ \&\ x>0\big);\label{eq:nonat+}\\
& \big(0\le h_1<h_2\ \&\ x_-(h_1)=x_-(h_2)=x\big)\implies\big(\mu(\{x\})>0\ \&\ x<0\big).\label{eq:nonat-}
\end{align}
Furthermore, the functions $x_+$ and $-x_-$ are 
\begin{enumerate}[(i)]
\item\label{x+-non-decr}
non-decreasing on $[0,m]$;
\item\label{x+-finite}
finite on $[0,m)$;
\item\label{x+-pos}
strictly positive on $(0,m]$;
\item\label{x+-left-cont}
left-continuous on $(0,m]$.
\end{enumerate}
\end{proposition}

Consider the lexicographic order $\prec$ on $[0,\infty]\times[0,1]$ defined by the formula
\begin{equation}\label{eq:lex_def}
 (x_1,u_1)\prec(x_2,u_2)\iff\big(x_1<x_2\text{ or }(x_1=x_2\ \&\ u_1<u_2)\big)
\end{equation}
for all $(x_1,u_1)$ and $(x_2,u_2)$ in $[0,\infty]\times[0,1]$. Extend this order symmetrically to $[-\infty,0]\times[0,1]$ by the formula 
\begin{equation*}
 (x_1,u_1)\prec(x_2,u_2)\iff(-x_1,u_1)\prec(-x_2,u_2) 
\end{equation*}
for all $(x_1,u_1)$ and $(x_2,u_2)$ in $[-\infty,0]\times[0,1]$. 

\begin{proposition}\label{prop:lex}
The function $\tG$ is $\prec$-nondecreasing on $[0,\infty]\times[0,1]$: if $(x_1,u_1)$ and $(x_2,u_2)$ are in $[0,\infty]\times[0,1]$ and $(x_1,u_1)\prec(x_2,u_2)$, then $\tG(x_1,u_1)\le\tG(x_2,u_2)$. 
Similarly, $\tG$ is $\prec$-nondecreasing on
$[-\infty,0]\times[0,1]$.
\end{proposition}

\begin{proposition}\label{lem:H}
For all $h\in[0,m]$ (recall definition \eqref{eq:m}), one has
\begin{align}
	H_+(h)&:=\E X\ii{X>0,\ \tG(X,U)\le h}=h,\label{eq:H+}\\
	H_-(h)&:=\E(-X)\ii{X<0,\ \tG(X,U)\le h}=h.\label{eq:H-}
\end{align}
\end{proposition}

The following proposition is a useful corollary of Proposition~\ref{lem:H}.

\begin{proposition}\label{prop:cont}
One has $\P\big(X\ne0,\ \tG(X,U)=h\big)=0$ for all real $h$. Therefore, $\P\big(\tG(X,U)=h\big)=0$ for all real $h\ne0$; that is, the distribution of the ``randomized'' version $\tG(X,U)$ of $G(X)$ may have an atom only at $0$.   
\end{proposition}


Along with the r.v.\ $X$, let 
$Y$, $Y_+$, $Y_-$ stand for any r.v.'s which are independent of $U$ and whose distributions are determined by the formulas
\begin{equation}\label{eq:P(Y)}
	\P(Y\in A)=\frac{\E|X|\ii{X\in A}}{\E|X|} \quad\text{and}\quad
		\P(Y_\pm\in A)=\frac{\E|X^\pm|\ii{X\in A}}{\E|X^\pm|}
\end{equation}
for all $A\in\B(\R)$; this is equivalent to 
\begin{equation}\label{eq:Ef(Y)}
	\E f(Y,U)=\frac1{2m}\E|X|\,f(X,U)  \quad\text{and}\quad
	\E f(Y_\pm,U)=\frac1m\E|X^\pm|\,f(X,U)
\end{equation}
for all Borel functions $f\colon\R^2\to\R$ bounded from below (or from above). 
Here and elsewhere, we use the standard notation $x^+:=\max(0,x)$ and $x^-:=\min(0,x)$. 
One should not confuse $Y_\pm$ with $Y^\pm$; in particular, by \eqref{eq:P(Y)}, 
$\P(Y^+=0)=\P(Y\le0)=\P(X\le0)\ne0$ (since $\E X=0$), while $\P(Y_+=0)=0$. 

Now one can state another corollary of Proposition~\ref{lem:H}: 

\begin{proposition}\label{prop:Y}
One has $\P(Y_+=0)=\P(Y_-=0)=\P(Y=0)=0$ and 
$\P\big(\tG(Y_+,U)\le h\big)=\P\big(\tG(Y_-,U)\le h\big)=\P\big(\tG(Y,U)\le h\big)=\frac hm$ for all $h\in[0,m]$. That is, the distribution of each of the three r.v's $\tG(Y_+,U)$, $\tG(Y_-,U)$, and $\tG(Y,U)$ is uniform on the interval $[0,m]$. 
\end{proposition}

At this point one is ready to admit that the very formulation of Theorem~\ref{th:main} may seem problematic for the following reasons. On the one hand, the two-value zero-mean r.v.'s $X_{a,b}$ are not defined (and cannot be reasonably defined) when one of the points $a$, $b$ is $\infty$ or $-\infty$ while the other one is nonzero. On the other hand, $\r(x,u)$ may take infinite values for some $u\in[0,1]$ and real nonzero $x$, which will make the r.v.\ $X_{x,\r(x,u)}$ undefined. For example, if $X$ has the zero-mean distribution (say $\mu_{\mathrm{Exp}}$) with density $e^{x-1}\ii{x<1}$, then $\r(x,u)=-\infty$ for all $(x,u)\in[1,\infty)\times[0,1]$; or, if $X$ has the distribution $\frac12\mu_{\mathrm{Exp}}+\frac14\de_{-1}+\frac14\de_1$, then $\r(x,u)=-\infty$ for $(x,u)\in\{(1,1)\}\cup\big((1,\infty)\times[0,1]\big)$. 

However, such concerns are taken care of by another corollary of Proposition~\ref{lem:H}: 

\begin{proposition}\label{prop:finite}
Almost surely, $|\r(X,U)|<\infty$. 
\end{proposition}

An application of Proposition~\ref{prop:cont} is the following refinement of Proposition~\ref{prop:finite}.  
Let, as usual, $\supp\nu$ denote the support of a given nonnegative measure $\nu$, which is defined as the set of all points $x\in\R$ such that for any open neighborhood $O$ of $x$ one has $\nu(O)>0$. 
Then, also as usual, 
$\supp X$ is defined as the support of the distribution $\mu_X$ of $X$. 

\begin{proposition}\label{prop:supp}
One has $\P\big(X\ne0,\ \r(X,U)\notin(\supp X)\setminus\{0\}\big)=0$; that is, almost surely on the event $X\ne0$, the values of the r.v.\ $\r(X,U)$ are nonzero and belong to $\supp X$. 
In particular, $\P\big(X\ne0,\ \r(X,U)=0\big)=0$. 
(Obviously, $\r(X,U)=0$ on the event $\{X=0\}$.)
\end{proposition}

In the sequel, the following definition will be quite helpful:
\begin{equation}\label{eq:hat}
	\hat x(x,u):=
	\begin{cases}
x_+(\tG(x,u)) & \text{ if }x\in[0,\infty], \\
x_-(\tG(x,u)) & \text{ if }x\in[-\infty,0] \\
\end{cases}
\end{equation}
for 
$u\in[0,1]$; 
cf.\ definition \eqref{eq:r} of the reciprocating function $\r$. 

\begin{proposition}\label{lem:hat}
Take any 
or $(x,u)\in[-\infty,\infty]\times[0,1]$ 
and let $h:=\tG(x,u)$ and, for brevity, $\hat x:=\hat x(x,u)$. Let $\mu$ stand for the distribution of $X$. Then
\begin{enumerate}[(i)]
	\item $0\le\hat x\le x$ if $x\ge0$;
		\item if $0\le\hat x<x$, then all of the following conditions must occur: 
		\begin{enumerate}
		\item
		$x_+(h+)>x_+(h)$; 
		\item
		$G(\hat x)=G(x-)=\tG(x,u)=h$; 
		\item
		$\mu\big((\hat x,x)\big)=0$; 
		\item
		$u=0$ or $\mu\big((\hat x,x]\big)=0$; 
		\item
		$u=0$ or $G(\hat x)=G(x)=h$; 
		\item
		$u=0$ or $x\ne x_+(h_1)$ for any $h_1\in[0,m]$; 
		\end{enumerate}
	\item $0\ge\hat x\ge x$ if $x\le0$;
\item if $0\ge\hat x>x$, then all of the following conditions must occur: 
		\begin{enumerate}
		\item
		$x_-(h+)<x_-(h)$; 
		\item
		$G(\hat x)=G(x+)=\tG(x,u)=h$; 
		\item
		$\mu\big((x,\hat x)\big)=0$; 
		\item
		$u=0$ or $\mu\big([x,\hat x)\big)=0$; 
		\item
		$u=0$ or $G(\hat x)=G(x)=h$; 
		\item
		$u=0$ or $x\ne x_-(h_1)$ for any $h_1\in[0,m]$; 
		\end{enumerate}
\item if $x=x_+(h_1)$ or $x=x_-(h_1)$ for some $h_1\in[0,m]$, then $\hat x(x,u)=x$ for all $u\in(0,1]$.
\end{enumerate}
\end{proposition}


From Proposition~\ref{lem:hat}, we shall deduce

\begin{proposition}\label{lem:regul}
Almost surely, $\hat x(X,U)=X$.
\end{proposition}

In view of Propositions~\ref{lem:regul} and \ref{lem:hat}, one may find it appropriate to 
refer to $\hat x(x,u)$ as the \emph{regularized} version of $x$, and to the function $\hat x$ as the  \emph{regularizing function} for (the distribution of) $X$.  

We shall use Proposition~\ref{lem:regul} to show that the mentioned in Introduction symmetry property $\r(-x)\equiv\r(x)$ of the reciprocating function for symmetric r.v.\ $X$ with a continuous strictly increasing d.f.\ essentially holds in general, without the latter two restrictions on the d.f.:

\begin{proposition}\label{prop:symm} 
The following conditions are equivalent to one another:
\begin{enumerate}[(i)]
	\item\label{symm1} $X$ is symmetric; 
	\item\label{symm2} $G$ is even; 
	\item\label{symm3} $x_-=-x_+$; 
	\item\label{symm4} $\r=-\hat x$; 
	\item\label{symm5} $\r(X,U)=-X$ a.s.  
\end{enumerate} 
\end{proposition}

Propositions~\ref{prop:cont} and \ref{lem:regul} can also be used to show that the term ``reciprocating function'' remains appropriate even when the d.f.\ of $X$ is not necessarily strictly increasing. Toward that end, let us first state 

\begin{proposition}\label{lem:hat=rr}
For any given $(x,u)\in\R\times[0,1]$, let 
\begin{equation*}
v:=\mathsf{v}(x,u):=
\begin{cases}
\frac{h-G(y+)}{G(y)-G(y+)} &\text{ if } x\ge0\ \&\ G(y)\ne G(y+),\\
\frac{h-G(y-)}{G(y)-G(y-)} &\text{ if } x\le0\ \&\ G(y)\ne G(y-),\\
1 &\text{otherwise,}
\end{cases}
\end{equation*}
where $h:=\tG(x,u)$ and $y:=\r(x,u)$; then  
\begin{equation}\label{eq:hat=rr}
	\r\big(\r(x,u),v\big)=\hat x(x,u).
\end{equation}
Moreover, the function $\mathsf{v}$ is Borel and takes its values in the interval $[0,1]$.  
\end{proposition}

Now one is ready for 

\begin{proposition}\label{prop:recip}
There exists a r.v.\ $V$ taking its values in $[0,1]$ (and possibly dependent on $(X,U)$) such that $\r\big(\r(X,U),V\big)=X$ a.s. 
In parti\-cular, for any continuous r.v.\ $X$ one has $\r(\r(X))=X$ a.s.\ (recall here part~\eqref{r for cont} of Remark~\ref{rem:cont}). 
\end{proposition}

\begin{remark*}\label{rem:symm}
In general, the identity $\r(x,u)=-x$ for a symmetric r.v.\ $X$ does not have to hold for all $x\in\R$ and $u\in[0,1]$, even if $X$ is continuous. For example, let $X$ be uniformly distributed on $[-1,1]$ and $x>1$; then $\r(x,u)=\r(x)=-1\ne-x$ for all $u$. 
Moreover, then $\r(\r(x))=1\ne x$, so that the identity $\r(\r(x))=x$ does not have to hold for all $x\in\R$, even if $X$ is continuous. 
Furthermore, if $X$ is not continuous and $V$ is not allowed to depend on $(X,U)$, then 
the conclusion $\r\big(\r(X,U),V\big)=X$ a.s.\ 
in Proposition~\ref{prop:recip} will not hold in general. 
For instance, in Example~\ref{ex:discrete} one has $\r\big(\r(1,u),v\big)=\r\big(\r(2,u),v\big)=
\ii{v\le\tfrac35}+2\ii{v>\tfrac35}$ for all $u$ and $v$ in $[0,1]$; 
so, for any r.v.\ $V$ taking its values in $[0,1]$ and independent of $(X,U)$, one has
$\P\big(\r\big(\r(X,U),V\big)\ne X\big)\ge\frac3{10}\P(V\le\frac35)+\frac1{10}\P(V>\frac35)\ge\frac1{10}>0$.
\end{remark*}

\subsection{Variations on the disintegration theme} \label{vars} 
In this subsection we shall consider a formal extension of Theorem~\ref{th:main}, stated as Proposition~\ref{prop:main}, which is in fact equivalent to Theorem~\ref{th:main}, and yet is more convenient in certain applications. 
A number of propositions which are corollaries to Theorem~\ref{th:main} or Proposition~\ref{prop:main} will be considered here, including certain identities for the joint distribution of $X$ and $\r(X,U)$. As noted before,  Theorem~\ref{th:main} implies a certain disintegration of the zero-mean distribution of $X$ into a mixture of two-point zero-mean distributions \big(recall \eqref{eq:Eg(X)}\big). We shall prove that such a disintegration can be obtained directly as well, and that proof is much simpler than the proof of Theorem~\ref{th:main}.  

Let us now proceed by noting first a special case of \eqref{eq:main}, with $g(x,r):=\ii{x=0,r\ne0}$ for all real $x$ and $r$. Then it follows  that $\r(X,U)\ne0$ almost surely on the event $\{X\ne0\}$: 
\begin{equation}\label{eq:Xne0,R=0}
\P\big(X\ne0,\ \r(X,U)=0\big)=0,	
\end{equation}
since $\P\big(X_{a,b}\ne0,\ R_{a,b}=0\big)=0$ for any $a$ and $b$ with $ab\le0$. \big(In fact, \eqref{eq:Xne0,R=0} is part of Proposition~\ref{prop:supp}, which will be proved in Subsection~\ref{proofs:props} -- of course, without relying on \eqref{eq:main} -- and then used in the proof Theorem~\ref{eq:main}.\big) Since $\r(x,u)=0$ if $x=0$, \eqref{eq:Xne0,R=0} can be rewritten in the symmetric form, as
\begin{equation}\label{eq:XR<0 or X=R=0}
\P\big(X\,\r(X,U)<0\text{ or }X=\r(X,U)=0\big)=1.	
\end{equation}

Next, note that the formalization of \eqref{eq:cond-pair} given in Theorem~\ref{th:main} differs somewhat from the way in which the notion of the conditional distribution is usually understood. 
Yet, Theorem~\ref{th:main} and its extension, Theorem~\ref{th:F}, are quite convenient in the applications, such as Corollaries~\ref{cor:student-normal} and \ref{cor:stud-asymm}, and others. 
However, Theorem~\ref{th:main} can be presented in a more general form -- as a statement on the joint distribution of the ordered pair $\big(X,\r(X,U)\big)$ and the (unordered) set $\{X,\r(X,U)\}$, which may appear to be in better accordance with informal statement \eqref{eq:cond-pair}: 
\begin{proposition}\label{prop:main}
Let $g\colon\R^2\times\R^2\to\R$ be any Borel function bounded from below (or from above), which is symmetric in the pair $(\tx,\tr)$ of its last two arguments: 
\begin{equation}\label{eq:g-symm}
g(x,r;\tr,\tx)=g(x,r;\tx,\tr)
\end{equation}
for all real $x,r,\tx,\tr$. Then
\begin{equation*}
\begin{split}
	\E g\big(X,\r(X,U);X,\r(X,U) & \big)\\
	=\int_{\R\times[0,1]} 
	&\E g\big(X_{x,\r(x,u)},R_{x,\r(x,u)};x,\r(x,u)\big)\,\P(X\in \d x)\,\d u.
	\end{split}
\end{equation*}
Instead of the condition that $g$ be bounded from below or above, it is enough to require only that $g(x,r;\tx,\tr)-cx-\tilde c\tr$ be so for some real constants $c$, $\tilde c$ -- over all real $x,r,\tx,\tr$. 
\end{proposition} 

Symmetry restriction \eqref{eq:g-symm} imposed on the functions $g$ in Proposition~\ref{prop:main} corresponds to the fact that the conditioning in \eqref{eq:cond} and \eqref{eq:cond-pair} 
is on the (unordered) set $\{X,\r(X,U)\}$, and of course not on the ordered pair $\big(X,\r(X,U)\big)$. Indeed, the
natural conditions $\psi(a,b)=\psi(b,a)=\tilde\psi(\{a,b\})$ (for all real $a$ and $b$) establish 
a one-to-one correspondence between the symmetric functions $(a,b)\mapsto\psi(a,b)$ of the ordered pairs $(a,b)$ and the functions $\{a,b\}\mapsto\tilde\psi(\{a,b\})$ of the sets $\{a,b\}$. This correspondence can be used to \emph{define} the Borel $\sigma$-algebra on the set of all sets of the form $\{a,b\}$ with real $a$ and $b$ as the $\sigma$-algebra generated by all symmetric Borel functions on $\R^2$. It is then with respect to this $\sigma$-algebra that the conditioning in the informal equation \eqref{eq:cond-pair} should be understood.  

Even if more cumbersome than Theorem~\ref{th:main}, Proposition~\ref{prop:main} will sometimes be more convenient to use. 
We shall prove Proposition~\ref{prop:main} (later in Section~\ref{sec:proofs}) and then simply note that Theorem~\ref{th:main} is a special case of Proposition~\ref{prop:main}. 

Alternatively, one could first prove Theorem~\ref{th:main} -- in a virtually the same way as Proposition~\ref{prop:main} is proved in this paper
\big(one only would have to use $g(a,b)$ instead of $g(a,b;a,b)[=g(a,b;b,a)]$\big), and then it would be easy to 
deduce the ostensibly more general Proposition~\ref{prop:main} 
from Theorem~\ref{th:main}, in view of \eqref{eq:XR<0 or X=R=0}.  Indeed, for any function $g$ as in Proposition~\ref{prop:main}, one can observe that $\E g\big(X,\r(X,U);X,\r(X,U)\big)=\E\tilde g\big(X,\r(X,U)\big)$ and 
$\E g\big(X_{a,b},R_{a,b};a,b\big)=\E\tilde g\big(X_{a,b},R_{a,b}\big)$ for all real $a$ and $b$ such that either $ab<0$ or $a=b=0$, 
where 
$\tilde g(a,b):=g(a,b;a,b)$. 

The following proposition, convenient in some applications, is a corollary of Proposition~\ref{prop:main}.  

\begin{proposition}\label{prop:g1g2}
Let $g:=g_1-g_2$, where 
$g_i\colon\R^2\times\R^2\to\R$ ($i=1,2$) are any Borel functions bounded from below (or from above), symmetric in their last two arguments. Suppose that
\begin{equation}\label{eq:g}
 \begin{gathered}
  g(0,0;0,0)=0;  \\ 
g(x,r;x,r)\,r=g(r,x;r,x)\,x \text{ for all real $x$ and $r$ with $xr<0$}.
 \end{gathered}
\end{equation}
Then $\E g_1\big(X,\r(X,U);X,\r(X,U)\big)=\E g_2\big(X,\r(X,U);X,\r(X,U)\big)$. 
\end{proposition}

Proposition~\ref{prop:g1g2} allows one to easily obtain identities for the distribution of the ordered pair $\big(X,\r(X,U)\big)$ or, more generally, for the conditional distribution of $\big(X,\r(X,U)\big)$ given the random set $\{X,\r(X,U)\}$. 

For instance, letting $g(x,r;\tx,\tr):=x\,\psi(\tx,\tr)$, one obtains 
the following proposition, which states that the conditional expectation of $X$ given the random set $\{X,\r(X,U)\}$ is zero:
$$\E\big(X\,|\,\{X,\r(X,U)\}\big)=0.$$
More formally, one has
\begin{proposition}\label{prop:cond-mean0}
Suppose that $\psi\colon\R^2\to\R$ is a symmetric Borel function, so that $\psi(x,r)=\psi(r,x)$ for all real $x$ and $r$. Suppose also that the function $(x,r)\mapsto x\,\psi(x,r)$ is bounded on $\R^2$. Then
\begin{equation*}
	\E X\psi\big(X,\r(X,U)\big)=0.
\end{equation*}
\end{proposition}

While Proposition~\ref{prop:cond-mean0} is a special case of Proposition~\ref{prop:g1g2} and hence of Proposition~\ref{prop:main}, the general case presented in Proposition~\ref{prop:main} will be shown to follow rather easily from this special case; essentially, this easiness is due to the fact that a distribution on a given two-point set is uniquely determined if the mean of the distribution is known -- to be zero, say, or to be any other given value.


Looking back at \eqref{eq:Xne0,R=0}, one can see that the ratio $\frac X{\r(X,U)}$ can be conventionally defined almost surely on the event $\{X\ne0\}$; let also $\frac X{\r(X,U)}:=-1$ on the event $\{X=0\}$. 
Letting then $g(x,r;\tr,\tx):=\big(\psi(r,x)+\psi(x,r)\,\frac xr\big)\ii{xr<0}\,\vpi(\tx,\tr)$ for all real $x,r,\tx,\tr$, 
where $\psi$ is any nonnegative Borel function and $\vpi$ is any symmetric nonnegative Borel function, 
one obtains from Proposition~\ref{prop:g1g2} the identity
\begin{equation}\label{eq:1+x/r}
	\E\psi\big(X,\r(X,U)\big)\,\frac X{\r(X,U)}\,
\vpi\big(X,\r(X,U)\big)=
-\E\psi\big(\r(X,U),X\big)\,
\vpi\big(X,\r(X,U)\big).
\end{equation}
In particular, letting here $\psi=1$, one sees
that the conditional expectation of $\dfrac X{\r(X,U)}$ given the two-point set $\{X,\r(X,U)\}$ is $-1$:
\begin{equation*}
		\E\Big(\frac X{\r(X,U)}\Big|\,\{X,\r(X,U)\}\Big)
=-1.
\end{equation*}
It further follows that   
\begin{equation}\label{eq:Ex/r}
	\E\frac X{\r(X,U)}=-1.
\end{equation}

On the other hand, letting $\frac{\r(X,U)}X:=-1$ on the event $\{X=0\}$, one has 

\begin{proposition}\label{prop:Er/x}
   If $X$ is symmetric, then $\E\frac{\r(X,U)}X=-1$; otherwise,   
\begin{equation}\label{eq:Er/x}
 \E\frac{\r(X,U)}X<-1. 
\end{equation}
\end{proposition}

The contrast between \eqref{eq:Ex/r} and \eqref{eq:Er/x} may appear surprising, as an ostensible absence of interchangeability between $X$ and $\r(X,U)$. However, this does not mean that the construction of the reciprocating function is deficient in any sense. In fact, as mentioned before, the disintegration based on $\r$ will be shown to be optimal in a variety of senses. 
Also, 
such ``non-interchangeability'' of $X$ and $\r(X,U)$ manifests itself even in the case of a ``pure'' two-point zero-mean distribution:  
\begin{equation}\label{eq:ER/X}
	\E\frac{X_{a,b}}{R_{a,b}}=-1,\quad
	\E\frac{R_{a,b}}{X_{a,b}}=-1+\frac{(a+b)^2}{ab}
\end{equation}
for all $a$ and $b$ with $ab<0$; recall \eqref{eq:r_ab}. 

The ``strange'' inequality $\E\frac X{\r(X,U)}\ne\E\frac{\r(X,U)}X$ (unless $X$ is symmetric) is caused only by the use of an inappropriate averaging measure -- which is the distribution of r.v.\ $X$, just one r.v.\ of the pair $\big(X,\r(X,U)\big)$ -- and this choice of one r.v.\ over the other breaks the symmetry. Here is how this concern is properly addressed:

\begin{proposition}\label{prop:Y,r(Y)}
For r.v.'s $Y$ and $Y_\pm$ described in the paragraph containing  \eqref{eq:P(Y)}, 
\begin{gather}
	\big(Y,\r(Y,U)\big)\D\big(\r(Y,U),Y\big);\label{eq:Yinterchange}\\
\big(Y_+,\r(Y_+,U)\big)\D\big(\r(Y_-,U),Y_-\big)
\D\big(x_+(H),x_-(H)\big);\label{eq:Y+-interchange} \\
\big\{Y,\r(Y,U)\big\}\D\big\{Y_+,\r(Y_+,U)\big\}
\D\big\{Y_-,\r(Y_-,U)\big\}
\D\big\{x_+(H),x_-(H)\big\},\label{eq:Ysets}
\end{gather} 
where $H$ is any r.v.\ uniformly distributed on $[0,m]$. 
In particular, $\r(Y,U)\D Y$, $\r(Y_+,U)\D Y_-\D x_-(H)$, $\r(Y_-,U)\D Y_+\D x_+(H)$, $\dfrac{\r(Y,U)}Y\D\dfrac Y{\r(Y,U)}$, and 
$\E\dfrac{\r(Y,U)}Y=\E\dfrac Y{\r(Y,U)}\big(<-1$ except when $X$ is symmetric, in which case one has ``$=-1$'' in place of ``$<-1$''\big); recall that $Y$, $Y_+$, and $Y_-$ are almost surely nonzero, by Proposition~\ref{prop:Y}.
\end{proposition}

Just as in Proposition~\ref{prop:main} versus \eqref{eq:cond-pair}, the equalities in distribution of the random two-point sets in \eqref{eq:Ysets} are understood as the equalities of the expected values of (say all nonnegative Borel) \emph{symmetric} functions of the corresponding ordered pairs of r.v.'s. 

Proposition~\ref{prop:Y,r(Y)} and, especially, relations \eqref{eq:Y+-interchange} suggest an alternative way to construct the reciprocating function $\r$. Namely, one could start with an arbitrary r.v.\ $H$ uniformly distributed in $[0,m]$ and then let $Y_\pm:=x_\pm(H)$. Then, by a disintegration theorem for the joint distribution of two r.v.'s (see e.g.\ \cite[Proposition~B.1]{gangbo}), there exist measurable functions $\r_\pm$ such that $\big(Y_+,\r_-(Y_+,U)\big)\D(Y_+,Y_-)
\D\big(\r_+(Y_-,U),Y_-\big)$; cf.\ \eqref{eq:Y+-interchange}. Finally, one would let $\r(y,u):=\r_\pm(y,u)$ if $\mp y>0$. 
However, this approach appears less constructive than the one represented by \eqref{eq:r} and thus will not be pursued here. 

Going back to \eqref{eq:1+x/r} and letting there $\vpi=1$ and $\psi(x,r)\equiv\ii{(x,r)\in A}$ for an arbitrary $A\in\B(\R^2)$, one has 
\begin{equation*}
 \mu_{(R,X)}(A)=\int_A\frac{-x}r\;\mu_{(X,R)}(\d x\times\d r),
\end{equation*}
where $R:=\r(X,U)$ and $\mu_Z$ denotes the distribution of a random point $Z$, with the rule $\frac{-0}0:=1$.  
This means that the distribution of the random point $\big(\r(X,U),X\big)$ is absolutely continuous relative to that of $\big(X,\r(X,U)\big)$, with the function  $(x,r)\mapsto\ii{x=r=0}+\frac{-x}r\,\ii{xr<0}$ as a Radon-Nikodym derivative.  

Specializing further, with $A$ of the form $B\times\R$ for some $B\in\B(\R)$, one has 
\begin{align*}
 \P\big(\r(X,U)\in B\big)
&=\E\ii{X\in B}\,\frac{-X}{\r(X,U)} \\
&=\int_{B\times[0,1]}\frac{-x}{\r(x,u)}\,\P(X\in\d x)\,\d u
=\int_B\P(X\in\d x)\,\int_0^1\frac{-x}{\r(x,u)}\,\d u,
\end{align*}
so that the distribution of $\r(X,U)$ is absolutely continuous relative to that of $X$, with the function $x\mapsto\int_0^1\frac{-x}{r(x,u)}\,\d u$ as a Radon-Nikodym derivative.  

Recall now the special case \eqref{eq:Eg(X)} of \eqref{eq:main}. In particular, identity \eqref{eq:Eg(X)} implies that an arbitrary zero-mean distribution can be represented as the mixture of two-point zero-mean distributions. 
However, such a mixture representation by itself is much easier to prove (and even to state) than Theorem~\ref{th:main}. For instance, one has 

\begin{proposition}\label{prop:mix}
Let $g\colon\R\to\R$ be any Borel function bounded from below (or from above) such that $g(0)=0$. Then
\begin{equation}\label{eq:Eg(X) other}
	\E g(X)=\int_0^m 
	\E g(X_h)\,\frac{\d h}{\E X_h^{\,+}},
\end{equation}
where $X_h:=X_{x_+(h),x_-(h)}$. 
\end{proposition}

We shall give a very short and simple proof of Proposition~\ref{prop:mix} (see Proof~1 on page~\pageref{eq:g_a}), which relies only on such elementary properties of the functions $x_+$ and $x_-$ as \eqref{eq:L1} and \eqref{x+-pos} of Proposition~\ref{lem:left-cont}. We shall also give an alternative proof of Proposition~\ref{prop:mix}, based on \cite[Theorem~2.2]{aizen} as well on some properties of the functions $x_+$ and $x_-$ provided by Propositions~\ref{lem:hat} and \ref{lem:left-cont} of this paper. The direct proof is a bit shorter and, in our view, simpler. 

This simplicity of the proof might be explained by the observation that -- while Proposition~\ref{prop:mix} \big(or, for that matter, identity \eqref{eq:Eg(X)}\big) describes the one-dimensional distribution of $X$ (as a certain mixture) -- Theorem~\ref{th:main} provides a mixture representation of the  two-dimensional distribution of the pair $\big(X,\r(X,U)\big)$, even though the distribution of this pair is completely determined by the distribution of $X$. Note that the random pair $\big(X,\r(X,U)\big)$ is expressed in terms of the reciprocating function, which in turn depends, in a nonlinear and rather complicated manner, on the distribution of $X$. 
Another indication of the simplicity of identity \eqref{eq:Eg(X) other} is that it \big(in contrast with \eqref{eq:main} and even with \eqref{eq:Eg(X)}\big) does not contain the randomizing random variable $U$. 
On the other hand, an obvious advantage of disintegration \eqref{eq:main} is that it admits such applications to self-normalized sums as Corollaries~\ref{cor:student-normal}  and \ref{cor:stud-asymm}.

However, there are a number of ways to rewrite \eqref{eq:Eg(X) other} in terms similar to those of \eqref{eq:Eg(X)}. Towards that end, for each function $g$ as in Proposition~\ref{prop:mix}, introduce the function $\Psi_g$ defined by the formula 
\begin{equation}\label{eq:Psi}
\Psi_g(h):=\frac{\E g(X_h)}{\E(X_h)^+}\quad\text{for all $h\in(0,m)$. }
\end{equation}
Then \eqref{eq:Eg(X) other} can be rewritten as 
\begin{equation}\label{eq:EPsi}
\E g(X)=m\E\Psi_g(H),
\end{equation}
where $H$ is any r.v.\ uniformly distributed on the interval $[0,m]$. One such r.v.\ is $m\tF(X,U)$, where $\tF(x,u):=F(x-)+u\cdot\big(F(x)-F(x-)\big)$ and $F$ is the d.f.\ of $X$. This follows in view of 
 
\begin{proposition}\label{prop:tF}
The r.v.\ $\tF(X,U)$ is uniformly distributed on the interval $[0,1]$; cf.\ Proposition~\ref{prop:Y}. 
\end{proposition}

Hence, for all $g$ as in Proposition~\ref{prop:mix}, one has an identity similar in form to \eqref{eq:Eg(X)}:
\begin{align*}
		\E g(X)&=m\int_{\R\times[0,1]}\Psi_g\big(m\tF(x,u)\big)\,\P(X\in \d x)\,\d u \\
		&=
		m\int_{\R\times[0,1]} 
	\E g\big(X_{a_+(x,u),a_-(x,u)}\big)
			\frac{	\,\P(X\in \d x)\,\d u}
	{
	\E(X_{a_+(x,u),a_-(x,u)})^+
	},
\end{align*}
where $a_+(x,u)$ and $a_-(x,u)$ stand for $x_+\big(m\tF(x,u)\big)$ and $x_-\big(m\tF(x,u)\big)$, respectively.

However, more interesting mixture representations are obtained if one uses Proposition~\ref{prop:Y} (and also Proposition~\ref{lem:regul}) instead of Proposition~\ref{prop:tF}:

\begin{proposition}\label{prop:Eg(X) x/r}
Let $g\colon\R\to\R$ is any Borel function bounded from below (or from above). 
Then, assuming the rule $\frac{0}{\r(0,u)}:=-1$ for all $u\in[0,1]$, one has 
\begin{align}
	\E g(X)&=
	\int_{\R\times[0,1]} 
	\E g\big(X_{x,\r(x,u)}\big)\,
	\P(X\in \d x)\,\d u; \label{eq:Eg(X) again}\\
	\E g(X)&=
	\int_{\R\times[0,1]} 
	\E g\big(X_{x,\r(x,u)}\big)\,
	\frac{-x}{\r(x,u)}\,
	\P(X\in \d x)\,\d u; \label{eq:Eg(X) -x/r}\\
	\E g(X)&=\frac12\,
	\int_{\R\times[0,1]} 
	\E g\big(X_{x,\r(x,u)}\big)\,
	\Big(1-\frac x{\r(x,u)}\Big)\,
	\P(X\in \d x)\,\d u. \label{eq:Eg(X) 1-x/r}
\end{align} 
\end{proposition}

Going back to 
\eqref{eq:1+x/r} and letting therein $\psi(x,r)\equiv\E g(X_{x,r})$ and $\vpi=1$, one can rewrite the right-hand side of identity \eqref{eq:Eg(X) -x/r} as 
$\E\psi(X,R)\frac{-X}R=\E\psi(R,X)
=\int_{\R^2}\psi(r,x)\,\mu_{(X,R)}(\d x,\d r) 
=\int_{\R^2}\E g(X_{r,x})\,\mu_{(X,R)}(\d x,\d r)$,
so that \eqref{eq:Eg(X) -x/r} can be rewritten as
\begin{equation*}
	\E g(X)=\int_{\R^2}\E g(X_{r,x})\,\mu_{(X,R)}(\d x,\d r);
\end{equation*}
here, as before, $R:=\r(X,U)$. 
Similarly \big(but in a simpler say, without using \eqref{eq:1+x/r}\big), identity \eqref{eq:Eg(X) again} can be rewritten as
\begin{equation*}
	\E g(X)=\int_{\R^2}\E g(X_{x,r})\,\mu_{(X,R)}(\d x,\d r).
\end{equation*}
Now it is immediately clear why the right-hand sides of \eqref{eq:Eg(X) again} and \eqref{eq:Eg(X) -x/r} are identical to each other: because $X_{x,r}\D X_{r,x}$. This is another way to derive \eqref{eq:Eg(X) -x/r}: from \eqref{eq:Eg(X) again} and \eqref{eq:1+x/r}. 
Of course, identity \eqref{eq:Eg(X) again} is the same as \eqref{eq:Eg(X)}, which was obtained as a special case of \eqref{eq:main}. Here, the point is that identity \eqref{eq:Eg(X)} can be alternatively deduced from the simple -- to state and to prove -- identity \eqref{eq:Eg(X) other}. 

However, no simple way is seen to deduce \eqref{eq:main} from \eqref{eq:Eg(X) other}. Toward such an end, one might start with the obvious identity $\E g\big(X,\r(X,U)\big)=\E g_1(X)$, where $g_1(x):=\int_0^1 g\big(x,\r(x,v)\big)\,\d v$. Then one might try to use \eqref{eq:Eg(X) again} with $g_1$ in place of $g$, which yields 
\begin{align*}
\E g\big(X,\r(X,U))
	&=\int_{\R\times[0,1]^2} 
	\E g\big(X_{x,\r(x,u)},\r\big(X_{x,\r(x,u)},v\big)\big)\,\P(X\in \d x)\,\d u\,\d v.
\end{align*}
At that, $\E g\big(X_{x,\r(x,u)},\r\big(X_{x,\r(x,u)},v\big)\big)
=\frac
	{g\big(x,\r(x,v)\big)\,\r(x,u)
	-g\big(\r(x,u),\r\big(\r(x,u),v\big)\big)\,x}
	{\r(x,u)-x}$. 
From this, one would be able to get \eqref{eq:main} if one could replace here the terms $\r(x,v)$ and $\r\big(\r(x,u),v\big)$ by $\r(x,u)$ and $x$, respectively, and it is not clear how this could be easily done, unless the distribution of $X$ is non-atomic (cf.\ Propositions~\ref{lem:hat=rr} and \ref{prop:recip}). Anyway, such an alternative proof would hardly be simpler than the proof of disintegration \eqref{eq:main} given in this paper. 


\subsection{Optimality properties of the two-point disintegration}\label{opt} 
Two-value zero-mean disintegration is not unique. 
For example, consider the symmetric distribution 
$\frac1{10}\de_{-2}+\frac4{10}\de_{-1}+\frac4{10}\de_{1}+\frac1{10}\de_{2}$ (cf. Example~\ref{ex:discrete}). This distribution can be represented either as the mixture 
$\frac3{10}(\frac13\de_{-2}+\frac23\de_{1})+
\frac3{10}(\frac13\de_{2}+\frac23\de_{-1})+
\frac4{10}(\frac12\de_{-1}+\frac12\de_{1})$ of two asymmetric and one symmetric two-point zero-mean distributions or as the mixture
$\frac15(\frac12\de_{-2}+\frac12\de_{2})+
\frac45(\frac12\de_{-1}+\frac12\de_{1})$ of two symmetric two-point zero-mean distributions; 
the latter representation is a special case of \eqref{eq:Eg(X)} or, equivalently, \eqref{eq:Eg(X) other}. 

We shall show that, in a variety of senses (indexed by the continuous superadditive functions as described below), representation \eqref{eq:Eg(X)} of an arbitrary zero-mean distribution as the mixture of two-point zero-mean distributions is on an average most symmetric. 
The proof of this optimality property is based on the stated below variants of a well-known theorem on optimal transportation of mass, which are most convenient for our purposes; cf.\ e.g.\  \cite{hoeff40} (translated in \cite[pp.\ 57--107] {hoeff-coll}), \cite{camb}, \cite{tchen}, \cite{rach}. 
We need to introduce some definitions. 

Let $I_1$ and $I_2$ be intervals on the real line. 
A function $k\colon I_1\times I_2\to\R$ is called \emph{superadditive} if  
\begin{equation*}
	k(a,c)+k(b,d)\ge k(a,d)+k(b,c)
\end{equation*}
for all $a,b$ in $I_1$ and $c,d$ in $I_2$ such that $a<b$ and $c<d$. 
So, superadditive functions are like the distribution functions on $\R^2$. 
For a function $k\colon I_1\times I_2\to\R$ to be superadditive, it is enough that it be continuous on $I_1\times I_2$ and twice continuously differentiable in the interior of $I_1\times I_2$ with a nonnegative second mixed partial derivative. 

Let $X_1$ and $X_2$ be any r.v.'s with values in the intervals $I_1$ and $I_2$, respectively. 
Let 
\begin{equation}\label{eq:tX}
\tX_1:=\tx_1(H)\quad\text{and}\quad\tX_2:=\tx_2(H),	
\end{equation}
where $H$ is any non-atomic r.v., and $\tx_1\colon\R\to I_1$ and $\tx_2\colon\R\to I_2$ are any nondecreasing left-continuous functions such that 
\begin{equation}\label{eq:tX=X}
\tX_1\D X_1 \quad\text{and}\quad \tX_2\D X_2. 	
\end{equation}  

\begin{proposition}\label{prop:trans1}
Let each of the intervals $I_1$ and $I_2$ be of the form $[a,b)$, where $-\infty<a<b\le\infty$. Suppose that a function $k$ is superadditive, right-continuous, and bounded from below on $I_1\times I_2$. Then
\begin{equation}\label{eq:trans}
	\E k(X_1,X_2)\le\E k(\tX_1,\tX_2).
\end{equation}
\end{proposition}

\begin{proposition}\label{prop:trans2}
Suppose that a function $k$ is superadditive, continuous, and bounded from above on $(0,\infty)^2$. 
Suppose that $X_1>0$ and $X_2>0$ a.s. 
Then \eqref{eq:trans} holds. 
\end{proposition}



\vspace*{-1.5cm}

\parbox{.53in}
{
\hspace*{-.4cm}
\pspicture(0,0)(6,3.6)
\psset{xunit=1.85mm}
\psset{yunit=5mm}
\psset{linewidth=.2mm}
\psline[linewidth=.3mm](0,0)(8,0)
\uput[270](4,-.0){$I_1$}
\psline[linewidth=.3mm](1,3)(6,3)
\uput[90](3.5,3.){$I_2$}
\psline[linestyle=dashed,dash=.1 .1]{->}(0,0)(1,3)
\psline[linestyle=dashed,dash=.1 .1]{->}(2,0)(1.625,3)
\psline[linestyle=dashed,dash=.1 .1]{->}(4,0)(3.,3)
\psline[linestyle=dashed,dash=.1 .1]{->}(6,0)(5.5,3)
\psline[linestyle=dashed,dash=.1 .1]{->}(8,0)(6,3)
\endpspicture
}
\parbox[t]{4.25in}
{Propositions~\ref{prop:trans1} and \ref{prop:trans2} essentially mean that, if the unit-trans\-portation cost function $k$ is superadditive, then a costliest plan of transportation of mass distribution $\mu_{X_1}$ on interval $I_1$ to mass distribution $\mu_{X_2}$ on $I_2$ is such that no two arrows in the picture here on the left may cross over; that is, smaller (respectively, larger) values in $I_1$ are matched with appropriate smaller (respectively, larger) values in $I_2$.}

\vspace*{3pt}

Note that no integrability conditions are required in Proposition~\ref{prop:trans1} or \ref{prop:trans2} except for the boundedness of $k$ from below or above; at that, either or both sides of inequality \eqref{eq:trans} may be infinite. 

Proposition~\ref{prop:trans2} is essentially borrowed from \cite[Corollary~2.2.(a)]{tchen}. 


\begin{proposition}\label{prop:best}
Suppose that one has a two-point zero-mean mixture representation of the distribution of a zero-mean r.v.\ $X$:  
\begin{equation}\label{eq:Eg(X) nu}
	\E g(X)=
	\int_S \E g(X_{y_+(s),y_-(s)})\,\nu(\d s)
\end{equation}
for all Borel functions $g\colon\R\to\R$ bounded from below or from above, 
where $\nu$ is a probability measure on a measurable space $(S,\Sigma)$, and $y_+\colon S\to(0,\infty)$ and $y_-\colon S\to(-\infty,0)$ are $\Sigma$-measurable functions. 
Then
\begin{enumerate}[(i)]
	\item\label{tilde nu} 
	equation 
	\begin{equation}\label{eq:tilde nu}
	\tilde\nu(\d s):=\frac{\E X_{y_+(s),y_-(s)}^{\ +}}m\;\nu(\d s)
\end{equation}
defines a probability measure $\tilde\nu$ on $(S,\Sigma)$, so that the functions $y_+$ and $y_-$ can (and will be) considered as r.v.'s on the probability space $(S,\Sigma,\tilde\nu)$;  
	\item\label{y+-} 
then, $y_+\D Y_+$ and $y_-\D Y_-$, where $Y_\pm$ are r.v.'s as in \eqref{eq:P(Y)};
	\item\label{best}
	let $H$ be any r.v.\ uniformly distributed on $[0,m]$; suppose also that a superaddtive function $k$ is either as in Proposition~\ref{prop:trans1} \big(with $I_1=I_2=[0,\infty)$\big) or as in Proposition~\ref{prop:trans2}; 
then  
	\begin{equation}\label{eq:best-k}
	\begin{split}
	\E k( & y_+, -y_-)\le
\E k\big(x_+(H),-x_-(H)\big)\\
	&=\E k\big(Y_+,-\r(Y_+,U)\big)
	=\E k\big(\r(Y_-,U),-Y_-\big)
	\overset{\mathrm{(symm)}}=
	\E k\big(\r(Y,U),-Y\big), 
	\end{split}
\end{equation}
where the symbol ``$\overset{\mathrm{(symm)}}=$'' means an equality which takes place in the case when the additional symmetry condition $k(x,-r)=k(r,-x)$ holds for all real $x$ and $r$ such that $xr<0$; 
in particular, for any $p>0$ and $\|Z\|_p:=(\E|Z|^p)^{1/p}$,
	\begin{gather}
	\label{eq:best-p-ratios}
	\Big\|\frac{y_\pm}{y_\mp}\Big\|_p\ge
	\Big\|\frac{x_\pm(H)}{x_\mp(H)}\Big\|_p 
	=\Big\|\frac{Y_\pm}{\r(Y_\pm,U)}\Big\|_p
	=\Big\|\frac{\r(Y_\mp,U)}{Y_\mp}\Big\|_p,  
\\
	\begin{split}\label{eq:best-p-ratio}	\Big\|\frac{y_+}{y_-}\Big\|_p^p+\Big\|\frac{y_-}{y_+}\Big\|_p^p&\ge	\Big\|\frac{x_+(H)}{x_-(H)}\Big\|_p^p+\Big\|\frac{x_-(H)}{x_+(H)}\Big\|_p^p \\	&=\Big\|\frac{Y_+}{\r(Y_+,U)}\Big\|_p^p+\Big\|\frac{\r(Y_+,U)}{Y_+}\Big\|_p^p	=\Big\|\frac{\r(Y_-,U)}{Y_-}\Big\|_p^p+\Big\|\frac{Y_-}{\r(Y_-,U)}\Big\|_p^p \\
&=2\Big\|\frac{\r(Y,U)}{Y}\Big\|_p^p=2\Big\|\frac{Y}{\r(Y,U)}\Big\|_p^p; 
	\end{split} 
	\end{gather}
for any $p\ge1$, 
	\begin{gather}
	\label{eq:best-p-symm}
	\begin{split}
	\|y_+ + y_-\|_p&\ge
	 \big\|x_+(H)+x_-(H)\big\|_p \\
	&=\big\|Y_+ +\r(Y_+,U)\big\|_p
	=\big\|\r(Y_-,U)+Y_-\big\|^p=\big\|\r(Y,U)+Y\big\|_p,  
	\end{split} \\
	\label{eq:best-p-width}
	\begin{split}
	\|y_+ - y_-\|_p&\le
	 \big\|x_+(H)-x_-(H)\big\|_p\\
	&=\big\|Y_+ -\r(Y_+,U)\big\|_p
	=\big\|\r(Y_-,U)-Y_-\big\|_p=\big\|\r(Y,U)-Y\big\|_p;  
	\end{split} \\
	\intertext{for any $p\le0$,}
	\label{eq:best-p-width,neg}
	\begin{split}
	\E(y_+ - y_-)^p&\le
	 \E\big(x_+(H)-x_-(H)\big)^p\\
	&=\E\big(Y_+ -\r(Y_+,U)\big)^p
	=\E\big(\r(Y_-,U)-Y_-\big)^p=\E\big|\r(Y,U)-Y\big|^p. 
	\end{split}
\end{gather}	
\end{enumerate}
\end{proposition}

\begin{remark*} Observe that 
the probability measure $\tilde\nu$ defined by \eqref{eq:tilde nu} -- which, according to part (ii) of Proposition~\ref{prop:best}, equalizes $y_\pm$ with $Y_\pm$ in distribution -- is quite natural, as one considers 
the problem of the most symmetric disintegration of an arbitrary zero-mean distribution into the mixture of two-point zero-mean distributions as 
the problem of the most symmetric transportation (or, in other words, matching) of the measure $A\mapsto\E X^+\ii{X\in A}$ to the measure $A\mapsto\E(-X^-)\ii{X\in A}$ (of the same total mass) or, equivalently, the most symmetric matching of the distribution of $Y_+$ with that of $Y_-$. 
Observe also that, in terms of $\tilde\nu$, mixture representation \eqref{eq:Eg(X) nu} can be rewritten in the form matching that of \eqref{eq:Eg(X) other}: 
\begin{equation*}
		\E g(X)=
	\int_S \E g(X_{y_+(s),y_-(s)})\,
	\frac{m\,\tilde\nu(\d s)}{\E X_{y_+(s),y_-(s)}^{\ +}},  
\end{equation*}
and at that $\int_S m\,\tilde\nu(\d s)=m=\int_0^m\d h$.
\end{remark*} 

\begin{remark}\label{rem:best} 
Inequality \eqref{eq:best-p-ratios} means that the two-point zero-mean disintegration given in this paper is, on an average, both least-skewed to the right and least-skewed to the left, where the averaging is done according to the distribution of $(Y,U)$ \big(or that of $(Y_+,U)$ or $(Y_-,U)$\big). Inequality \eqref{eq:best-p-ratios} is obtained  
as a special case of \eqref{eq:best-k} (in view of Proposition~\ref{prop:trans2}) with $k(y_1,y_2)=-\frac{y_1^p}{y_2^p}$ or $k(y_1,y_2)=-\frac{y_2^p}{y_1^p}$ for positive $y_1,y_2$; inequality \eqref{eq:best-p-ratio} is a ``two-sided'' version of \eqref{eq:best-p-ratios}.
Generalizing both these one- and two-sided versions, one can take $k(y_1,y_2)\equiv-\frac{f_1(y_1)}{g_1(y_2)}-\frac{f_2(y_2)}{g_2(y_1)}$, where the functions $f_1,f_2$ are nonnegative, continuous, and  nondecreasing, and the functions $g_1,g_2$ are strictly positive,  continuous, and nondecreasing. 

Another two-sided expression of least average skewness is given by \eqref{eq:best-p-symm}, which is obtained as a special case of \eqref{eq:best-k} (again in view of Proposition~\ref{prop:trans2}) with $k(y_1,y_2)\equiv-|y_1-y_2|^p$, for positive $y_1,y_2$; 
using $-|(y_1-y_2)^\pm|^p$ instead of $-|y_1-y_2|^p$, one will have the corresponding right- and left-sided versions; note that, in any of these versions, the condition that 
$\|Y_+\|_p<\infty$ or $\|Y_-\|_p<\infty$ is not needed. More generally, one can take $k(y_1,y_2)\equiv-f(c_1y_1-c_2y_2)$, where $f$ is any nonnegative convex function and $c_1,c_2$ are any  nonnegative constants. 

On the other hand, \eqref{eq:best-p-width} implies that our disintegration has the greatest $p$-average width $|y-\r(y,u)|$. This two-sided version is obtained 
by \eqref{eq:best-k} in view of Proposition~\ref{prop:trans1} 
with $k(y_1,y_2)\equiv|y_1+y_2|^p$, again  for positive $y_1,y_2$; using $|(y_1+y_2)^\pm|^p$ instead will provide the corresponding right- and left-sided versions. 
The largest-$p$-average-width property can also be expressed by taking $|y_1y_2|^p$ or $|y_1^\pm y_2^\pm|^p$ in place of $|y_1+y_2|^p$ or $|(y_1+y_2)^\pm|^p$. 

More generally, one can take $k(y_1,y_2)\equiv f(c_1y_1+c_2y_2)$, where $f$ is any nonnegative convex function and $c_1,c_2$ are again any  nonnegative constants; cf.\ \eqref{eq:best-p-width,neg}. 
%
Thus, our disintegration can be seen as most inhomogeneous in the widths of the two-point zero-mean distributions constituting the mixture. 

Another way to see this is to take any nonnegative $a,b$ and then, in Proposition~\ref{prop:best}, the superadditive function $k(y_1,y_2)\equiv\ii{y_1\ge a,y_2\ge b}$ or $k(y_1,y_2)\equiv\ii{y_1<a,y_2<b}$. 
Then one sees that 
our disintegration 
makes each of the two probabilities --
the large-width probability $\P(y_-\le-a,y_+\ge b)$ and the small-width probability $\P(-a<y_-,y_+<b)$ -- the greatest possible (over all the two-point zero-mean disintegrations, determined by the functions $y_\pm$ as in Proposition~\ref{prop:best}). 

Moreover, \emph{each} of these two properties -- most-large-widths and most-small-widths -- is equivalent to \emph{each} of the two least-average-skewness properties: the least-right-skewness and the least-left-skewness. 
Indeed, for our disintegration, the right-skewness probability $\P(y_->-a,y_+\ge b)$ is the least possible, since it complements the large-width probability $\P(y_-\le-a,y_+\ge b)$ to
$\P(y_+\ge b)$, and it complements the small-width probability $\P(-a<y_-,y_+<b)$ to $\P(y_->-a)$, and at that each of the probabilities $\P(y_+\ge b)$ and $\P(y_->-a)$ is the same over all the disintegrations -- recall part (ii) of Proposition~\ref{prop:best}. Similarly one shows that the least left-skewness is equivalent to each of the properties:  most-large-widths and most-small-widths.
So, there is a rigid trade-off between average skewness and width homogeneity.  

On the other hand, reviewing the proofs of Propositions~\ref{prop:trans1} and \ref{prop:trans2} (especially, see \eqref{eq:Fubini}), one realizes that the superadditive functions $k$ of the form $k(y_1,y_2)\equiv\ii{y_1\ge a,y_2\ge b}$ serve as elementary building blocks; more exactly, these elementary superadditive functions 
(together with the functions that depend only on one of the two arguments) 
represent the extreme rays of the convex cone that is the set of all superadditive functions.  
From these elementary superadditive functions, 
an arbitrary superadditive function can be obtained by mixing and/or limit transition. 
One can now conclude that the exact equivalence between 
the least average skewness and the most inhomogeneous width (of a two-point zero-mean disintegration) occurs at the fundamental, elementary level.   
\end{remark} 
 

\begin{remark*} 
It is rather similar (and even slightly simpler) to obtain  an analogue of Proposition~\ref{prop:best} for the mentioned disintegration (given in \cite[Theorem~2.2]{aizen}) of any probability distribution into the mixture of two-point distributions with the same skewness coefficients (but possibly with different means). In fact, 
a same-skewness analogue of \eqref{eq:best-p-width} in the limit case $p=\infty$ was obtained in \cite[Theorem~2.3]{aizen}; note that the corresponding $L_\infty$ norm of the width equals $\infty$ unless the support of the distribution is bounded. In this paper, we shall not further pursue the matters mentioned in this paragraph.   
\end{remark*}



\subsection{Characteristic properties of reciprocating functions}\label{charact}
To model reciprocating functions, one needs to characterize them. 
Let us begin here with some identities which follow from Proposition~\ref{prop:mix}: 

\begin{proposition}\label{prop:x+- idents}
One has
\begin{equation}\label{eq:x+- idents}
	\int_0^m\frac{\d h}{x_+(h)}=\P(X>0);\quad
	\int_0^m\frac{\d h}{-x_-(h)}=\P(X<0).
\end{equation}
\end{proposition}

It is interesting that identities \eqref{eq:x+- idents} together with properties \eqref{x+-non-decr}--\eqref{x+-left-cont} of Proposition~\ref{lem:left-cont} 
completely characterize the functions $x_+$ and $x_-$. This allows effective modeling of asymmetry patterns of a zero-mean distribution.  
 
\begin{proposition}\label{prop:x+-charact}
For an arbitrary $m\in(0,\infty)$, 
let $y_+\colon[0,m]\to[0,\infty]$ and \break 
$-y_-\colon[0,m]\to[0,\infty]$ be arbitrary functions with $y_+(0)=y_-(0)=0$ and properties \eqref{x+-non-decr}--\eqref{x+-left-cont} of Proposition~\ref{lem:left-cont} such that \big(cf.\ \eqref{eq:x+- idents}\big)
\begin{equation}\label{eq:y+- idents}
	\int_0^m\frac{\d h}{y_+(h)}
	+\int_0^m\frac{\d h}{-y_-(h)}\le1.
\end{equation}
Then there exists a unique zero-mean distribution for which the functions $x_+$ and $x_-$ coincide on $[0,m]$ with the given functions $y_+$ and $y_-$, respectively. 
\end{proposition}

For example, take $m=1$ and let $x_-(h)=-c$ and $x_+(h)=\frac c{1-h}$ for all $h\in(0,1)$, where the constant $c$ is chosen so that the sum of the two integrals in \eqref{eq:x+- idents} be $\frac12$. Then $\P(X=0)=\frac12$, $c=3$, $G(x)=\ii{x\le-3}+(1-\frac3x)\ii{x\ge3}$, and 
$\P(X\le x)=\frac13\ii{-3\le x<0}+\frac56\ii{0\le x<3}+
(1-\frac3{2x^2})\ii{x\ge3}$ for all real $x$.

In applications such as Corollaries~\ref{cor:student-normal} and \ref{cor:stud-asymm}, which are stated in terms of the reciprocating function $\r$, it is preferable to model $\r$ (rather than the functions $x_\pm$). 
Toward that end, let us provide various characterizations of the reciprocating function $\r$. 

For any (nonnegative) measure $\mu$ on $\B(\R)$, let $\mu_+$ be the measure defined by the formula 
\begin{equation}\label{eq:mu+}
	\mu_+(A):=\mu\big(A\cap[0,\infty)\big) 
\end{equation}
for all $A\in\B(\R)$.

For any function $r\colon[-\infty,\infty]\times[0,1]\to[-\infty,\infty]$, let $r_+$ denote the restriction of $r$ to the set $[0,\infty]\times[0,1]$:
\begin{equation*}
 r_+:=r|_{[0,\infty]\times[0,1]}.
\end{equation*}

\begin{proposition}\label{prop:s,nu}
Take any function $\sss\colon[0,\infty]\times[0,1]\to[-\infty,\infty]$ and any  measure $\nu\colon\B(\R)\to[0,\infty)$. Then the following two conditions are equivalent to each other:
\begin{enumerate}[(I)]
	\item
there exists a zero-mean probability measure $\mu$ on $\B(\R)$ 
such that $\mu_+=\nu$ and $(\r_\mu)_+=\sss$; 
\item all of the following conditions hold:
\end{enumerate}
\begin{enumerate}[\qquad(a)]
 \item\label{0} 
 $\sss(0,u)=0$ for all $u\in[0,1]$;
\item\label{nonincr} 
$\sss$ is $\prec$-nonincreasing \big(recall definition \eqref{eq:lex_def}\big);
\item\label{x-l.c} 
$\sss(x,0)$ is left-continuous in $x\in(0,\infty]$;
\item\label{u-l.c} 
$\sss(x,u)$ is left-continuous in $u\in(0,1]$ for each $x\in[0,\infty]$;
\item\label{nu-mass} 
$\nu\big((-\infty,0)\big)=0$; 
\item\label{m<infty} 
$m_\nu:=G_\nu(\infty)<\infty$; 
\item\label{nu,s cont1} 
$\big(\nu(\{x\})=0\ \&\ x\in(0,\infty]\big)
\implies\sss(x,1)=\sss(x,0)$; 
\item\label{nu,s cont2} 
$\Big(\nu\big((x,y)\big)=0\ \&\ 0\le x<y\le\infty\Big)
\implies\sss(x,1)=\sss(y,0)$; 
\item\label{fin} 
$\big(\tG_\nu(x,u)<m_\nu\ \&\ (x,u)\in[0,\infty)\times[0,1]\big)
\implies\sss(x,u)>-\infty$; 
\item\label{s<0} 
$\big(\tG_\nu(x,u)>0\ \&\ (x,u)\in(0,\infty]\times[0,1]\big)
\implies\sss(x,u)<0$; 
\item\label{mu<1} 
$\displaystyle\int_{(0,\infty)\times[0,1]}
\Big(1-\dfrac x{\sss(x,u)}\Big)\,\nu(\d x)\,\d u\le1$.
\end{enumerate}
Moreover, under condition (II), the measure $\mu$ as in condition (I) 
is unique. 
\end{proposition}

In the sequel, we shall be referring to conditions \eqref{0}--\eqref{mu<1} listed in Proposition~\ref{prop:s,nu} as \ref{prop:s,nu}(II)\eqref{0}--\eqref{mu<1} or as \ref{prop:s,nu}(II)(\ref{0}--\ref{mu<1}). Similar references will be being made to conditions listed in other propositions. 

Proposition~\ref{prop:s,nu} implies, in particular, that the set of all zero-mean probability measures $\mu$ on $\B(\R)$ is ``parameterized'' via the one-to-one mapping 
$$\mu\longleftrightarrow(\sss,\nu)=\big(\mu_+,(\r_\mu)_+\big);$$
also, Proposition~\ref{prop:s,nu} provides a complete description of the ``parameter space'' (say $\M$) consisting of all such pairs $(\sss,\nu)=\big(\mu_+,(\r_\mu)_+\big)$. 

Next, we characterize the projection of the parameter space $\M$ onto the ``first coordinate axis''; that is, the set of all functions $\sss$ such that $(\sss,\nu)$ is in $\M$ for some measure $\nu$. In other words, we are now going to characterize the set of the ``positive parts'' $\r_+$ of the reciprocating functions of all zero-mean probability measures $\mu$ on $\B(\R)$. 
Toward that end, with  any function $\sss\colon[0,\infty]\times[0,1]$ associate the ``level'' sets 
\begin{equation}\label{eq:M_s}
	M_\sss(z):=\{(x,u)\in[0,\infty]\times[0,1]\colon\sss(x,u)=z\}, 
\end{equation}
for all $z\in[-\infty,\infty]$, 
and also 
\begin{equation}\label{eq:a,b}
\begin{aligned}
a_\sss &:=\sup\{x\in[0,\infty]\colon\exists u\in[0,1]\ (x,u)\in M_\sss(0)\};\\
b_\sss &:=\inf\{x\in[0,\infty]\colon\exists u\in[0,1]\ (x,u)\in M_\sss(-\infty)\}; 
\end{aligned}	
\end{equation}
here $\inf\emptyset:=\infty$; note that the set of which $a_\sss$ is the supremum contains the point $0$ and hence is never empty.

\begin{proposition}\label{prop:s(x,u)}
Take any function $\sss\colon[0,\infty]\times[0,1]\to[-\infty,\infty]$. 
Then the following two conditions are equivalent to each other:
\begin{enumerate}[(I)]
	\item
	there exists a zero-mean probability measure $\mu$ on $\B(\R)$ 
such that $(\r_\mu)_+=\sss$; 
	\item
conditions \ref{prop:s,nu}(II)(\ref{0}--\ref{u-l.c}) hold, along with these three conditions:  
\begin{enumerate}[\ ]
\item[(\ref{fin}')]\label{fin-s} 
the set $M_\sss(-\infty)$ has one of the following three forms:
\begin{itemize}
	\item 
	$\emptyset$ or 
	\item
	$[b,\infty]\times[0,1]$ for some $b\in(0,\infty]$ or
	\item
	$\{(b,1)\}\cup\big((b,\infty]\times[0,1]\big)$ for some $b\in(0,\infty)$; 
\end{itemize}
in fact, this $b$ necessarily coincides with $b_\sss$;
\item[(\ref{s<0}')]\label{s<0_s} 
the set $M_\sss(0)$ has one of the following two forms:
\begin{itemize}
	\item
	$[0,a]\times[0,1]$ for some $a\in[0,\infty]$ or
	\item
	$\big([0,a)\times[0,1]\big)\cup\{(a,0)\}$ for some $a\in(0,\infty)$; 
\end{itemize}
in fact, this $a$ necessarily coincides with $a_\sss$;
\item[(\ref{mu<1}')]\label{mu<1_s}
$\displaystyle\int_0^1
\dfrac{\d u}{\sss(a_\sss,u)}>-\infty$ if $\sss(a_\sss,1)\ne\sss(a_\sss,0)$.
\end{enumerate}
\end{enumerate}
\end{proposition}


Now let us characterize those 
$\sss=\r_+$ 
that determine the corresponding reciprocating function $\r$ uniquely. 

\begin{proposition}\label{prop:s(x,u)_uniq}
Take any function $\sss\colon[0,\infty]\times[0,1]\to[-\infty,\infty]$. 
Then the following two conditions are equivalent to each other:
\begin{enumerate}[(I)]
	\item\label{(I)}
	there exists a \emph{unique} function $\r$ such that $\r_+=\sss$ and 
	$\r$ coincides with the reciprocating function $\r_\mu$ of some zero-mean probability measure $\mu$ on $\B(\R)$; 
	\item
conditions \ref{prop:s,nu}(II)(\ref{0}--\ref{u-l.c}) and  
\ref{prop:s(x,u)}(II)(\ref{fin}'--\ref{mu<1}') hold along with this almost-strict-decrease condition:  
\end{enumerate}

\begin{enumerate}[\qquad (u)]
	\item
\label{growth} 
for any $x$ and $y$ such that $0\le x<y\le\infty$, one of the following three conditions must occur:
\begin{itemize}
	\item 
	$\sss(x,1)>\sss(y,0)$ or 
	\item
	$\sss(x,1)=\sss(y,0)=-\infty$ or
	\item
	$\sss(x,1)=\sss(y,0)=0$. 
\end{itemize}
\end{enumerate}
Moreover, if either condition (I) or (II) holds, then for any zero-mean probability measure $\mu$ on $\B(\R)$ such that $(\r_\mu)_+=\sss$ one has $\supp(\mu_+)=\R\cap[a_\sss,b_\sss]$. 
\end{proposition}

Proposition~\ref{prop:s(x,u)_uniq} shows that the intrinsic ``cause'' (that is, the ``cause'' expressed only in terms of the function $\sss$ itself) of the possible non-uniqueness of $\r$ given $\sss$ is that $\sss$ may fail to satisfy the almost-strict-decrease condition \ref{prop:s(x,u)_uniq}(II)(u), while an extrinsic ``cause'' of such non-uniqueness is that the support set of the ``positive'' part $\mu_+$ of $\mu$ may fail to be connected. On the other hand, the next proposition shows that another extrinsic ``cause'' of the possible non-uniqueness is that the ``negative'' part $\mu_-$ of $\mu$ may fail to be non-atomic, where $\mu_-$ is the measure defined by the formula (cf.\ \eqref{eq:mu+})
\begin{equation*}
	\mu_-(A):=\mu\big(A\cap(-\infty,0)\big) 
\end{equation*}
for all $A\in\B(\R)$.

\begin{proposition}\label{prop:mu-,uniq}
Take any function $\sss\colon[0,\infty]\times[0,1]\to[-\infty,\infty]$. Then there exists \emph{at most one} function $\r$ such that $\r_+=\sss$ and $\r=\r_\mu$ for some zero-mean probability measure $\mu$ on $\B(\R)$ such that $\mu_-$ is non-atomic. (Of course, the same conclusion holds with $\mu$ in place of $\mu_-$.) 
\end{proposition}

Next, let us restrict our attention to the reciprocating functions of non-atomic zero-mean probability measures. Compare the following with Proposition~\ref{prop:s,nu}; at that, recall Remark~\ref{rem:cont}(ii). 

\begin{proposition}\label{prop:s(x),nu}
Take any function $\sss\colon[0,\infty]\to[-\infty,\infty]$ and any \emph{non-atomic} measure $\nu\colon\B(\R)\to\R$. Then the following two conditions are equivalent to each other:
\begin{enumerate}[(I)]
	\item
there exists a \emph{non-atomic} zero-mean probability measure $\mu$ on $\B(\R)$ 
such that $\mu_+=\nu$ and $(\r_\mu)_+=\sss$; 
\item 
conditions \ref{prop:s,nu}(II)(\ref{0}--\ref{x-l.c},\ref{nu-mass},\ref{m<infty},\ref{fin},\ref{s<0}) hold \big(with $\sss(x)$ and $G(x)$ in place of $\sss(x,u)$ and $\tG(x,u)$\big), along with conditions 
\begin{enumerate}[\qquad ]
\item[(\ref{nu,s cont2}')]\label{nu,s cont2 no-atom} 
$0\le x<y\le\infty\implies
\Big(\nu\big((x,y)\big)=0\ \iff \sss(x)=\sss(y)\Big)$; 
\item[(\ref{mu<1}'')]\label{mu=1} 
$\displaystyle\int_{(0,\infty)}
\Big(1-\dfrac x{\sss(x)}\Big)\,\nu(\d x)=1$.
\end{enumerate}
\end{enumerate}
Moreover, under condition (II), the measure $\mu$ as in (I) is unique. 
\end{proposition}

The following ``non-atomic'' version of  Propositions~\ref{prop:s(x,u)} and \ref{prop:s(x,u)_uniq} 
is based in part on the well-known theorem that every non-empty closed set (say in $\R^d$) without isolated points is the support of some non-atomic probability measure; see e.g.\ \cite{varad}. 

\begin{proposition}\label{prop:s(x)}
Take any function $\sss\colon[0,\infty]\to[-\infty,\infty]$. Then the following two conditions are equivalent to each other:
\begin{enumerate}[(I)]
	\item
there exists a function $\r$ such that $\r_+=\sss$ and $\r=\r_\mu$ for some 
\emph{non-atomic} zero-mean probability measure $\mu$ on $\B(\R)$;  
\item 
conditions \ref{prop:s,nu}(II)(\ref{0}--\ref{x-l.c}) \big(with $\sss(x)$ in place of $\sss(x,u)$\big) hold, along with condition 
\begin{enumerate}[\qquad ]
\item[(\ref{nu,s cont2}'')]\label{s cont2 no-atom} 
$\big(0\le x<y\le\infty\ \&\ \sss(x+)<\sss(x)\big)\implies
\sss(y)<\sss(x+)$. 
\end{enumerate}
\end{enumerate}
Moreover, under condition (II), the function $\r$ as in (I) is unique. 
\end{proposition}

Now we restrict our attention further, to non-atomic zero-mean probability measures with a \emph{connected support}. 
Take any $a_-$ and $a_+$ such that $-\infty\le a_-<0<a_+\le\infty$ and let $I:=\R\cap[a_-,a_+]$.
The following are the ``connected support'' versions of  Propositions~\ref{prop:s(x),nu} and \ref{prop:s(x)}. 

\begin{proposition}\label{prop:s(x),nu conn}
Take any function $\sss\colon[0,\infty]\to[-\infty,\infty]$ and any non-atomic measure $\nu\colon\B(\R)\to\R$. Then the following two conditions are equivalent to each other:
\begin{enumerate}[(I)]
	\item
there exists a \emph{non-atomic} zero-mean probability measure $\mu$ on $\B(\R)$ such that $\supp\mu=I$, $\mu_+=\nu$, and  
$(\r_\mu)_+=\sss$; 
\item 
conditions $\sss(0)=0$, $\nu\big((-\infty,0)\big)=0$, 
$m_\nu=G_\nu(\infty)<\infty$, and \ref{prop:s(x),nu}(II)(\ref{mu<1}'') 
hold, along with conditions 
\begin{enumerate}[\qquad ]
\item[(\ref{nonincr}')] 
$\sss$ is strictly decreasing on $[0,a_+]$; 
\item[(\ref{x-l.c}')] 
$\sss$ is continuous on $[0,\infty]$; 
\item[(\ref{nu,s cont2}'')] 
$\supp\nu=I_+:=\R\cap[0,a_+]$; 
\item[(\ref{fin}'')] 
$\sss=a_-$ on $[a_+,\infty]$. 
\end{enumerate}
\end{enumerate}
Moreover, under condition (II), the measure $\mu$ as in (I) is unique. 
\end{proposition}

\begin{proposition}\label{prop:s(x) conn}
Take any function $\sss\colon[0,\infty]\to[-\infty,\infty]$. Then the following two conditions are equivalent to each other:
\begin{enumerate}[(I)]
	\item
there exists a 
function $\r$ such that $\r_+=\sss$ and $\r=\r_\mu$ for some 
\emph{non-atomic} zero-mean probability measure $\mu$ on $\B(\R)$ with $\supp\mu=I$; 
\item 
conditions $\sss(0)=0$ and \ref{prop:s(x),nu conn}(II)(\ref{nonincr}',\ref{x-l.c}',\ref{fin}'') hold. 
\end{enumerate} 
Moreover, under condition (II), the function $\r$ as in (I) is unique.  
\end{proposition}

In contrast with Proposition~\ref{prop:s(x) conn}, the following proposition characterizes the reciprocating functions $\r$ of non-atomic zero-mean probability measures with a connected support 
(rather than the ``positive parts'' $\sss=\r_+$ of such functions $\r$).  

\begin{proposition}\label{prop:r cont charact}\ 
Take any function $\r\colon[-\infty,\infty]\to[-\infty,\infty]$. Then the following two conditions are equivalent to each other:
\begin{enumerate}[(I)]
	\item
there exists a \emph{non-atomic} zero-mean probability measure $\mu$ on $\B(\R)$ such that $\supp\mu=I$ and  
$\r_\mu=\r$; 
\item 
condition $\r(0)=0$ holds, along with the following:  
\begin{enumerate}[\qquad ]
\item[(\ref{nonincr}'')] 
$\r$ is strictly decreasing on $[a_-,a_+]$; 
\item[(\ref{x-l.c}'')] 
$\r$ is continuous on $[-\infty,\infty]$; 
\item[(\ref{fin}''')] 
$\r=a_+$ on $[-\infty,a_-]$ and $\r=a_-$ on $[a_+,\infty]$;
\item[($\r\circ\r$)] 
$\r\big(\r(x)\big)=x$ for all $x\in[a_-,a_+]$.  
\end{enumerate} 
\end{enumerate}
\end{proposition}

Our final characterization concerns the case when it is desirable to avoid zero-mean probability measures $\mu$ with a density that is discontinuous at $0$ (say, as an unlikely shape). 

\begin{proposition}\label{prop:loc-symm}\ 
Take any function $\r\colon[-\infty,\infty]\to[-\infty,\infty]$. Then the following two conditions are equivalent to each other:
\begin{enumerate}[(I)]
	\item
there exists a \emph{non-atomic} zero-mean probability measure $\mu$ on $\B(\R)$ such that $\r_\mu=\r$, $\supp\mu=I$, and in a neighborhood of $0$ measure $\mu$ has a continuous strictly positive density;  
\item 
condition \ref{prop:r cont charact}(II) holds, along with the following: $\r$ is continuously differentiable in a neighborhood of $0$.  
\end{enumerate}
Moreover, if either condition (I) or (II) holds, then necessarily $\r'(0)=-1$, that is, one has the approximate local symmetry condition  $\r(x)\sim-x$ as $x\to0$.
\end{proposition}

\subsection{Modeling reciprocating functions}\label{model} 
As pointed out by Bartlett \cite{bartlett} and confirmed by Ratcliffe \cite{ratcl}, skewness affects the $t$ distribution (and hence that of the self-normalized sum) more than kurtosis does. These results are in agreement with the result by Hall and Wang \cite{hall-wang}.

Tukey \cite[page 206]{tukey-biol} wrote, ``It would be highly desirable to have a modified version of the $t$-test with a greater resistance to skewness... .'' This concern is addressed in the present paper by such results as Corollaries~\ref{cor:student-normal} and \ref{cor:stud-asymm}.  

Closely related to this 
is the question of modeling asymmetry. Tukey \cite{tukey-trans} proposed using the power-like transformation functions of the form 
$z(y)=a(y+c)^p+b$, $y>-c$,
with the purpose of symmetrizing the data.  
To deal with asymmetry and heavy tails, Tukey also proposed (see Kafadar \cite[page 328]{kafadar} and Hoaglin \cite{hoaglin}) the so-called $g$-$h$ technology, whereby to fit the data to a $g$-$h$ distribution, which is the distribution of a r.v.\ of the form $
e^{hZ^2/2}(e^{gZ}-1)/g$, where $Z\sim N(0,1)$, so that the parameters $g$ and $h$ are responsible, respectively, for the skewness of the distribution and the heaviness of the tails. 

We propose modeling asymmetry using reciprocating functions.  
In view of Propositions~\ref{prop:s(x) conn} and \ref{prop:r cont charact}, the reciprocating function $\r$ of any non-atomic zero-mean probability measure $\mu$ with a connected support can be constructed as follows. 

\begin{constr}\label{constr:s}
\begin{enumerate}[(i)]
	\item
	Take any $a_-$ and $a_+$ such that $-\infty\le a_-<0<a_+\le\infty$ and let $I_+:=\R\cap[0,a_+]$.
	\item
	Take any function $\sss\colon[0,\infty]\to[-\infty,\infty]$ such that $\sss(0)=0$, $\sss$ equals $a_-$ on $[a_+,\infty]$ and is strictly decreasing and continuous on $[0,a_+]$. 
	\item
	Define $\r$ by the formula	
\begin{equation*}
\r:=
	\begin{cases}
a_+ &\text{ on }[-\infty,a_-];\\
\ts^{-1} &\text{ on }[a_-,0];\\
\sss &\text{ on }[0,\infty],
\end{cases}	
\end{equation*}
where $\ts:=\sss|_{[0,a_+]}$, the restriction of the function $\sss$ to the interval $[0,a_+]$. 
\end{enumerate} 
\end{constr}

\begin{example}\label{ex:r-powers}
In accordance with Construction~\ref{constr:s} and Proposition~\ref{prop:loc-symm}, one can suggest the two-parameter family of reciprocating functions defined by the formula: 
\begin{equation*}
	\r(x):=\r_{p,c}(x):=
	\begin{cases}
\frac cp\,\big(1-(1+x/c)^p\big) &\text{ for }x\in[0,\infty];\\
c\,\big((1-px/c)^{1/p}-1\big) &\text{ for }x\in(-\infty,0]\text{ such that }px<c,
\end{cases}
\end{equation*}
with the convention that $\r(\infty):=\r(\infty-)$; here, $p\in\R\setminus\{0\}$ and $c>0$ are real numbers, which may be referred to as the shape and scale parameters, respectively. Indeed, one can see that mere re-scaling of $\mu$ (or a corresponding r.v.\ $X$) results only in a change of $c$: if $\r_{p,1}=\r_X$ for some zero-mean r.v.\ $X$, then $\r_{p,c}=\r_{cX}$. 
For $x\in[-\infty,0]$ such that $px\ge c$ (that is, for $x\in[-\infty,\frac cp]$ when $p<0$), we set $\r_{p,c}(x):=\infty$, in accordance with the general description of Construction~\ref{constr:s}. 
Let us also extend the family of functions $\r_{p,c}$ to $p=0$ by continuity:
\begin{equation*}
	\r_{0,c}(x):=\lim_{p\to0}\r_{p,c}(x)=
	\begin{cases}
-c\,\ln(1+x/c) &\text{ for }x\in[0,\infty];\\
c\,(e^{-x/c}-1) &\text{ for }x\in[-\infty,0].
\end{cases}
\end{equation*}
The corresponding intervals $[a_-,a_+]$ here coincide with  $[-\infty,\infty]$ if $p\ge0$ and with $[\frac cp,\infty]$ if $p<0$. 

Case $p=1$ corresponds to the pattern of perfect symmetry of $\mu$; that is, $\r(x)=-x$ for all $x$ (recall Proposition~\ref{prop:symm}). Case $p>1$ corresponds to a comparatively long (or, equivalently, heavy) left tail of $\mu$, so that $\mu$ will be skewed to the left. Similarly, case $p<1$ corresponds to a comparatively long (or heavy) right tail of $\mu$. Thus, $p$ can be considered as the asymmetry parameter. 

Another limit case is when $p\to\pm\infty$ and $c\to\infty$ in such a manner that $\frac cp\to\pm\la$, for some $\la\in(0,\infty)$, 
and this limit is given by
\begin{equation*}
	\r_{\pm\infty,\la}(x):=
	\begin{cases}
\rule[-8pt]{0pt}{18pt}
\pm\la\,(1-e^{\pm x/\la}) &\text{ for }x\in[0,\infty];\\
\pm\la\,\ln(1\mp x/\la) &\text{ for $x\in[-\infty,0]$ such that $\pm x<\la$},
\end{cases}
\end{equation*}
where $\la$ and $\pm\infty$ play, respectively, the roles of the scale and shape (or, more specifically, asymmetry) parameters. 

Yet another limit case is when $c\to0$ and $p\to1$ in such a manner that $c^{p-1}\to\ka$, for some $\ka\in(0,\infty)$, 
and this limit is given by
\begin{equation*}
	\r_{1,0;\ka}(x):=
	\begin{cases}
-x/\ka &\text{ for }x\in[0,\infty];\\
-\ka x &\text{ for }x\in[-\infty,0],
\end{cases}
\end{equation*}
However, in this case the property $\r'(0)=-1$ is lost (in fact, $\r_{1,0;\ka}$ is not differentiable at $0$) unless $\ka=1$, so that, by Proposition~\ref{prop:loc-symm}, no corresponding zero-mean distribution $\mu$ can have a density that is strictly positive and continuous at $0$. 

\smallskip
\parbox{1.3in}{
\includegraphics[width=1.2in,height=1.2in]{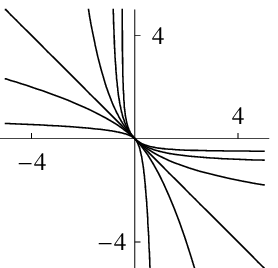}
}
\parbox{3.5in}{
Here on the left one can see parts of the graphs $\{\big(x,\r_{p,1}(x)\big)\colon x\in[a_-,a_+]\}$ with $p=-2,-1,0,1,2,8$. 
Each graph is symmetric about the diagonal $\Delta:=\{(x,x)\colon x\in[-\infty,\infty]\}$, as it should be according to the reciprocity property \ref{prop:r cont charact}(II)($\r\circ\r$). The tighter the graph of the reciprocating function embraces the first quadrant, the more skewed is the corresponding distribution to the right; and the tighter the graph embraces the third quadrant, the more skewed is the distribution to the left. In 
}

\medskip
\noindent this example, the greater is $p$, the more skewed to the left must the corresponding zero-mean distribution $\mu$ be. 
\end{example}

\begin{constr}\label{constr:hyperb}
Reciprocity property \ref{prop:r cont charact}(II)($\r\circ\r$) of $\r$ implies that the graph $\{\big(x,\r(x)\big)\colon x\in[a_-,a_+]\}$ can be obtained in the form $\{\big(x,y\big)\colon F(x,y)=0, x\in[a_-,a_+], y\in[a_-,a_+]\}$, where $F$ is a symmetric function, which must also satisfy condition $F(0,0)=0$, since $\r(0)=0$. 
\end{constr}
A simplest such function is the quadratic function $F$ given by the formula
\begin{equation}\label{eq:F}
	F(x,y)\equiv Ax^2+2Bxy+Ay^2+cx+cy, 
\end{equation}
so that the graphs are elliptic or hyperbolic arcs symmetric about the diagonal $\Delta$ and passing through the origin. However, here we shall not consider this construction in detail. 

Instead, let us turn to 
\begin{constr}\label{constr:a}
The symmetry of the graph $\{\big(x,\r(x)\big)\colon x\in(a_-,a_+)\}$ of a reciprocating function $\r$ about the diagonal $\Delta$ suggests that 
$\r$ is uniquely determined by 
a function (say $\aaa$) that maps, for each $x\in(a_-,a_+)$, the width $\w(x):=|x-\r(x)|$ to the asymmetry 
$\al(x):=x+\r(x)$ of the zero-mean distribution on the two-point set $\{x,\r(x)\}$. 
\big(Note that, by \ref{prop:r cont charact}(II)(\ref{nonincr}'',\ref{x-l.c}''), the width function $\w$ is continuous on $[-\infty,\infty]$, strictly increasing on $[0,a_+]$ (from $0$ to $a_+-a_-$), and strictly decreasing on $[a_-,0]$ (from $a_+-a_-$ to $0$).\big) 
The function $\aaa$ may be referred to as the \emph{asymmetry pattern function} of a given zero-mean distribution. 
\end{constr}

Details of Construction~\ref{constr:a} are presented in 
\begin{proposition}\label{prop:a}\ 

\begin{enumerate}[(I)]
	\item If $\r$ is a reciprocating function as in Proposition~\ref{prop:r cont charact}, then there exists a unique function $\aaa\colon[0,a_+-a_-)\to\R$ 
such that $\aaa(0)=0$, 
\begin{equation}\label{eq:a}
	x+\r(x)=\aaa\big(|x-\r(x)|\big)
\end{equation}
for all $x\in(a_-,a_+)$, and the following strict Lip(1) condition (Lipschitz with constant factor 1)  holds: 
\begin{equation*}
	|\aaa(w_2)-\aaa(w_1)|<w_2-w_1
\end{equation*}
for all $w_1$ and $w_2$ such that $0\le w_1<w_2<a_+-a_-$. 
Also, $\aaa(w)\to a_+ + a_-$ as $w\uparrow a_+ - a_-$. 
\item
Vice versa, if a function $\aaa\colon[0,a_+-a_-)\to\R$ is strictly Lip(1), $\aaa(0)=0$, and $\aaa(w)\to a_+ + a_-$ as $w\uparrow a_+ - a_-$, then there exists a unique reciprocating function $\r$ such as in Proposition~\ref{prop:r cont charact} that satisfies condition \eqref{eq:a}. In fact, then one necessarily has 
\begin{equation}\label{eq:r,xi,rho}
	\r(x)=
	\begin{cases}
-\rho\big(\xi^{-1}(x)\big) & \text{ if }x\in[0,a_+);\\
\xi\big(\rho^{-1}(-x)\big) & \text{ if }x\in(a_-,0],
	\end{cases}
\end{equation}
where the functions $\xi$ and $\rho$ are defined by 
\begin{equation}\label{eq:xi,rho}
	\text{$\xi(w):=\tfrac12(w+\aaa(w))$ and $\rho(w):=\tfrac12(w-\aaa(w))$ 
for all $w\in[0,a_+-a_-)$,}
\end{equation}
and they are continuously and strictly increasing on $[0,a_+-a_-)$ from $0$ to $a_+$ and $-a_-$, respectively. 
\item
Moreover, a reciprocating function $\r$ such as in Proposition~\ref{prop:r cont charact} 
is continuously differentiable in a neighborhood of $0$ 
if and only if the corresponding asymmetry pattern function $\aaa$ is continuously differentiable in an open right neighborhood (r.n.) of $0$ and $\aaa'(0+)=0$.  
\end{enumerate}
\end{proposition}

In particular, Proposition~\ref{prop:a} shows that the asymmetry pattern function $\aaa$ is necessarily Lipschitz and hence absolutely continuous, with a density $\aaa'(w)=\frac{\d\aaa(w)}{\d w}$ such that 
\begin{equation*}
 -1<\aaa'(w)<1
\end{equation*}
for almost all $w\in[0,a_+-a_-)$. In view of \eqref{eq:a}, this density $\aaa'$ may be considered as the rate of change of asymmetry $\al=x+\r(x)$ relative to the varying width $w=|x-\r(x)|$ of the constituent zero-mean distribution on the two point set $\{x,\r(x)\}$. For instance, if at the given width $w$ this rate $\aaa'(w)$ is close to 1, then at this width $w$ the distribution's skewness to the right is growing fast. Also, for all $w\in(0,a_+-a_-)$, the ratio 
$$\frac{\aaa(w)}w=\frac1w\int_0^w \aaa'(v)\,\d v$$
represents the average asymmetry-to-width rate over all widths from $0$ to $w$. 
Thus, Construction~\ref{constr:a} provides a flexible and sensitive tool to model asymmetry patterns.  


One can see that in Example~\ref{ex:r-powers} the asymmetry-to-width rate $\aaa'$ strictly increases or decreases from $0$ to $1$ or $-1$ as $w$ increases from $0$ to $\infty$, depending on whether $p<1$ or $p>1$, and $\aaa'(w)=0$ for all $w\in[0,\infty)$ if $p=1$. Moreover, 
\begin{equation*}
	\aaa'(w)\sim
	\begin{cases}
1-c_1\,e^{-w/\la} & \text{ if } p=-\infty,\\
1-c_2\,w^{p-1} & \text{ if } -\infty<p\le0,\\
1-c_3\,w^{1-1/p} & \text{ if } 0<p<1,\\
-1+c_4\,w^{1/p-1} & \text{ if } 1<p\le\infty
	\end{cases}	
\end{equation*}
as $w\to\infty$, where $c_1,\dots,c_4$ are positive real constants, depending only on the parameters $p$ and $c$
; here 
$1/p-1:=-1$ for $p=\infty$. 

Let us now provide examples of two parametric families of reciprocating functions obtained using the asymmetry-to-width rate $\aaa'$ as the starting point. 

\begin{example}\label{ex:a}
Take any $\al\in[-1,1]$ and $c\in(0,\infty)$, and consider the asymmetry-to-width rate of the form
\begin{equation*}
 \aaa'(w)=\al\Big(1-\frac{c^2}{(c+w)^2}\Big)
\end{equation*}
for all $w\in[0,\infty)$, so that, for $\al\in(0,1]$, the rate $\aaa'(w)$ increases from $0$ to $\al$ as $w$ increases from $0$ to $\infty$; similarly, for $\al\in[-1,0)$, the rate $\aaa'(w)$ decreases from $0$ to $\al$ as $w$ increases from $0$ to $\infty$. Then the corresponding asymmetry pattern function $\aaa$ is given by
\begin{equation*}
 \aaa(w)=\aaa_{\al,c}(w)=\al\frac{w^2}{c+w}
\end{equation*}
for all $w\in[0,\infty)$, and, by \eqref{eq:r,xi,rho}, the corresponding reciprocating function $\r$ is given by
\begin{equation*}
 \r(x)=\r_{\al,c}(x)=
\frac{c+2\al x-\sqrt{(c+2|x|)^2+8\al cx}}{2(\al+\sign x)}
\end{equation*}
for $\al\in(-1,1)$ and 
all $x\in\R$; expressions for $\r_{1,c}$ and $\r_{-1,c}$ are of different forms. 
Note that $\r_{\al,c}(x)\sim\frac{\al\mp1}{\al\pm1}\,x$ as $x\to\pm\infty$, for each $\al\in(-1,1)$; on the other hand, $\r_{1,c}((-\frac c2)+)=\infty$, $\r_{1,c}(\infty-)=-\frac c2$, $\r_{-1,c}(\frac c2-)=-\infty$, $\r_{-1,c}((-\infty)+)=\frac c2$. 
The parameters $\al$ and $c$ are, respectively, the shape (or, more specifically, asymmetry) and scale parameters. The graph of $\r_{\al,c}$ is the union of two hyperbolic arcs of two different hyperbolas: $-(1+\al)r^2+2\al xr+(1-\al)x^2+cr+cx=0$ (used for $x\ge0$) and $(1-\al)r^2+2\al xr-(1+\al)x^2+cr+cx=0$ (used for $x\le0$) -- cf. \eqref{eq:F}.  

\smallskip
\parbox{1.3in}{
\includegraphics[width=1.2in,height=1.15in]{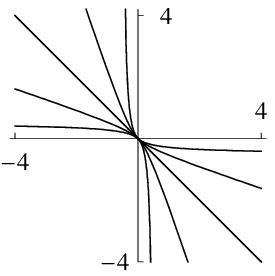}
}
\parbox{3.5in}
{Yet, by Proposition~\ref{prop:a}, all these reciprocating functions $\r_{\al,c}$ are continuously differentiable in neighborhood of $0$ (in fact, they are so wherever on $\R$ they take finite values). 
On the left one can see parts of the graphs $\{\big(x,\r_{\al,1}(x)\big)\colon x\in[a_-,a_+]\}$ with $\al=-1,-\frac12,0,\frac12,1$. 
In such an example, the shape (or, more specifically, asymmetry) parameter $\al$ can also be considered as a scale parameter -- but in the direction of the diagonal $\Delta=\{(x,x)\colon x\in[-\infty,\infty]\}$.  
}

\medskip

\end{example}

\begin{example}\label{ex:aa}
Take any $\al\in[-1,1]$ and $c\in(0,\infty)$, and consider the asymmetry-to-width rate of the form
\begin{equation*}
 \aaa'(w)=
 \frac{16\al c^3}{3\sqrt3} \,\frac w{(c^2+w^2)^2}
\end{equation*}
for all $w\in[0,\infty)$, 
so that, for $\al\in(0,1]$, the rate $\aaa'(w)$ increases from $0$ to $\al$ and then decreases from $\al$ to $0$ 
as $w$ increases from $0$ to $c/\sqrt3$ to $\infty$; similarly, for $\al\in[-1,0)$, the rate $\aaa'(w)$ decreases from $0$ to $\al$ and then increases from $\al$ to $0$ as $w$ increases from $0$ to $c/\sqrt3$ to $\infty$.
The corresponding asymmetry pattern function $\aaa$ is given by
\begin{equation*}
 \aaa(w)=\frac{8\al c}{3\sqrt3} \,\frac{w^2}{c^2+w^2}
\end{equation*}
for all $w\in[0,\infty)$, and, using \eqref{eq:r,xi,rho}, one can see that the corresponding reciprocating function $\r=\r_{\al,c}$ is given by an algebraic expression involving certain cubics. 
In particular, $\r_{\al,c}(x)\sim-x+\frac{8\al c}{3\sqrt3}$ as $|x|\to\infty$. 
Again, the parameters $\al$ and $c$ are, respectively, the shape (or, more specifically, asymmetry) and scale parameters. 
Alternatively, in this example as well, the shape/asymmetry parameter $\al$ can also be considered as a scale parameter, in the direction of the diagonal $\Delta$. 

\smallskip
\parbox{1.3in}{
\includegraphics[width=1.2in,height=1.2in]{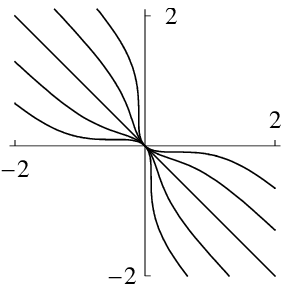}
}
\parbox{3.5in}{
Again, the parameters $\al$ and $c$ are, respectively, the shape (or, more specifically, asymmetry) and scale parameters. 
Alternatively, in this example as well, the shape/asymmetry parameter $\al$ can also be considered as a scale parameter, in the direction of the diagonal $\Delta$. 
Here on the left one can see parts of the graphs $\{\big(x,\r_{\al,1}(x)\big)\colon x\in[a_-,a_+]\}$ with $\al=-1,-\frac12,0,\frac12,1$. 
}

\medskip
\end{example}

\begin{constr}\label{constr:xi,rho}
Looking back at Proposition~\ref{prop:a}, one can see that yet another way to construct an arbitrary 
reciprocating function $\r$ as in Proposition~\ref{prop:r cont charact} is by using \eqref{eq:r,xi,rho} with arbitrary functions $\xi$ and $\rho$ that are continuously and strictly increasing on $[0,a_+-a_-)$ from $0$ to $a_+$ and $-a_-$, respectively \big(and also using condition \ref{prop:r cont charact}(II)(\ref{fin}''') to complete the construction of $\r$\big). In fact, the functions $\xi$ and $\rho$ defined by \eqref{eq:xi,rho} also satisfy the strict Lip(1) condition; still, even if $\xi$ or $\rho$ violates this Lip(1) restriction, the function $\r$ defined by \eqref{eq:r,xi,rho} will have all the characteristic properties \ref{prop:r cont charact}(II)(\ref{nonincr}'')--($\r\circ\r$).    
However, in this paper we shall not pursue this construction further.
\end{constr} 

Examples~\ref{ex:r-powers}, \ref{ex:a}, \ref{ex:aa} of parametric families of reciprocating functions already appear to represent a wide enough variety. 
Moreover, Constructions~1--4 given in this subsection appear sufficiently convenient and flexible for efficient modeling of  asymmetry patterns that may arise in statistical practice. 
In any case, each of these constructions -- of reciprocating functions for non-atomic distributions with connected support -- is quite universal. 
\big(For discrete distributions, it appears more convenient to model asymmetry patterns based on the characterization of the functions $x_\pm$ provided by Proposition~\ref{prop:x+-charact}.\big) 
In any such parametric or nonparametric model, the reciprocating function can be estimated in a standard manner, as follows: substituting the empirical distribution for the ``true'' unknown distribution $\mu$, one obtains empirical estimates of the function $G$ and hence empirical estimates of the functions $x_\pm$ and $\r$; then, if desired, the empirical estimate of $\r$ can be fit into an appropriate parametric family of reciprocating functions. 

\section{Proofs
} \label{sec:proofs}
In Subsection~\ref{proofs:props}) we shall prove the propositions stated in Section~\ref{discuss} and then, in Subsection~\ref{proofs:main}, 
the theorems and corollaries stated in Section~\ref{sec:results}. 

\subsection{Proofs of propositions} \label{proofs:props} 

\begin{proof}[Proof of Proposition~\ref{lem:left-cont}]
Implication $\Longleftarrow$ in \eqref{eq:L1} follows immediately from definition \eqref{eq:x+}, since $x_+(h)$ is a lower bound of the corresponding set. 
Implication $\Longrightarrow$ in \eqref{eq:L1} follows from \eqref{eq:x+} as well, taking also into account that, by \eqref{eq:G properties}, the function $G$ is non-decreasing on $[0,\infty]$ and right-continuous on $[0,\infty)$ .  Thus, one obtains \eqref{eq:L1}. Equivalence \eqref{eq:L1-} is proved similarly. 

Inequalities \eqref{eq:less+} and \eqref{eq:less-} follow immediately from \eqref{eq:L1} and \eqref{eq:L1-}. 
The first inequalities in \eqref{eq:bet+} and \eqref{eq:bet-} follow immediately from \eqref{eq:less+} and \eqref{eq:less-}, while the second ones are special cases of
\eqref{eq:L1} and \eqref{eq:L1-}, respectively. 

Next, let us prove \eqref{eq:nonat+} and \eqref{eq:nonat-}. Assume indeed that $0\le h_1<h_2$ and $x_+(h_1)=x_+(h_2)=x$. Then, by \eqref{eq:L1} and \eqref{eq:bet+}, $G(x)\ge h_2>h_1\ge G(x-)$, so that $x\,\mu(\{x\})>0$. 
This proves \eqref{eq:nonat+}. Quite similarly one proves \eqref{eq:nonat-}. 

Property \eqref{x+-non-decr} follows immediately from definitions \eqref{eq:x+} and \eqref{eq:x-}. 

Property \eqref{x+-finite} follows because $G(x)\to m$ as $|x|\to\infty$.

Since the functions $\pm x_\pm$ are nonnegative \big(by definitions \eqref{eq:x+} and \eqref{eq:x-}\big) and $G(0)=0$, property \eqref{x+-pos} follows by \eqref{eq:L1} and \eqref{eq:L1-}, which imply that $G(x_\pm(h))\ge h>0$ for all $h\in(0,m]$.

Finally, let us now prove property \eqref{x+-left-cont}. Take any 
$h_0\in(0,m]$ and let $x_0:=x_+(h_0)$. Then, by property \eqref{x+-pos}, one has $x_0>0$. 
Next, take any $x\in(0,x_0)$. Then, by \eqref{eq:less+} and \eqref{eq:bet+},  $G(x)<h_0\le G(x_0)$ and, by \eqref{eq:L1}, one has $x_+(h)\in(x,x_0]$ for all $h$ in the interval $(G(x),G(x_0)]$ and hence for all $h$ in the nonempty subinterval $(G(x),h_0]$ of $(G(x),G(x_0)]$. This implies that $x_+$ is left-continuous on $(0,m]$; similarly, $-x_-$ is so. 
\end{proof}

\begin{proof}[Proof of Proposition~\ref{prop:lex}]
If $0\le x_1<x_2\le\infty$ and $0\le u_1,u_2\le1$ then $\tG(x_1,u_1)\le G(x_1)\le G(x_2-)\le\tG(x_2,u_2)$; and if $0\le x\le\infty$ and $0\le u_1<u_2\le1$ then  $\tG(x,u_1)\le\tG(x,u_2)$, by \eqref{eq:H}. This shows that the function $\tG$ is indeed $\prec$-nondecreasing on $[0,\infty]\times[0,1]$. 
Similarly it is shown that $\tG$ is $\prec$-nondecreasing on
$[-\infty,0]\times[0,1]$.
\end{proof}

\begin{proof}[Proof of Proposition~\ref{lem:H}]
Identity \eqref{eq:H-} follows from \eqref{eq:H+} by substituting $-X$ for $X$. 
So, it remains to prove that $H_+(h)=h$ for all $h\in[0,m]$. Fix any $h\in[0,m]$ and write 
$$H_+(h)=\E X\ii{(X,U)\in M},\quad\text{where}\quad
M:=\{(x,u)\in(0,\infty)\times[0,1]\colon\tG(x,u)\le h\}.$$
Introduce also
$$x_h:=\sup\{z\in[0,\infty)\colon G(z-)\le h\}$$
and consider the following cases.

\emph{Case 1: $x_h=\infty$.}\quad Then $m=G(\infty)=G(\infty-)\le h\le m$, whence $h=m$, and so, $G(z)\le h$ and $\tG(z,u)\le h$ for all $z\in[0,\infty)$ and $u\in[0,1]$. That is, $M=(0,\infty)\times[0,1]$ and thus $H_+(h)=\E X\ii{X>0}=m=h$, so that one does have \eqref{eq:H+} in Case~1. 

\emph{Case 2: $x_h<\infty$.}\quad Then $x_h=\max\{z\in[0,\infty)\colon G(z-)\le h\}$ \big(because the function $z\mapsto G(z-)$ is left-continuous on $(0,\infty)$\big). So,
$G(x_h-)\le h<G(z-)$ for all $z>x_h$, whence $G(x_h-)\le h\le G(x_h)$. 
Now one has to distinguish the following two subcases.

\emph{\quad Subcase 2.1: $G(x_h-)=G(x_h)$.}\quad Then $G(x_h-)=h=G(x_h)$ and $M=(0,x_h]\times[0,1]$ \big(because 
(i) $\tG(z,u)\le G(z)\le G(x_h-)=h$ for all $(z,u)\in(0,x_h)\times[0,1]$, 
(ii) $\tG(x_h,u)=h$ for all $u\in[0,1]$, and
(iii) $\tG(z,u)\ge G(z-)>h$ for all $(z,u)\in(x_h,\infty)\times[0,1]$\big). 
It follows that $H_+(h)=\E X\ii{0<X\le x_h}=G(x_h)=h$, whence $H_+(h)=h$ in Subcase~2.1.

\emph{\quad Subcase 2.2: $G(x_h-)<G(x_h)$.}\quad Then
$u_h:=(h-G(x_h-))/(G(x_h)-G(x_h-))\in[0,1]$ and $\tG(x_h,u_h)=h$. Also, reasoning as in Subcase~2.1, here one can see that 
$M=\big((0,x_h)\times[0,1]\big)\,\cup\,\big(\{x_h\}\times[0,u_h]\big)$. 
It follows that 
\begin{align*}
H_+(h)&=\E X\ii{0<X<x_h}+x_h\P(X=x_h)\,\P(U\in[0,u_h])\\
&=G(x_h-)+(G(x_h)-G(x_h-))\,u_h
=\tG(x_h,u_h)=h,
\end{align*}
whence $H_+(h)=h$ in Subcase~2.2 as well.
\end{proof}

\begin{proof}[Proof of Proposition~\ref{prop:cont}]
By Proposition~\ref{lem:H}, for all $h\in[0,m]$, one has 
$0=H_+(h)-H_+(h-)=\E X\ii{X>0,\ \tG(X,U)=h}$, whence $\P(X>0,\ \tG(X,U)=h)=0$; similarly, $\P(X<0,\tG(X,U)=h)=0$; note also that $\tG(X,U)\in[0,m]$ a.s. 
\end{proof}

\begin{proof}[Proof of Proposition~\ref{prop:Y}]
This follows immediately from Proposition~\ref{lem:H}, since, by \eqref{eq:P(Y)}, $\P\big(\tG(Y_\pm,U)\le h\big)=\frac1m\,H_\pm(h)$ and $\P\big(\tG(Y,U)\le h\big)=\frac1{2m}\,\big(H_+(h)+H_-(h)\big)$ for all $h\in[0,m]$.
\end{proof}

\begin{proof}[Proof of Proposition~\ref{prop:finite}]
This follows from Proposition~\ref{prop:cont}. Indeed, in view of property~\eqref{x+-finite} in Proposition~\ref{lem:left-cont} and definition \eqref{eq:r}, 
the event $\{|\r(X,U)|=\infty\}$ is contained in the event $\{X\ne0,\ \tG(X,U)=m\}$. 
\end{proof}

\begin{proof}[Proof of Proposition~\ref{prop:supp}]
For any $h\in[0,m]$, it follows from 
\eqref{eq:L1} that $G(x)<h\le G(x_+(h))$ and hence $\P(X\in(x,x_+(h)])>0$ for all $x\in[0,x_+(h))$, so that $x_+(h)\in\supp X$ provided that $x_+(h)\in(0,\infty)$. 
Similarly, $x_-(h)\in\supp X$ whenever $x_-(h)\in(-\infty,0)$.  
So, for any $u\in(0,1)$ and $x\in\R$ one has $\r(x,u)\in\supp X$ whenever $\r(x,u)\in\R\setminus\{0\}$. 

Next, if $\r(x,u)=0$ for some $x\in\R\setminus\{0\}$ and $u\in(0,1)$, then $\tG(x,u)=0$ \big(by \eqref{eq:r} and property~\eqref{x+-pos} of 
Proposition~\ref{lem:left-cont}\big).
So, by Proposition~\ref{prop:cont}, $\P(X\ne0,\ \r(X,U)=0)=0$.

It remains to recall Proposition~\ref{prop:finite}.
\end{proof}

\begin{proof}[Proof of Proposition~\ref{lem:hat}]\ 

\textbf{(i)} For $x\in[0,\infty]$, one has $G(x)\ge\tG(x,u)=h$, so that $\hat x(x,u)=x_+(h)\le x$, by \eqref{eq:x+}; also, \eqref{eq:hat} clearly implies that here $\hat x(x,u)\ge0$. This proves part~(i) of the proposition; part~(iii) is quite similar.

\textbf{(ii)} Assume that $0\le\hat x(x,u)<x$. 

(a):\ Note that $h=\tG(x,u)\le m<\infty$. 
Take any $h_1\in(h,\infty)$ and then any 
$z\in[0,\infty]$ such that $G(z)\ge h_1$ (if such a point $z$ exists). Then $G(z)>h=\tG(x,u)\ge G(x-)$; since $G$ is nondecreasing on $[0,\infty]$, it follows that $z\ge x$. 
That is, $z\ge x$ for all $z\in[0,\infty]$ such that $G(z)\ge h_1$. So, by \eqref{eq:x+}, $x_+(h_1)\le x$ for all $h_1\in(h,\infty)$ and hence 
$x_+(h+)\ge x>\hat x(x,u)=x_+(h)$. This verifies condition~(a) of part~(ii). 

(b):\ Using the monotonicity of $G$, condition $0\le\hat x<x$, \eqref{eq:hat}, \eqref{eq:L1}, and \eqref{eq:H}, one has $G(x-)\ge G(\hat x)\ge\tG(x,u)\ge G(x-)$. Now condition~(b) of part~(ii) follows. 

(c):\ By just checked condition~(b), $G(\hat x)=G(x-)$. Now condition~(c) follows by \eqref{eq:G(x)}. 

(d), (e):\ Again by condition~(b), $\tG(x,u)=G(x-)$. So, if $u\ne0$ then, by \eqref{eq:H}, $G(x)=G(x-)$, whence condition~(e) follows by (b). 
In turn, condition~(d) follows from (e). 

(f):\ Assume that $u\ne0$ and $x=x_+(h_1)$ for some $h_1\in[0,m]$. On the other hand, by \eqref{eq:hat}, $\hat x=x_+(h)$. So, the condition $\hat x<x$ means that $x_+(h)<x_+(h_1)$, whence $h<h_1$, by property (i) of Proposition~\ref{lem:left-cont}. Also, $x=x_+(h_1)$ implies $G(x)\ge h_1$, by \eqref{eq:L1}. So, $G(x)\ge h_1>h$. This contradicts condition~(e) and thereby verifies condition~(f). 

Thus, part~(ii) of the proposition is proved; part~(iv) is quite similar.

\textbf{(v)} This part follows immediately from parts (i), (ii)(f), (iii), and (iv)(f).
\end{proof}

\begin{proof}[Proof of Proposition~\ref{lem:regul}]
According to parts (ii)(a) and (iv)(a) of Proposition~\ref{lem:hat}, event
$\{\hat x(X,U)\ne X\}$ is contained in event  
$\{X\ne0,\ \tG(X,U)\in D\}$, where 
$D$ stands for the set of all points in $\R$ at which at least one of the monotonic functions $x_+$ or $x_-$ is discontinuous. Since the set $D$ is at most countable, 
$\P\big(\hat x(X,U)\ne X\big)
\le\sum_{h\in D}\P\big(X\ne0,\ \tG(X,U)=h\big)=0,$ by Proposition~\ref{prop:cont}. 
\end{proof}

\begin{proof}[Proof of Proposition~\ref{prop:symm}]
Implications 
\eqref{symm1}$\Rightarrow$
\eqref{symm2}$\Leftrightarrow$
\eqref{symm3}$\Rightarrow$
\eqref{symm4}
follow straight from the corresponding definitions.
Implication 
\eqref{symm4}$\Rightarrow$
\eqref{symm5} 
follows by Proposition~\ref{lem:regul}. 
Implication 
\eqref{symm2}$\Rightarrow$
\eqref{symm1} 
follows by the identity 
\begin{equation*}
	\P(X\in A)=\int_A\tfrac1x\,\d G(x)
\end{equation*}
for all $A\in\B(\R\setminus\{0\})$, which in turn follows from definition \eqref{eq:G(x)}. 
It remains to prove implication \eqref{symm5}$\Rightarrow$
\eqref{symm3}. Toward this end, assume \eqref{symm5} and 
observe the equivalence  $$x_-\big(\tG(x,u)\big)\ne-x_+\big(\tG(x,u)\big) \iff \hat x(x,u)\ne x$$
for all $(x,u)\in\R\times[0,1]$ such that $\r(x,u)=-x$. Therefore and by Proposition~\ref{prop:Y}, \eqref{eq:P(Y)}, and Proposition~\ref{lem:regul}, 
\begin{align*}
\frac1m\int_0^m\ii{x_-(h)\ne -x_+(h)}\,\d h
&=
\P\Big(x_-\big(\tG(Y,U)\big)\ne-x_+\big(\tG(Y,U)\big)\Big) \\ 
&=\P\big(\hat x(Y,U)\ne X\big)
=\frac1{2m}\E|X|\ii{\hat x(X,U)\ne X}=0,	
\end{align*}
so that $x_-=-x_+$ almost everywhere on $[0,m]$ (with respect to the Lebesgue measure) and hence on an everywhere dense subset of $[0,m]$. Now it remains to recall property~\eqref{x+-left-cont} in Proposition~\ref{lem:left-cont}, taking also into account that $x_\pm(0)=0$. 
\end{proof}

\begin{proof}[Proof of Proposition~\ref{lem:hat=rr}]
First of all, $\mathsf{v}$ is Borel by part~\eqref{Borel} of Remark~\ref{rem:cont}.
Next, in the case $x\ge0$, one has 
$y=\r(x,u)=x_-(h)$
and, by \eqref{eq:bet-},  $G(y+)\le h\le G(y)$; so, $v\in[0,1]$ and $h=\tG(y,v)$, whence $\r\big(\r(x,u),v\big)=\r(y,v)=x_+(h)=\hat x(x,u)$;
that is, \eqref{eq:hat=rr} follows in the case $x\ge0$; the case $x\le0$ is quite similar.
\end{proof}

\begin{proof}[Proof of Proposition~\ref{prop:recip}] This follows immediately from Propositions~\ref{lem:hat=rr} and \ref{lem:regul}, on letting $V:=\mathsf{v}(X,U)$.
\end{proof}

\begin{proof}[Proof of Proposition~\ref{prop:main}]
By monotone convergence, without loss of generality (w.l.o.g.) let us assume that the function $g$ is bounded. Now, in view of \eqref{eq:Xab} and the independence of $X$ and $U$, observe that the difference between the left-hand side and the right-hand side of \eqref{eq:main} equals $\E X\psi\big(X,\r(X,U)\big)$, where
$$\psi(x,y):=\frac{g(x,r;x,r)-g(r,x;r,x)}{x-r}\,\ii{xr\le0,x\ne r}$$
for all real $x$ and $r$, so that $\psi(x,r)$ is understood as $0$ if $x=r$. The function $\psi$ is symmetric, and the expression $|x\psi(x,r)|\le|g(x,r;x,r)-g(r,x;r,x)|$ is bounded over all real $x$ and $r$. It remains to refer to Proposition~\ref{prop:cond-mean0}, proved later in this paper. 
\end{proof}

\begin{proof}[Proof of Proposition~\ref{prop:g1g2}] This follows immediately from Proposition~\ref{prop:main} and \eqref{eq:XR<0 or X=R=0}.
\end{proof}

To prove Proposition~\ref{prop:cond-mean0}, we shall use some notation and two lemmas, as follows. 

For all real $a$ and $b$, let
\begin{align}
e_1(a,b)&:=e_{1,X}(a,b):=\E X\ii{X<a,\ \r(X,U)>b} \quad\text{and}\label{eq:e1}\\
e_2(a,b)&:=e_{2,X}(a,b):=\E X\ii{\r(X,U)<a,\ X>b}.\label{eq:e2} 
\end{align}

\begin{lemma}\label{lem:e1}
For all real $a$ and $b$ such that $a\le0\le b$,
\begin{equation*}
	e_1(a,b)=-m+G(a)\vee G(b).
\end{equation*}\end{lemma}

\begin{proof}
Let us consider the following two cases.

\emph{Case 1: $G(a)>G(b)$.}\quad Then for all $x\in\R$ and $u\in(0,1)$
\begin{align}
x<a&\implies x<0\ \&\ \tG(x,u)\ge G(x+)\ge G(a)>G(b)\notag\\
	&\implies x<0\ \&\  x_+(\tG(x,u))>b\label{eq:by L1}\\
	&\implies \r(x,u)>b,\notag
\end{align}
where implication \eqref{eq:by L1} follows from \eqref{eq:L1}. 
So, in this case
\begin{align*}
e_1(a,b)=\E X\ii{X<a}&=\E X\ii{X<0}-\E X\ii{X\in[a,0)}\\
&=-m+G(a)=-m+G(a)\vee G(b).
\end{align*}

\emph{Case 2: $G(a)\le G(b)$.}\quad Then for all $x\in\R$ and $u\in(0,1)$
\begin{align}
x<0\ \&\ \r(x,u)>b &\iff x<0\ \&\ x_+(\tG(x,u))>b\notag\\
&\iff x<0\ \&\ \tG(x,u))>G(b)\label{eq:again by L1}\\
&\implies x<0\ \&\ G(x)\ge\tG(x,u)>G(b)\ge G(a)\notag\\
&\implies x<a,\notag
\end{align}
where equivalence \eqref{eq:again by L1} follows from \eqref{eq:L1}. 
Also, $x<a$ implies $x<0$, since $a\le0$.
So, in Case~2 event $\{X<a,\ \r(X,U)>b\}$ coincides with $\{X<0,\ \r(X,U)>b\}$ and hence, by \eqref{eq:again by L1}, with $\{X<0,\ \tG(X,U)>G(b)\}$. So,
\begin{align*}
e_1(a,b)&=\E X\ii{X<0,\ \tG(X,U)>G(b)}\\
&=\E X\ii{X<0}-\E X\ii{X<0,\ \tG(X,U)\le G(b)}\\
&=-m+G(b)=-m+G(a)\vee G(b),
\end{align*}
where the third equality follows by \eqref{eq:H-}.
\end{proof}

\begin{lemma}\label{lem:e}
For all real $a$ and $b$,
\begin{equation*}
	\E X\ii{X<a,\r(X,U)>b}+\E X\ii{\r(X,U)<a,X>b}=0.
\end{equation*}
\end{lemma}

\begin{proof}
We have to prove that 
$e_1+e_2=0$ on $\R^2$, where $e_1$ and $e_2$ are given by \eqref{eq:e1} and \eqref{eq:e2}. 
Observe that 
\begin{align}
G_{-X}(x)&=G_X(-x);\label{eq:G-}\\
\r_{-X}(x,u)&=-\r_X(-x,u);\notag
\\
e_{1,-X}(x,y)&=-e_{2,X}(-y,-x)\label{eq:e-}
\end{align}
for all real $x$ and $y$ and $u\in(0,1)$.
Let us now consider the four possible cases.

\emph{Case 1: $a\le0\le b$.}\quad Then, by \eqref{eq:e-}, Lemma~\ref{lem:e1}, and \eqref{eq:G-},
\begin{align*}
e_2(a,b)=e_{2,X}(a,b)=-e_{1,-X}(-b,-a)
&=m-G_{-X}(-b)\vee G_{-X}(-a)\\
&=m-G_X(b)\vee G_X(a)=-e_1(a,b),
\end{align*}
again by Lemma~\ref{lem:e1}. So, $e_1(a,b)+e_2(a,b)=0$ in Case~1.

\emph{Case 2: $a>0$ and $b\ge0$.}\quad Then
\begin{equation}
e_1(a,b)=e_1(0,b),\label{eq:e1a0}	
\end{equation}
since the inequalities $\r(X,U)>b$ and $b\ge0$ imply that $\r(X,U)>0$ and hence $X<0$, so that $X<a$. 

Next, in view of condition $b\ge0$ and Proposition~\ref{prop:supp}, one has $\P\big(X>b,\ \r(X,U)=0\big)=0$. 
This implies $e_2(0+,b)=e_2(0,b)$, whence
$e_2(a,b)=e_2(0+,b)=e_2(0,b)=-e_1(0,b)$; the third equality here follows by Case~1. Now, in view of \eqref{eq:e1a0}, one concludes that $e_1(a,b)+e_2(a,b)=0$ in Case 2 as well.

\emph{Case 3: $a\le0$ and $b<0$.}\quad This case follows from Case~2 by \eqref{eq:e-}. 

\emph{Case 4: $b<0<a$.}\quad In this case, taking into account the inequality $X\,\r(X,U)\le0$, one has 
\begin{align*}
e_1(a,b)&=e_2(b+,a-)-e_2(0+,a-)-e_2(b+,0-), \\
e_2(a,b)&=e_1(b+,a-)-e_1(0+,a-)-e_1(b+,0-). 
\end{align*}
By adding these two equalities, one obtains Case~4 from the already considered Cases~1, 2, 3. 
\end{proof}

\begin{proof}[Proof of Proposition~\ref{prop:cond-mean0}]
We have to show that $T(\psi)=0$, where 
$$T(\psi):=\E X\psi\big(X,\r(X,U)\big).$$ 
In view of the identity $\psi=\max(0,\psi)-\max(0,-\psi)$, let us assume w.l.o.g.\ that $\psi\ge0$ on $\R$. Then, by the symmetry of $\psi$, one has the identity $\psi(x,y)=\frac12\int_0^\infty\psi_{A_t}(x,y)\,\d t$ for all real $x$ and $y$, where $A_t:=\{(x,y)\in\R^2\colon\psi(x,y)\ge t\}$ and 
$$\psi_A(x,y):=\ii{(x,y)\in A}+\ii{(y,x)\in A}$$
for all real $x$ and $y$ and all 
$A\in\B(\R^2)$. 
Hence, by Fubini's theorem, it is enough to show that the finite signed measure $\tau$ defined by the formula $\tau(A):=T(\psi_A)$ for $A\in\B(\R^2)$ is zero. 
So, it is enough to show that $\tau(A)=0$ for the sets $A$ of the form $(-\infty,a)\times(b,\infty)$, for all real $a$ and $b$, since the set of all such sets generates the entire $\sigma$-algebra $\B(\R^2)$. 
Now it remains to refer to Lemma~\ref{lem:e}.
\end{proof}

\begin{proof}[Proof of Proposition~\ref{prop:Er/x}]
From Theorem~\ref{th:main} and \eqref{eq:ER/X}, it follows that (i) $\E\frac{\r(X,U)}X\le-1$ always and (ii) $\E\frac{\r(X,U)}X=-1$ iff $\r(X,U)+X=0$ a.s. It remains to use the equivalence \eqref{symm5}$\Leftrightarrow$\eqref{symm1} of Proposition~\ref{prop:symm}. 
\end{proof}

\begin{proof}[Proof of Proposition~\ref{prop:Y,r(Y)}] 
Let $\psi\colon\R^2\to\R$ be any nonnegative Borel function. 
By \eqref{eq:Ef(Y)}, 
\begin{align}
	\E\psi\big(Y_\pm,\r(Y_\pm,U)\big)
	&=\pm\frac1m\,\E X^\pm\psi\big(X,\r(X,U)\big)  
		\quad\text{ and }\quad \notag\\
		\E\psi\big(Y,\r(Y,U)\big)
	&=\frac1{2m}\,\E|X|E\psi\big(X,\r(X,U)\big)\notag\\
	&=\tfrac12\,\E\psi\big(Y_+,\r(Y_+,U)\big) 
	+\tfrac12\,\E\psi\big(Y_-,\r(Y_-,U)\big). \label{eq:Y,Y+-}
\end{align}
Letting now $g_1(x,r;\tx,\tr)\equiv|x|\psi(x,r)$  and $g_2(x,r;\tx,\tr)\equiv|x|\psi(r,x)$ 
in Proposition~\ref{prop:g1g2}, one has 
$\E\psi\big(Y,\r(Y,U)\big)=\E\psi\big(\r(Y,U),Y\big)$, which proves \eqref{eq:Yinterchange}. 

The first equality in \eqref{eq:Y+-interchange} is proved similarly, with $g_1(x,r;\tx,\tr)\equiv x^+\psi(x,r)$  and $g_2(x,r;\tx,\tr)\equiv-x^-\psi(r,x)$.  

To prove the second equality in \eqref{eq:Y+-interchange}, 
let $H:=\tG(Y_-,U)$; then, by Proposition~\ref{prop:Y}, the r.v.\ $H$ is indeed uniformly distributed in $[0,m]$. 
Recall also that $Y_-\le0$ a.s. Therefore and in view of \eqref{eq:hat} and Proposition~\ref{lem:regul}, $x_-(H)=x_-\big(\tG(Y_-,U)\big)=\hat x(Y_-,U)\as Y_-$. On the other hand, by \eqref{eq:r}, $x_+(H)=x_+\big(\tG(Y_-,U)\big)=\r(Y_-,U)$. Hence, $\big(\r(Y_-,U),Y_-\big)\as\big(x_+(H),x_-(H)\big)$. 

The second and third equalities in \eqref{eq:Ysets} follow immediately from \eqref{eq:Y+-interchange}. In turn, these two  equalities imply the first equality in \eqref{eq:Ysets}, in view of \eqref{eq:Y,Y+-} (used with symmetric $\psi$).    

The rest of Proposition~\ref{prop:Y,r(Y)} follows immediately from \eqref{eq:Yinterchange} and \eqref{eq:Y+-interchange}, except for the ``except when'' statement in the parentheses. To prove this latter statement, note that, in view of the inequality $\frac ab+\frac ba<-2$ for all real $a$ and $b$ with $ab<0$ and $a\ne-b$, the equality  
$\E\frac{\r(Y,U)}Y=\E\frac Y{\r(Y,U)}$ implies that 
$\E\frac{\r(Y,U)}Y=\E\frac Y{\r(Y,U)}=\frac12\big(\E\frac{\r(Y,U)}Y+\E\frac Y{\r(Y,U)}\big)<-1$ unless $\r(Y,U)=-Y$ a.s. It remains now to refer to 
\eqref{eq:P(Y)} and Proposition~\ref{prop:symm}. 
\end{proof}

\begin{proof}[Proof 1 of Proposition~\ref{prop:mix}]
W.l.o.g.\ the function $g$ is bounded (by monotone convergence) and nonnegative \big(by the identity $g=\max(0,g)-\max(0,-g)$\big). Write $g(x)=\int_0^\infty g_{A_t}(x)\,\d t$ for all $x\in\R$, where $g_A(x):=\ii{x\in A}$ and $A_t:=\{x\in\R\colon g(x)\ge t\}$. So, by Fubini's theorem, w.l.o.g.\ $g=g_A$ for some  $A\in\B(\R\setminus\{0\})$.  Let then $\la(A)$ and $\rho(A)$ denote, respectively, the left-hand side \big(say $L(g)$\big) and the right-hand side \big(say $R(g)$\big) of \eqref{eq:Eg(X) other} with $g=g_A$. It remains to show that the measures $\la$ and $\rho$ coincide on $\B(\R\setminus\{0\})$. Since the sets $A$ of the form $(-\infty,-b)$ or $(b,\infty)$ for some $b>0$ generate the  $\sigma$-algebra $\B(\R\setminus\{0\})$, it suffices to show that $L\big(g_{(-\infty,-b)}\big)=R\big(g_{(-\infty,-b)}\big)$ and $L\big(g_{(b,\infty)}\big)=R\big(g_{(b,\infty)}\big)$ for all $b>0$. 

Let next
$g_a(x):=x\ii{0<x\le a}$ and observe that 
\begin{equation}\label{eq:g_a}
\int_0^\infty g_a(x)\,\nu(\d a)
=x^+\,\nu\big([x,\infty)\big)
=g_{(b,\infty)}(x)
\end{equation} 
for all $x\in\R$ if $\nu=\nu_b$, where $\nu_b$ is the finite signed measure on $(0,\infty)$ uniquely determined by the condition that 	$x\,\nu_b\big([x,\infty)\big)=\ii{x>b}$ for all  $x\in\R$. 

On the other hand, for any $a>0$, one has $L(g_a)=\E g_a(X)=G(a)$ \big(by \eqref{eq:G(x)}\big) and 
$$R(g_a)=\int_0^m\ii{x_+(h)\le a}\,\d h=\int_0^m\ii{G(a)\ge h}\,\d h=G(a)$$
\big(by \eqref{eq:L1}\big).
So, $L(g_a)=R(g_a)$ for all $a>0$. 

Observe also that $\int_0^\infty |g_a(x)|\,|\nu_b(\d a)|\le c_b\,x^+$ for all $x\in\R$ and $b>0$, where $c_b:=\,\int_0^\infty |\nu_b(\d a)|<\infty$. So,
again by Fubini's theorem \big(and in view of \eqref{eq:g_a}\big), it follows that $L(g_{(b,\infty)})=R(g_{(b,\infty)})$ for all $b>0$. Similarly, $L(g_{(-\infty,-b)})=R(g_{(-\infty,-b)})$ for all $b>0$. 
\end{proof}

\begin{proof}[Proof 2 of Proposition~\ref{prop:mix}]
W.l.o.g.\ the function $g$ in Proposition~\ref{prop:mix} is nonnegative (otherwise, consider its positive and negative parts).  Let $\psi(x):=g(x)/|x|$ for all real $x\ne0$ and $\psi(0):=0$. Then, by \eqref{eq:Ef(Y)} and \cite[Theorem~2.2]{aizen}, 
\begin{equation}\label{eq:aizen}
\begin{aligned}
	\E g(X)&=2m\,\E\psi(Y)
	=2m\,\int_0^1
	\Big(\tfrac12\,\psi\big(Y_1(t)\big)
	+\tfrac12\,\psi\big(Y_2(t)\big)\Big)
	\,\d t \\
	&=m\,\int_0^1
	\Big(\frac{g\big(Y_1(t)\big)}{|Y_1(t)|}
	+\frac{g\big(Y_2(t)\big)}{|Y_2(t)|}\Big)
	\,\d t
	=m\,\int_0^1 \E g(Z_t)\,\frac{\d t}{\E Z_t^{\;+}},
\end{aligned}	
\end{equation}
where $Z_t:=X_{Y_1(t),Y_2(t)}$, $Y_1(t):=F^{-1}(\frac t2)$, $Y_2(t):=F^{-1}(1-\frac t2)$, $F^{-1}(u):=\break
\inf\{y\in\R\colon F(y)\ge u\}$, and, with $G=G_X$, 
\begin{equation*}
	F(y):=\P(Y\le y)=
	\begin{cases}
	\tfrac12-\tfrac1{2m}\,G(y+) &\text{ if }y\le0,\\
	\tfrac12+\tfrac1{2m}\,G(y) &\text{ if }y\ge0,
		\end{cases}
\end{equation*}
the latter equality taking place in view of \eqref{eq:P(Y)} and \eqref{eq:G(x)}.  

Next, fix any $t\in(0,1)$. Then $1-\frac t2>\frac12$, and so, 
\begin{equation}\label{eq:Y_2(t)}
\begin{aligned}
	Y_2(t)
	&=\inf\{y\ge0\colon\tfrac12+\tfrac1{2m}\,G(y)\ge1-\tfrac t2\}\\
	&=\inf\{y\ge0\colon G(y)\ge m(1-t)\}=x_+\big(m(1-t)\big).
\end{aligned} 
\end{equation}
Also, $\frac t2<\frac12$, whence, letting $h:=m(1-t)$, one has $h\in(0,m)$ and 
\begin{equation*}
\begin{aligned}
	Y_1(t)
	&=\inf\{y\le0\colon\tfrac12-\tfrac1{2m}\,G(y+)\ge\tfrac t2\}\\
	&=\inf\{y\le0\colon G(y+)\le h\}.
\end{aligned} 
\end{equation*}
Therefore, $Y_1(t)\in(-\infty,0]$ (since $G(y+)\underset{y\to-\infty}\longrightarrow m>h$), $G\big(Y_1(t)+\big)\le h$ \big(since the function $x\mapsto G(x+)$ is right-continuous on $[-\infty,0]$\big), $G(y+)>h$ for all $y<Y_1(t)$, and so, $G\big(Y_1(t)\big)\ge h$. 
Now \eqref{eq:L1-} yields $Y_1(t)\le x_-(h)$. If at that $Y_1(t)<x_-(h)$ then $G\big(Y_1(t)+\big)\ge G\big(x_-(h)\big)\ge h\ge G\big(Y_1(t)+\big)$, which implies that 
$\tG\big(Y_1(t),0\big)=G\big(Y_1(t)+\big)=h$; 
therefore, by \eqref{eq:hat}, 
$\hat x\big(Y_1(t),0\big)
=x_-\Big(\tG\big(Y_1(t),0\big)\Big)=x_-(h)>Y_1(t)$. 
Hence, by part (iv)(a) of Proposition~\ref{lem:hat}, $Y_1(t)=x_-(h)$ unless $h=m(1-t)$ is a point of discontinuity of the nonincreasing function $x_-$. 
Thus, $Y_1(t)=x_-\big(m(1-t)\big)$ for almost all $t\in(0,1)$. Now \eqref{eq:Eg(X) other} follows in view of \eqref{eq:aizen} and \eqref{eq:Y_2(t)}. 
\end{proof}

\begin{proof}[Proof of Proposition~\ref{prop:tF}]
This is quite similar to the proof of Proposition~\ref{prop:Y}.
\end{proof}

\begin{proof}[Proof of Proposition~\ref{prop:Eg(X) x/r}]
Let $g\colon\R\to\R$ is any Borel function bounded from below (or from above). 
In addition to the function $\Psi_g$ defined by \eqref{eq:Psi}, 
introduce the functions $\Psi_{g,+}$ and $\Psi_{g,-}$ defined by the formulas 
\begin{equation*}
\Psi_{g,+}(h):=\frac{\E g(X_h)}{x_+(h)}\quad\text{and}\quad
\Psi_{g,+}(h):=-\frac{\E g(X_h)}{x_-(h)}
\end{equation*}
for all $h\in(0,m)$, 
so that 
$$\Psi_g=\Psi_{g,+}+\Psi_{g,-}.$$ 

Consider now the case $g(0)=0$. 
Then, in view of \eqref{eq:EPsi}, Proposition~\ref{prop:Y}, \eqref{eq:Ef(Y)}, \eqref{eq:r}, Proposition~\ref{lem:regul}, and \eqref{eq:XR<0 or X=R=0}, 
\begin{align}
\E g(X)&=m\E\Psi_{g,+}(H)+m\E\Psi_{g,-}(H)\notag\\	&=m\E\Psi_{g,+}\big(\tG(Y_+,U)\big)+m\E\Psi_{g,-}\big(\tG(Y_-,U)\big)
\label{eq:+-}\\
	&=\E X^+\Psi_{g,+}\big(\tG(X,U)\big)+\E(-X^-)\Psi_{g,-}\big(\tG(X,U)\big)\notag\\
&=\int_{(0,\infty)\times[0,1]}\,\P(X\in\d x)\,\d u\notag\\
&\qquad\qquad
\times\frac{g\Big(x_-\big(\tG(x,u)\big)\Big)\,x_+\big(\tG(x,u)\big)-g\Big(x_+\big(\tG(x,u)\big)\Big)\,x_-\big(\tG(x,u)\big)}
{x_+\big(\tG(x,u)\big)-x_-\big(\tG(x,u)\big)}\notag\\
&+
\int_{(-\infty,0)\times[0,1]}\,\P(X\in\d x)\,\d u\notag\\
&\qquad\qquad
\times\frac{g\Big(x_+\big(\tG(x,u)\big)\Big)\,x_-\big(\tG(x,u)\big)-g\Big(x_-\big(\tG(x,u)\big)\Big)\,x_+\big(\tG(x,u)\big)}
{x_-\big(\tG(x,u)\big)-x_+\big(\tG(x,u)\big)}\notag\\
	&=\bigg(\int_{(0,\infty)\times[0,1]}+\int_{(-\infty,0)\times[0,1]}\bigg)\,
\frac{g\big(\r(x,u)\big)\,x-g(x)\,\r(x,u)}
{x-\r(x,u)}\,\P(X\in\d x)\,\d u \notag\\
&=
\int_{\R\times[0,1]} 
	\E g\big(X_{x,\r(x,u)}\big)\,
	\P(X\in \d x)\,\d u.\notag
\end{align}
So, identity \eqref{eq:Eg(X) again} is proved in the case when $g(0)=0$. But for $g(x)\equiv\ii{x=0}$, \eqref{eq:Eg(X) again} follows by Proposition~\ref{prop:supp}. So, \eqref{eq:Eg(X) again} is completely proved. The proofs of \eqref{eq:Eg(X) -x/r} and \eqref{eq:Eg(X) 1-x/r} are similar, but using 
$m\E\Psi_{g,+}\big(\tG(Y_-,U)\big)+m\E\Psi_{g,-}\big(\tG(Y_+,U)\big)$ and $2m\E\Psi_{g}\big(\tG(Y,U)\big)$ instead of 
$m\E\Psi_{g,+}\big(\tG(Y_+,U)\big)+m\E\Psi_{g,-}\big(\tG(Y_-,U)\big)$ in \eqref{eq:+-}; \eqref{eq:Eg(X) 1-x/r} is also an obvious corollary of \eqref{eq:Eg(X) again} and \eqref{eq:Eg(X) -x/r}.
\end{proof}

\begin{proof}[Proof of Proposition~\ref{prop:trans1}] 
Since $k$ is bounded from below, 
w.l.o.g.\ one has $k\ge0$. Then w.l.o.g.\ $\E k(\tX_1,\tX_2)<\infty$, since otherwise inequality \eqref{eq:trans} is trivial. 
Just to simplify writing, assume that $I_1=I_2=[0,\infty)$. 
Then $0\le k(\tX_1,0)+k(0,\tX_2)\le k(\tX_1,\tX_2)+k(0,0)$ a.s.\ (by the superadditivity), whence the r.v.'s $k(\tX_1,0)$ and $k(0,\tX_2)$ are integrable, and so are $k(X_1,0)$ and $k(0,X_2)$, 
by \eqref{eq:tX=X}; moreover, $\E k(X_1,0)=\E k(\tX_1,0)$ and $\E k(0,X_2)=\E k(0,\tX_2)$. 

Let $\mu_k$ be the nonnegative measure on $\B\big((0,\infty)^2\big)$ defined by the formula 
\begin{equation*}
	\mu_k\big((a,b]\times(c,d]\big):=k(a,c)+k(b,d)-k(a,d)-k(b,c)
\end{equation*}
for all $a,b,c,d$ in $[0,\infty)$ such that $a<b$ and $c<d$. Then 
\begin{equation}\label{eq:Fubini}
	k(X_1,X_2)=k(X_1,0)+k(0,X_2)-k(0,0)
	+\iint_{(0,\infty)^2}\ii{x_1\le X_1,x_2\le X_2}\,\mu_k(\d x_1,\d x_2)
\end{equation}
a.s.\ Hence, by Fubini's theorem, 
\begin{equation*}
\begin{aligned}
	\E k(X_1,X_2)&=\E k(X_1,0)+\E k(0,X_2)-k(0,0) \\
	&+\iint_{(0,\infty)^2}\E\ii{x_1\le X_1,x_2\le X_2}\,\mu_k(\d x_1,\d x_2). 
\end{aligned}
\end{equation*} 
A similar equality holds with $\tX_1$ and $\tX_2$ in place of $X_1$ and $X_2$. 
Recall that $\E k(X_1,0)=\E k(\tX_1,0)$ and $\E k(0,X_2)=\E k(0,\tX_2)$. 
It remains to observe that $\E\ii{x_1\le X_1,x_2\le X_2}
\le\P(x_1\le X_1)\wedge\P(x_2\le X_2)
=\P(x_1\le\tX_1)\wedge\P(x_2\le\tX_2)
=\E\ii{x_1\le\tX_1,x_2\le\tX_2}$ for all $x_1$ and $x_2$, 
where the latter equality is easy to deduce from \eqref{eq:tX}; 
alternatively, inequality $\E\ii{x_1\le X_1,x_2\le  X_2}\le\E\ii{x_1\le\tX_1,x_2\le\tX_2}$ follows (say) by \cite[Theorem~2]{tchen}, since $(z_1,z_2)\mapsto\ii{x_1\le z_1,x_2\le z_2}$ is a \emph{bounded} right-continuous superadditive function. 
\end{proof}

\begin{proof}[Proof of Proposition~\ref{prop:trans2}] The proof is quite similar to that of \cite[Corollary~2.2(a)]{tchen}. We shall only indicate the necessary changes in that proof, in the notations used there, including the correction of a couple of typos: use the interval $I_\vp:=(\vp,\frac1\vp]$ with $\vp\downarrow0$ instead of $(-B,B]$ and, accordingly, replace $QB$ by $I_\vp^2$; w.l.o.g.\ one may assume here that $h=0$; one does not need to assume that the integral $\vpi\,dH$ at the end of \cite[page~819]{tchen} is finite; on line 1 of \cite[page~820]{tchen}, there should be $\int(h-\vpi)\,d\overline{H}$ and $\liminf$ instead of $\int(h-\vpi)\,dH$ and $\lim$, respectively. 
\end{proof}

\begin{proof}[Proof of Proposition~\ref{prop:best}]\ 

\textbf{\eqref{tilde nu}} 
In view of \eqref{eq:Ef(Y)}, \eqref{eq:Eg(X) nu} \big(with $X^+g(X)$ in place of $g(X)$\big), \eqref{eq:Xab}, and \eqref{eq:tilde nu}, 
\begin{equation}\label{eq:E g(Y+)}
	\E g(Y_+)=\frac1m\,\E X^+\,g(X)
	=\frac1m\,\int_S 
	\frac{y_+(s)\,g\big(y_+(s)\big)\,y_-(s)}{y_-(s)-y_+(s)}\,\nu(\d s)
	=\int_S g\big(y_+(s)\big)\,\tilde\nu(\d s)
\end{equation}
for any bounded Borel function $g\colon\R\to\R$. 
In particular, letting here $g\equiv1$, one sees that $\tilde\nu$ is a probability measure, which proves part \eqref{tilde nu} of the proposition. 

\textbf{\eqref{y+-}} Identity \eqref{eq:E g(Y+)} means that 
$Y_+\D y_+$. Similarly, $Y_-\D y_-$. This proves part \eqref{y+-} of the proposition. It remain to prove part 

\textbf{\eqref{best}} 
The inequality in \eqref{eq:best-k} follows immediately from Propositions~\ref{prop:trans1}, \ref{prop:trans2}, and the just proved part \eqref{y+-} of Proposition~\ref{prop:best}. 
The equalities in \eqref{eq:best-k} follow immediately from relations \eqref{eq:Y+-interchange} and \eqref{eq:Ysets} in Proposition~\ref{prop:Y,r(Y)}. 

As explained in Remark~\ref{rem:best}, relations \eqref{eq:best-p-ratios}, \eqref{eq:best-p-symm}, and \eqref{eq:best-p-width} are special cases of \eqref{eq:best-k}. 
Next, \eqref{eq:best-p-ratio} follows immediately from  \eqref{eq:best-p-ratios}, \eqref{eq:Ysets}, and \eqref{eq:Yinterchange}.  
Finally, \eqref{eq:best-p-width,neg} follows from \eqref{eq:best-k} in view of Proposition~\ref{prop:trans1} with $I_1=I_2=[0,\infty)$ by taking $k(y_1,y_2)\equiv(y_1+y_2+\vp)^p$ for $\vp>0$, and then letting $\vp\downarrow0$ and using the monotone convergence theorem. 

So, part \eqref{best} and thus the entire Proposition~\ref{prop:best} are proved. 
\end{proof}

\begin{proof}[Proof of Proposition~\ref{prop:x+- idents}]
The first identity in \eqref{eq:x+- idents} is a special case of Proposition~\ref{prop:mix}, with $g(x)\equiv\ii{x>0}$; the second identity is quite similar.  
\end{proof}

\begin{proof}[Proof of Proposition~\ref{prop:x+-charact}]
Let
\begin{equation}\label{eq:L}
	L(x):=
\begin{cases}
\sup\{h\in[0,m]\colon y_+(h)\le x\} &\text{ if }x\in[0,\infty],\\
\sup\{h\in[0,m]\colon -y_-(h)\le-x\} &\text{ if }x\in[-\infty,0]. 
\end{cases}
\end{equation}
Then one can check that 
relations \eqref{eq:G properties}, \eqref{eq:L1}, and \eqref{eq:L1-} hold with
the functions $L$ and $y_\pm$ in place of $G$ and $x_\pm$, respectively. 
Introduce also a nonnegative measure $\nu$ on $\B(\R)$ by the formula
\begin{equation}\label{eq:mu}
\nu(A):=\int_{A\setminus\{0\}}\big|\tfrac1x\,\d L(x)\big|	
\end{equation}
for all $A\in\B(\R)$, so that \big(cf.\ \eqref{eq:G(x)}\big) 
\begin{equation}\label{eq:LL}
	L(x)=
\begin{cases}
\int_{(0,x]} z\,\nu(\d z) & \text{ if }x\in[0,\infty], \\
\int_{[x,0)}(-z)\,\nu(\d z) & \text{ if }x\in[-\infty,0]. 
\end{cases}
\end{equation}
Then, using \eqref{eq:L1} (with $L$ and $y_+$ instead of $G$ and $x_+$) 
and Fubini's theorem, one has 
\begin{align*}
&\int_0^m\frac{\d h}{y_+(h)}
=\int_0^m\d h\int_0^\infty\ii{y_+(h)\le\tfrac1t}\,\d t
=\int_0^m\d h\int_0^\infty\ii{L(\tfrac1t)\ge h}\,\d t=\\
&\int_0^\infty L(\tfrac1t)\,\d t=\int_0^\infty\d t\,\int_\R x\ii{0<x\le\tfrac1t}\,\nu(\d x)
=\int_\R \ii{x>0}\,\nu(\d x)=\nu\big((0,\infty)\big). 
\end{align*}
Similarly, $\int_0^m\dfrac{\d h}{-y_-(h)}
=\nu\big((-\infty,0)\big)$. 
So, by condition\eqref{eq:y+- idents}, one has $\nu(\R\setminus\{0\})\le1$. 
So, there exists a unique probability distribution $\mu$ on $\R$ such that $\mu(A)=\nu(A)$ for all $A\in\B(\R\setminus\{0\})$. Let $X$ be any r.v.\ with this distribution $\mu$. Then, by \eqref{eq:LL} and \eqref{eq:L}, one has 
$\E X^+=L(\infty)=m=L(-\infty)=\E(-X^-)$, whence $\E X=0$. 
Also, $G_\mu=L$, and so, in view of \eqref{eq:L1} and \eqref{eq:L1-}, the functions $x_\pm$ for the zero-mean distribution $\mu$ coincide with
$y_\pm$. 
The uniqueness of $\mu$ follows because (i) the functions $x_\pm$ uniquely determine the function $G$ \big(via \eqref{eq:L1} and \eqref{eq:L1-}\big) and (ii) the function $G$ uniquely determines the distribution \big(cf.\ \eqref{eq:mu}\big). 
\end{proof}

\begin{proof}[Proof of Proposition~\ref{prop:s,nu}]\ 

\fbox{Checking (I)$\implies$(II).} Here it is assumed that there exists a zero-mean probability measure $\mu$ on $\B(\R)$ whose reciprocating function $\r:=\r_\mu$ satisfies conditions $\mu_+=\nu$ and $\r_+=\sss$. We have to show at this point that then conditions \eqref{0}--\eqref{mu<1} necessarily take place. 

\textbf{\eqref{0}}\ Since $\r_+=\sss$, one has  $\sss(0,u)=\r(0,u)=0$ for all $u\in[0,1]$, by definition \eqref{eq:r}.   

\textbf{\eqref{nonincr}}\ The conditions $\r=\r_\mu$, $\mu_+=\nu$,  and $\r_+=\sss$ imply that 
\begin{equation}\label{eq:s=x-(tG)}
	\sss(x,u)=x_{-,\mu}\big(\tG_\mu(x,u)\big)\quad\text{for all}\quad (x,u)\in[0,\infty]\times[0,1].
\end{equation}
So, condition \eqref{nonincr} follows by Proposition~\ref{prop:lex}, since $x_{-,\mu}$ is nonincreasing on $[0,m_\mu]$ \big(by part~\eqref{x+-non-decr} of Proposition~\ref{lem:left-cont}\big).

\textbf{\eqref{x-l.c}, \eqref{u-l.c}}\ These conditions follow by \eqref{eq:s=x-(tG)}, Proposition~\ref{prop:lex}, and property~\eqref{x+-left-cont} in Proposition~\ref{lem:left-cont}, because $\tG_\mu(x,0)=G_\mu(x-)$ is left-continuous in $x\in[0,\infty]$ and $\tG_\mu(x,u)$ is affine and hence continuous in $u\in[0,1]$ for every $x\in[0,\infty]$. \big(Note that, if $\tG(x,u)=0$ for some $(x,u)\in(0,\infty]\times[0,1]$, then $\tG(z,v)=0$ for all $(z,v)$ such that $(0,0)\prec(z,v)\prec(x,u)$.\big)

\textbf{\eqref{nu-mass}, \eqref{m<infty}}\ The necessity of these conditions is obvious. 

\textbf{\eqref{nu,s cont1}}\ If $\nu(\{x\})=0$ and $x\in(0,\infty]$ then $\mu(\{x\})=0$, $\tG_\mu(x,1)=\tG_\mu(x,0)$, and so, by  \eqref{eq:s=x-(tG)}, $\sss(x,1)=\sss(x,0)$. 

\textbf{\eqref{nu,s cont2}}\ If $\nu\big((x,y)\big)=0$ and $0\le x<y\le\infty$ then $\mu\big((x,y)\big)=0$, 
$G_\mu(y-)-G_\mu(x)=\int_{(x,y)}z\,\mu(\d z)=0$, and so, by  \eqref{eq:s=x-(tG)}, $\sss(x,1)=x_{-,\mu}\big(G_\mu(x)\big)=x_{-,\mu}\big(G_\mu(y-)\big)=\sss(y,0)$. 

\textbf{\eqref{fin}}\ This follows from \eqref{eq:s=x-(tG)} and property~\eqref{x+-finite} in Proposition~\ref{lem:left-cont}. 

\textbf{\eqref{s<0}}\ Similarly, this follows from \eqref{eq:s=x-(tG)} and property~\eqref{x+-pos} in Proposition~\ref{lem:left-cont}. 

\textbf{\eqref{mu<1}}\ By Proposition~\ref{prop:Y}, the r.v.\ $\tG_\mu(Y_{+,\mu},U)$ is uniformly distributed in the interval $[0,m_\nu]$, where $Y_{+,\mu}$ is 
any r.v. which is independent of $U$ and such that (cf.\ \eqref{eq:P(Y)}) 
$
		\P(Y_{+,\mu}\in A)=
	\frac1{m_\nu}\,\int_A x^+\,\mu(\d x)	
$ 
for all $A\in\B(\R)$. Therefore and in view of \eqref{eq:s=x-(tG)}, 
\begin{equation}\label{x-mu}
	\int_0^{m_\nu}\frac{\d h}{x_{-,\mu}(h)}
	=m\,\E\frac1{x_{-,\mu}(\tG_\mu(Y_{+,\mu},U))}
	=\int_{(0,\infty)\times[0,1]}\frac x{\sss(x,u)}\,\nu(\d x)\,\d u
\end{equation}
(cf.\ \eqref{eq:Ef(Y)}).  
On the other hand, by Proposition~\ref{prop:x+- idents}, $\int_0^{m_\nu}\frac{\d h}{x_{-,\mu}(h)}=-\mu\big((-\infty,0)\big)$. 
Thus, 
\begin{equation}\label{eq:mu<1}
	\int_{(0,\infty)\times[0,1]}\Big(1-\frac x{\sss(x,u)}\Big)\,\nu(\d x)\,\d u
	=\nu\big((0,\infty)\big)+\mu\big((-\infty,0)\big)
	=\mu\big(\R\setminus\{0\}\big)\le1, 
\end{equation}
so that the necessity of condition \eqref{mu<1} is verified. 

\fbox{Checking (II)$\implies$(I).} \ 

\fbox{Step 1.} 
Here \big(assuming the conditions \eqref{0}--\eqref{mu<1} to hold\big) we shall show that there exists a unique function $y_-\colon[0,m_\nu]\to\R$ such that (cf.\ \eqref{eq:s=x-(tG)}) 
\begin{equation}\label{eq:s=y-(tG)} 
	\sss(x,u)=y_-\big(\tG_\nu(x,u)\big)\quad\text{for all}\quad (x,u)\in[0,\infty]\times[0,1].
\end{equation} 

Toward this end, let us first observe that the ``range''  $\tG_\nu([0,\infty]\times[0,1])$ contains the entire interval $[0,m_\nu]$. Indeed, for any given $h\in[0,m_\nu]$, let $x:=x_{+,\nu}(h)$, so that $x\in[0,\infty]$. Then (cf.\ \eqref{eq:bet+}) $G_\nu(x-)\le h\le G_\nu(x)$. Hence, $h=\tG_\nu(x,u)$ for some $u\in[0,1]$. \big(Here, we used \eqref{eq:H} and, tacitly, condition \eqref{m<infty}.\big)

Now, to complete Step~1 it is enough to show for all points $(x,u)$ and $(y,v)$ in $[0,\infty]\times[0,1]$ one has the implication
\begin{equation}\label{eq:function}
	\tG_\nu(x,u)=\tG_\nu(y,v)\implies\sss(x,u)=\sss(y,v).
\end{equation}
Let us assume that indeed $\tG_\nu(x,u)=\tG_\nu(y,v)$; we have to show that $\sss(x,u)=\sss(y,v)$. 
W.l.o.g.\ let us also assume that here $(x,u)\prec(y,v)$, whence, by condition \eqref{nonincr}, $\sss(x,u)\ge\sss(y,v)$. One of the following two cases must take place.

\emph{Case 1: $x=y$ and $u<v$}.\ Then $\tG_\nu(x,u)=\tG_\nu(y,v)$ implies, by \eqref{eq:H}, that $x=0$ or $\nu(\{x\})=0$, whence, by conditions \eqref{0} and \eqref{nu,s cont1}, $\sss(x,1)=\sss(x,0)$, so that, in view of conditions $x=y$ and \eqref{nonincr}, $\sss(y,v)=\sss(x,v)\ge\sss(x,1)=\sss(x,0)\ge\sss(x,u)$. This, together with the inequality $\sss(x,u)\ge\sss(y,v)$, yields $\sss(x,u)=\sss(y,v)$.  

\emph{Case 2: $x<y$}.\ Then 
\begin{equation}\label{eq:chain}
\tG_\nu(x,u)\le\tG_\nu(x,1)=G_\nu(x)\le G_\nu(y-)=\tG_\nu(y,0)\le\tG_\nu(y,v),	
\end{equation}
and so, $\tG_\nu(x,u)=\tG_\nu(y,v)$ implies that all the three inequalities in \eqref{eq:chain} are in fact equalities. 
In particular, one has $\tG_\nu(x,u)=\tG_\nu(x,1)$, which implies, by Case~1, that 
\begin{equation}\label{eq:ident1}
\sss(x,u)=\sss(x,1). 	
\end{equation}
Similarly, one has $\tG_\nu(y,0)=\tG_\nu(y,v)$, which implies that 
\begin{equation}\label{eq:ident2}
\sss(y,v)=\sss(y,0). 	
\end{equation}

To conclude Step~1, recall that one also has $G_\nu(x)=G_\nu(y-)$, whence $0=G_\nu(y-)-G_\nu(x)=\int_{(x,y)}z\,\nu(\d z)$, so that $\nu\big((x,y)\big)=0$, and, by condition 
\eqref{nu,s cont2}, $\sss(x,1)=\sss(y,0)$. This, together with \eqref{eq:ident1} and \eqref{eq:ident2}, yields the desired conclusion $\sss(x,u)=\sss(y,v)$ of implication 
\eqref{eq:function}. 

\fbox{Step 2.}\ Here \big(again assuming the conditions \eqref{0}--\eqref{mu<1} to hold\big) we shall show that all the conditions \eqref{x+-non-decr}--\eqref{x+-left-cont} in Proposition~\ref{lem:left-cont} are satisfied with $y_-$ \big(defined by \eqref{eq:s=y-(tG)}\big) and $m_\nu$ in place of $x_-$ and $m$. 

\textbf{\eqref{x+-non-decr}}\ That $y_-$ is non-increasing on $[0,m_\nu]$ follows because (by Proposition~\ref{prop:lex}) $\tG_\nu$ is $\prec$-non-decreasing $[0,\infty]\times[0,1]$ and, by condition \eqref{nonincr}, $\sss$ is $\prec$-non-increasing. Indeed, take any $h_1$ and $h_2$ such that $0\le h_1<h_2\le m_\nu$. Then there exist some points $(x_1,u_1)$ and $(x_2,u_2)$ in $[0,\infty]\times[0,1]$ 
such that $\tG(x_1,u_1)=h_1$ and $\tG(x_2,u_2)=h_2$; at that, necessarily $(x_1,u_1)\prec(x_2,u_2)$, and so, $y_-(h_1)=\sss(x_1,u_1)\ge\sss(x_2,u_2)=y_-(h_2)$.  

\textbf{\eqref{x+-finite}}\ That $y_-$ is finite on $[0,m_\nu)$ follows by condition \eqref{fin}. Indeed, take any $h\in[0,m_\nu)$
and any $(x,u)\in[0,\infty]\times[0,1]$ such that $\tG_\nu(x,u)=h$. Then $x<\infty$, since $\tG_\nu(\infty,u)=G_\nu(\infty)=m_\nu\ne h$ for all $u\in[0,1]$. Hence, by \eqref{fin}, $\sss(x,u)>-\infty$. Also, $\sss(x,u)\le0$, in view of conditions \eqref{0} and \eqref{nonincr}.
So, by \eqref{eq:s=y-(tG)}, $y_-(h)=y_-\big(\tG_\nu(x,u)\big)=\sss(x,u)\in(-\infty,0]$. 

\textbf{\eqref{x+-pos}}\ First at this point, note that, by \eqref{eq:s=y-(tG)} and condition \eqref{0},  $y_-(0)=\break
y_-\big(\tG_\nu(0,0)\big)=\sss(0,0)=0$. Next, 
take any $h\in(0,m_\nu]$
and any $(x,u)\in[0,\infty]\times[0,1]$ such that $\tG_\nu(x,u)=h$. Then $x>0$, since $\tG_\nu(0,u)=0$ for all $u\in[0,1]$. Hence, again by \eqref{eq:s=y-(tG)}, $y_-(h)=y_-\big(\tG_\nu(x,u)\big)=\sss(x,u)<0$, in view of condition \eqref{s<0}. Thus, indeed $-y_->0$ on $(0,m_\nu]$. 

\textbf{\eqref{x+-left-cont}}\ To complete Step~2 of the proof of Proposition~\ref{prop:s,nu}, we have to show that the function $y_-$ defined by \eqref{eq:s=y-(tG)} 
is left-continuous on $(0,m_\nu]$. 
For brevity, let here $x_h:=x_{+,\nu}(h)$ for all $h$. As was seen before, for each $h\in[0,m_\nu]$ there exists some $u_h\in[0,1]$ such that $\tG_\nu(x_h,u_h)=h$. 
Take any $h_0\in(0,m_\nu]$. Then $x_{h_0}>0$; cf.\ property~\eqref{x+-pos} in Proposition~\ref{lem:left-cont}. 
One of the following two cases must take place.

\emph{Case 1: $p_0:=\nu(\{x_{h_0}\})>0$ and $u_{h_0}>0$.}\ Take any $\de\in\big(0,\min(h_0,p_0\,x_{h_0}u_0)\big)$. 
Take next any $h\in(h_0-\de,h_0]$, so that $h\in(0,h_0]$ and \big(cf.\ property \eqref{x+-non-decr} in Proposition~\ref{lem:left-cont}\big) $x_h\le x_{h_0}$. In fact, we claim that $x_h=x_{h_0}$: 
otherwise, one would have $x_h<x_{h_0}$, whence $h=\tG_\nu(x_h,u_h)\le G(x_{h_0}-)=
\tG_\nu(x_{h_0},u_{h_0})-p_0\,x_{h_0}u_0
=h_0-p_0\,x_{h_0}u_0<h_0-\de<h$, a contradiction. 
So, for all $h\in(h_0-\de,h_0]$, one has $\tG_\nu(x_{h_0},u_h)=h$ and $y_-(h)=\sss(x_{h_0},u_h)$ (by \eqref{eq:s=y-(tG)}). Moreover, $u_h$ is increasing in $h\in(h_0-\de,h_0]$, since $\tG_\nu(x_{h_0},u_h)=\tG_\nu(x_h,u_h)=h$ for all $h\in(h_0-\de,h_0]$. So, by condition \eqref{u-l.c}, $y_-(h)=\sss(x_{h_0},u_h)\to\sss(x_{h_0},u_{h_0})=y_-(h_0)$ as $h\uparrow h_0$. 

\emph{Case 2: $p_0=0$ or $u_{h_0}=0$.}\ Then, with $x_h$ and $u_h$ defined as above, one has $h_0=\tG_\nu(x_{h_0},u_{h_0})=G(x_{h_0}-)=\tG_\nu(x_{h_0},0)$, so that $y_-(h_0)=\sss(x_{h_0},0)$. 
Moreover, for every $h\in(0,h_0)$ one has $0\le x_h<x_{h_0}$ -- because the inequality $x_{h_0}\le x_h$ would imply that $h_0=G(x_{h_0}-)\le G(x_h-)=\tG_\nu(x_h,0)\le\tG_\nu(x_h,u_h)=h$, which contradicts the condition $h\in(0,h_0)$. 
Hence and by condition~\eqref{nonincr}, $\sss(x_{h_0},0)\le\sss(x_h,u_h)\le\sss(x_h,0)$. 
Since $x_h=x_{+,\nu}(h)$ is left-continuous and non-decreasing in $h\in(0,m_\nu)$ \big(cf.\ properties~\eqref{x+-non-decr} and \eqref{x+-left-cont} in Proposition~\ref{lem:left-cont}\big), it follows by condition~\eqref{x-l.c} that $\sss(x_h,0)\to\sss(x_{h_0},0)$ as $h\uparrow h_0$. 
Therefore, by virtue of inequalities $\sss(x_{h_0},0)\le\sss(x_h,u_h)\le\sss(x_h,0)$, one has  $\sss(x_h,u_h)\to\sss(x_{h_0},0)$ as $h\uparrow h_0$. 
In view of equalities \eqref{eq:s=y-(tG)}, $h\equiv\tG_\nu(x_h,u_h)$, and $y_-(h_0)=\sss(x_{h_0},0)$, this means that $y_-(h)\to y_-(h_0)$ as $h\uparrow h_0$. 

This completes Step 2 of the proof. 

\fbox{Step 3.}\ Now we are prepared to complete the entire proof of Proposition~\ref{prop:s,nu}. 
By the just completed Step~2, the function $y_-$ has all the properties \eqref{x+-non-decr}--\eqref{x+-left-cont} listed in Proposition~\ref{lem:left-cont} for $x_-$. 
Letting now $y_+:=x_{+,\nu}$, observe that the function $y_+$ too has the same four properties \big(the proof of these properties of $y_+=x_{+,\nu}$ is practically the same as the corresponding part of the proof of Proposition~\ref{lem:left-cont}\big).  

Observe next that (cf.\ Proposition~\ref{prop:x+- idents}) 
\begin{equation*}
 \int_0^{m_\nu}\frac{\d h}{y_+(h)}=\nu\big((0,\infty)\big)
=\int_{(0,\infty)\times[0,1]}\nu(\d x)\,\d u.
\end{equation*}
Similarly to \eqref{x-mu} but using \eqref{eq:s=y-(tG)} instead of \eqref{eq:s=x-(tG)}, one also has
\begin{equation*}
 \int_0^{m_\nu}\frac{\d h}{y_-(h)}
=\int_{(0,\infty)\times[0,1]} \frac{x}{\sss(x,u)}\,\nu(\d x)\,\d u.
\end{equation*}
Hence, by condition \eqref{mu<1}, 
$\int_0^{m_\nu}\frac{\d h}{y_+(h)}+\int_0^{m_\nu}\frac{\d h}{-y_-(h)}\le1$. 

It follows now by Proposition~\ref{prop:x+-charact} that there exists a unique zero-mean probability measure $\mu$ such that $x_{+,\mu}=y_+$ and $x_{-,\mu}=y_-$. Since $x_{+,\mu}=y_+=x_{+,\nu}$, in view of \eqref{eq:L1}
one has 
\begin{equation}\label{eq:Gmu=Gnu}
G_\mu=G_\nu \quad\text{on}\quad [0,\infty] 
\end{equation}
and hence, by 
condition~\eqref{nu-mass}, $\mu_+=\nu$. 
Also, \eqref{eq:Gmu=Gnu} implies that $\tG_\mu=\tG_\nu$ on $[0,\infty]\times[0,1]$. Now, by virtue of equalities \eqref{eq:s=y-(tG)}, $y_-=x_{-,\mu}$, and \eqref{eq:r}, one has $\r_+=\sss$ (again for $\r:=\r_\mu$). 

Finaly, it remains to prove the uniqueness of $\mu$ given $\mu_+$ and $\r_+$. By Step~1 above, the function $x_{-,\mu}$ is uniquely determined by the condition 
\begin{equation*}
	\r_+(x,u)=x_{-,\mu_+}\big(\tG_{\mu_+}(x,u)\big)\quad\text{for all}\quad (x,u)\in[0,\infty]\times[0,1]
\end{equation*}
(cf.\ \eqref{eq:s=x-(tG)} and \eqref{eq:s=y-(tG)}). 
Also, the function $x_{+,\mu}$ is uniquely determined by $\tG_{\mu_+}$. It remains to refer to the uniqueness part of the statement of Proposition~\ref{prop:s,nu}. 
\end{proof}

\begin{proof}[Proof of Proposition~\ref{prop:s(x,u)}]\ 

\fbox{Checking (I)$\implies$(II)} To prove this implication, assume that condition (I) holds. Then conditions \ref{prop:s,nu}(II)\eqref{0}--\eqref{u-l.c} follow by Proposition~\ref{prop:s,nu}. Let us now prove that conditions \ref{prop:s(x,u)}(II)(\ref{fin}'), (\ref{s<0}'), (\ref{mu<1}') also hold. Let $b:=b_\sss$ and $a:=a_\sss$. 

\fbox{Checking (\ref{fin}'):} By (already established) condition
\ref{prop:s,nu}(II)\eqref{nonincr} and definitions~\eqref{eq:M_s} and \eqref{eq:a,b}, 
$(b,\infty]\times[0,1]\subseteq M_\sss(-\infty)$ while $M_\sss(-\infty)\cap\big([0,b)\times[0,1]\big)=\emptyset$. Let us consider the cases $b=\infty$ and $b<\infty$ separately. 

\emph{Case 1: $b=\infty$.} Then $\sss>-\infty$ on $[0,\infty)\times[0,1]$. 
Omitting the subscript ${}_\mu$ everywhere, recalling \eqref{eq:r}, and using the fact that $\mu(\{\infty\})=0$, one has $\sss(\infty,u)=\r(\infty,u)=x_-\big(\tG(\infty,0)\big)=x_-\big(\tG(\infty,1)\big)=\r(\infty,1)=\sss(\infty,1)$ for all $u\in[0,1]$, so that either $M_\sss(-\infty)=\{\infty\}\times[0,1]=[b,\infty]\times[0,1]$ or $M_\sss(-\infty)=\emptyset$, depending on whether $\sss(\infty,1)=-\infty$ or not. 
Thus, (\ref{fin}') holds in Case~1. 


\emph{Case 2: $b\in[0,\infty)$.} Then, by definition \eqref{eq:a,b}, convention $b=b_\sss$, and monotonicity condition \ref{prop:s,nu}(II)\eqref{nonincr}, there exists a nonincreasing sequence $(x_n)$ in $[0,\infty)$ such that $(x_n,1)\in M_\sss(-\infty)$ for all $n$ and $x_n\searrow b$. So, for all $n$ one has $\sss(x_n,1)=-\infty$ and, by condition \ref{prop:s,nu}(II)\eqref{fin}, $\tG(x_n,1)=m$ or, equivalently,  $G(x_n)=m$. Hence, by the right continuity of $G$ on $[0,\infty)$,  $G(b)=m$, and so, for all $n$ one has $G(b)=G(x_n)$, which (together with $\sss=\r_+$) in turn yields $\sss(b,1)=\sss(x_n,1)=-\infty$. In particular, in view of condition \ref{prop:s,nu}(II)\eqref{0}, $b\ne0$, so that $b\in(0,\infty)$. 
Now let us consider separately the two subcases of Case~2, depending on whether $\sss(b,u)>-\infty$ for all $u\in[0,1)$. 

\emph{Subcase 2.1: $\sss(b,u)>-\infty$ for all $u\in[0,1)$.} Then $M_\sss(-\infty)=\{(b,1)\}\cup(b,\infty]\times[0,1]$, so that (\ref{fin}') holds.

\emph{Subcase 2.2: $\sss(b,u_b)=-\infty$ for some $u_b\in[0,1)$.} Then, by \ref{prop:s,nu}(II)\eqref{fin}, $\tG(b,u_b)=m$. Now, if it were true that $\tG(b,0)<m$, then it would follow that $m=\tG(b,u_b)=(1-u_b)\tG(b,0)+u_b\tG(b,1)<m$, since $0\le u_b<1$.  This contradiction shows that $\tG(b,0)=m$ and hence for all $u\in[0,1]$ one has $m\ge\tG(b,u)\ge\tG(b,0)=m$, and so, $\tG(b,u)=m=\tG(b,u_b)$. This implies $\sss(b,u)=\sss(b,u_b)=-\infty$ for all $u\in[0,1]$. So, $M_\sss(-\infty)=[b,\infty]\times[0,1]$, and again (\ref{fin}') holds.

This completes the checking of (\ref{fin}'). 

\fbox{Checking (\ref{s<0}'):} Here, note first that, by \eqref{0}, $(0,u)\in M_\sss(0)$ for all $u\in[0,1]$. Also,
$[0,a)\times[0,1]\subseteq M_\sss(0)$ while $M_\sss(0)\cap\big((a,\infty]\times[0,1]\big)=\emptyset$. Let us consider the cases $a=0$ and $a>0$ separately. 

\emph{Case 1: $a=0$.} Then 
$\sss<0$ on $(0,\infty]\times[0,1]$. So, 
$M_\sss(0)=\{0\}\times[0,1]=[0,a]\times[0,1]$, and (\ref{s<0}') holds. 

\emph{Case 2: $a\in(0,\infty]$.} Then 
by definition \eqref{eq:a,b}, convention $a=a_\sss$, and monotonicity condition \ref{prop:s,nu}(II)\eqref{nonincr}, there exists a nondecreasing sequence $(x_n)$ in $(0,\infty]$ such that $(x_n,0)\in M_\sss(0)$ for all $n$ and $x_n\nearrow a$. Therefore, for all $n$ one has $\sss(x_n,0)=0$ and, by condition \ref{prop:s,nu}(II)\eqref{x-l.c}, $\sss(a,0)=0$. 
Now let us consider separately the two subcases of Case~2, depending on whether $\sss(a,u)<0$ for all $u\in(0,1]$. 

\emph{Subcase 2.1: $\sss(a,u)<0$ for all $u\in(0,1]$.} Then $M_\sss(0)=[0,a)\times[0,1]\cup\{(a,0)\}$. 
At that, $a<\infty$ -- because $\mu(\{\infty\})=0$, and so, by condition~\ref{prop:s,nu}(II)\eqref{nu,s cont1}, $a=\infty$ would imply $\sss(a,1)=\sss(a,0)=0$, which would contradict the definition of Subcase~2.1. Also, $a>0$, since $\sss(a,u)=0$ for $a=0$ and all $u\in[0,1]$, by condition~\ref{prop:s,nu}(II)\eqref{0}.
Thus, (\ref{s<0}') holds.

\emph{Subcase 2.2: $\sss(a,u_a)=0$ for some $u_a\in(0,1]$.} Then, by \ref{prop:s,nu}(II)\eqref{s<0}, $\tG(a,u_a)=0$. Now, if it were true that $\tG(a,1)>0$, then it would follow that $0=\tG(a,u_a)=(1-u_a)\tG(a,0)+u_a\tG(a,1)>0$, since $0<u_a\le1$.  This contradiction shows that $\tG(a,1)=0$ and hence for all $u\in[0,1]$ one has $\tG(a,u)=0$, and so, $\sss(a,u)=0$. This implies $M_\sss(0)=[0,a]\times[0,1]$, and again (\ref{s<0}') holds.

This completes the checking of (\ref{s<0}'). 

\fbox{Checking (\ref{mu<1}'):} Suppose that indeed $\sss(a,1)\ne\sss(a,0)$. Then, by \ref{prop:s,nu}(II)(\ref{0},\ref{nu,s cont1},\ref{nonincr}), one has $a\,\mu(\{a\})>0$ and $\sss(a,1)<\sss(a,0)\le0$. So,  
\begin{equation*}
	\int_{[0,1]}\frac a{-\sss(a,u)}\;\mu(\{a\})\,\d u
\le	\int_{(0,\infty)\times[0,1]}\Big(1-\frac x{\sss(x,u)}\Big)\,\mu(\d x)\,\d u\le1 
\end{equation*}
by \ref{prop:s,nu}(II)\eqref{mu<1}, 
whence (\ref{mu<1}') follows. 

Thus, implication (I)$\implies$(II) is proved.  

\fbox{Checking (II)$\implies$(I)} To prove this implication, assume that condition (II) indeed holds. Then, in view of 
(\ref{fin}'), $b$ is never $0$ \big(in particular, $b=\infty$ if $M_\sss(-\infty)=\emptyset$\big). Also, by \ref{prop:s,nu}(II)\eqref{nonincr}, $a\le b$. So, one has one of the following three cases: $0\le a<b\le\infty$, $a=b=\infty$, and $0<a=b<\infty$, which we shall consider separately. 

\emph{Case 1: $0\le a<b\le\infty$.} Then consider the function $\vpi$ defined on the set $R_a:=[a,\infty)\cap(0,\infty)$ by the formula 
\begin{equation}\label{eq:vpi}
		\vpi(x):=\int_0^1\Big(1+x-\frac x{\sss(x,u)}\Big)\,\d u;
\end{equation}
here, $\frac x{\sss(x,u)}:=0$ if $\sss(x,u)=-\infty$; also, $\vpi(x):=\infty$ if $\sss(x,u)=0$ for all $u\in[0,1]$. One can see that this definition is correct; moreover, $\vpi(x)\in(1,\infty]$ for all $x\in R_a$, and $\vpi(x)=\infty$ for some $x\in R_a$ only if $x=a$ and $\sss(a,1)=\sss(a,0)$. Indeed, if $x>a$ then, by \ref{prop:s,nu}(II)(\ref{nonincr},\ref{0}) and \eqref{eq:a,b}, $\sss(x,u)\le\sss(x,0)<0$ for all $u\in[0,1]$, so that $0<\int_0^1\frac x{-\sss(x,u)}\,\d u<\infty$. In particular, 
for all $x\in(a,b)$ one has $\vpi(x)\in(1,\infty)$.
It also follows that, if $x\in R_a$ and $\sss(x,u)=0$ for some but not all $u\in[0,1]$, then necessarily $x=a$ and, by \ref{prop:s,nu}(II)\eqref{nonincr}, $\sss(x,1)\ne\sss(x,0)$, so that (\ref{mu<1}') implies that the integral $\int_0^1\frac x{-\sss(x,u)}\,\d u$ is finite \big(and, by \ref{prop:s,nu}(II)(\ref{0},\ref{nonincr}), it is also strictly positive\big). Let also $\vpi(x):=0$ for $x\in\R\setminus R_a$. 

Now one is prepared to introduce the measure $\nu_1$ by the formula
\begin{align*}
	\nu_1(A)&:=\int_{A\cap(a,b)} \frac{p_1(x)\,\d x}{\vpi(x)}
\end{align*}
for all $A\in\B(\R)$, where $p_1$ is any probability density function that is strictly positive on $(a,b)$ and zero elsewhere. 
Recall that $\vpi(x)\in(1,\infty)$ for all $x\in(a,b)$. So, the measure $\nu_1$ is finite, nonzero,  nonnegative, and non-atomic (even more, it is absolutely continuous, with respect to the Lebesgue measure). Moreover, 
$\int_\R\vpi\,\d\nu_1=\int_{(a,b)}p_1(x)\,\d x=1$, 
$0<\nu_1(\R)=\nu_1\big((0,\infty)\big)<1$, and $\supp\nu_1=\R\cap[a,b]$.

Next, introduce the (possibly empty) set
\begin{equation}\label{eq:set D}
	D:=\{x\in(0,\infty]\colon\sss(x,1)\ne\sss(x,0)\}.  
\end{equation}
Observe that 
\begin{equation}\label{eq:D}
	D\subseteq[a,b].
\end{equation}
Indeed, if $x>b$ then, by \ref{prop:s(x,u)}(\ref{fin}'), $(x,u)\in M_\sss(-\infty)$ for all $u\in[0,1]$, whence $\sss(x,0)=-\infty=\sss(x,1)$, and so, $x\notin D$. Similarly, if $0\le x<a$ then, by \ref{prop:s(x,u)}(\ref{s<0}'), $(x,u)\in M_\sss(0)$ for all $u\in[0,1]$, whence $\sss(x,0)=0=\sss(x,1)$, and so, $x\notin D$. 

Observe also that the set $D$
is at most countable; this follows because, for any $x$ and $y$ in $D$ such that $x<y$, one has, by \ref{prop:s,nu}\eqref{nonincr}, $\sss(y,1)<\sss(y,0)\le\sss(x,1)<\sss(x,0)$, so that the open intervals $\big(\sss(x,1),\sss(x,0)\big)$  and $\big(\sss(y,1),\sss(y,0)\big)$ are disjoint. So, there exists a function $p_2\colon D\to(0,\infty)$ such that $\sum_{x\in D}p_2(x)=1$ (unless $D=\emptyset$). Take now any such function $p_2$ and introduce the finite nonnegative discrete measure $\nu_2$ by the formula
\begin{equation}\label{eq:nu2}
	\nu_2(A):=\sum_{x\in A\cap D} \frac{p_2(x)}{\vpi(x)}
\end{equation}
for all $A\in\B(\R)$, where $\vpi$ is still given by \eqref{eq:vpi}.
Since $D\subseteq R_a$ and 
$\vpi(x)=\infty$ for some $x\in R_a$ only if $x=a$ and $\sss(a,1)=\sss(a,0)$, it follows that $1<\vpi(x)<\infty$ for all $x\in D$. Therefore, definition \eqref{eq:nu2} is correct; moreover, $\nu_2(\{x\})>0$ for all $x\in D$, while $0\le\nu_2(\R)=\nu_2\big((0,\infty)\big)<1$. Furthermore,
$\int_\R\vpi\,\d\nu_2=\sum_{x\in D}p_2(x)=1$ (unless $D=\emptyset$, in which case $\nu_2$ is the zero measure). 

Let $\psi(x):=\int_0^1\Big(1-\frac x{\sss(x,u)}\Big)\,\d u$ for $x\in R_a$, so that $1\le\psi\le\vpi$ on $R_a$, and hence $0<\nu_1(\R)=\nu_1(R_a)\le\int_{R_a}\psi\,\d(\nu_1+\nu_2)\le
\int_{R_a}\vpi\,\d(\nu_1+\nu_2)\le2$. 
Therefore, there exists a finite strictly positive constant $c$ such that for the measure 
\begin{equation}\label{eq:nu_c}
\nu:=c(\nu_1+\nu_2)	
\end{equation}
one has $\int_{R_a}\psi\,\d\nu=1$, so that condition \ref{prop:s,nu}(II)\eqref{mu<1} is satisfied. 
Clearly, conditions \ref{prop:s,nu}(II)(\ref{nu-mass},\ref{m<infty}) will be satisfied as well. 

To complete the consideration of Case~1 in the proof of implication (II)$\implies$(I), it suffices to show that conditions \ref{prop:s,nu}(II)\eqref{nu,s cont1}, \eqref{nu,s cont2}, \eqref{fin}, \eqref{s<0} hold, for any choice of $c\in(0,\infty)$ in \eqref{eq:nu_c}. 

\fbox{Checking \ref{prop:s,nu}(II)\eqref{nu,s cont1}} Suppose that $x\in(0,\infty]$ and $\sss(x,1)\ne\sss(x,0)$, so that $x\in D$ and hence, by \eqref{eq:nu_c}, $\nu(\{x\})=c\,\nu_2(\{x\})>0$. This verifies \ref{prop:s,nu}(II)\eqref{nu,s cont1}.  

\fbox{Checking \ref{prop:s,nu}(II)\eqref{nu,s cont2}} Suppose that $0\le x<y\le\infty$ and $\nu\big((x,y)\big)=0$. Then, by \eqref{eq:nu_c}, $\nu_1\big((x,y)\big)=0$. Since $\supp\nu_1=\R\cap[a,b]$, it follows that either $0\le x<y\le a$ or $b\le x<y\le\infty$. By \ref{prop:s(x,u)}(II)(\ref{s<0}'), $M_\sss(0)\ni(a,0)$, whence $\sss(a,0)=0$, and so, $0\le x<y\le a$ will imply $\sss(y,0)=0$ and $\sss(x,1)=0$, and hence $\sss(x,1)=\sss(y,0)$. In the other case, when $b\le x<y\le\infty$, it follows that $b<\infty$, and so, by \ref{prop:s(x,u)}(II)(\ref{fin}'), $(b,1)\in M_\sss(-\infty)$; therefore, $\sss(b,1)=-\infty$, and hence $\sss(x,1)=-\infty=\sss(y,0)$. This completes the verification of \ref{prop:s,nu}(II)\eqref{nu,s cont2}.

\fbox{Checking \ref{prop:s,nu}(II)\eqref{fin}} Take any $(x,u)\in[0,\infty)\times[0,1]$ such that $\sss(x,u)=-\infty$. Then, by \ref{prop:s(x,u)}(II)(\ref{fin}'), $x\in[b,\infty)$. If, moreover, $x>b$ then, by \eqref{eq:D}, $[x,\infty)\cap D=\emptyset$, whence $\nu\big([x,\infty)\big)=c\,\nu_1\big([x,\infty)\big)+c\,\nu_2\big([x,\infty)\big)=0$; so, $m_\nu\ge\tG_\nu(x,u)\ge G_\nu(x-)=m_\nu-\int_{[x,\infty)}z\,\nu(\d z)=m_\nu$, and we conclude that indeed $\tG_\nu(x,u)=m_\nu$. 

Let now $x=b$, so that the assumption $\sss(x,u)=-\infty$ becomes $\sss(b,u)=-\infty$; this admits only two possible forms of $M_\sss(-\infty)$ of the three ones listed in condition \ref{prop:s(x,u)}(II)(\ref{fin}'). Accordingly, let us consider the following two subcases. 

\emph{Subcase \ref{fin}1: $M_\sss(-\infty)=
\{(b,1)\}\cup\big((b,\infty]\times[0,1]\big)$.} Then necessarily $u=1$, and so, $\tG_\nu(x,u)=\tG_\nu(b,1)=G_\nu(b)=m_\nu-\int_{(b,\infty)}z\,\nu(\d z)=m_\nu$, since $\supp\nu\subseteq[a,b]$. 

\emph{Subcase \ref{fin}2: $M_\sss(-\infty)=
[b,\infty]\times[0,1]$.} Then $\sss(b,0)=-\infty=\sss(b,1)$. So, $b\notin D$ and hence $\nu_2(\{b\})=0$, which yields $\nu(\{b\})=0$. Therefore, $m_\nu\ge\tG_\nu(x,u)=\tG_\nu(b,u)
\ge G_\nu(b-)=m_\nu-\int_{[b,\infty)}z\,\nu(\d z)=m_\nu$, since $\supp\nu\subseteq[a,b]$ and $\nu(\{b\})=0$. 

Thus, in both subcases one has  $\tG_\nu(x,u)=m_\nu$. This completes the verification of \ref{prop:s(x,u)}(II)\eqref{fin}.

\fbox{Checking \ref{prop:s,nu}(II)\eqref{s<0}} This is similar to checking \eqref{fin}, so we shall provide fewer details. Take any $(x,u)\in(0,\infty]\times[0,1]$ such that $\sss(x,u)=0$. Then, by \ref{prop:s(x,u)}(II)(\ref{s<0}'), $x\in[0,a]$. If, moreover, $x<a$ then 
$\nu\big([0,x]\big)=0$; so, $0\le\tG_\nu(x,u)\le G_\nu(x)=\int_{(0,x]}z\,\nu(\d z)=0$, and we conclude that indeed $\tG_\nu(x,u)=0$. 

Let now $x=a$, so that $\sss(a,u)=0$. Consider the two possible subcases. 

\emph{Subcase \ref{s<0}1: $M_\sss(0)=
\big([0,a)\times[0,1]\big)\cup\{(a,0)\}$.} Then $u=0$, and so, $\tG_\nu(x,u)=\tG_\nu(a,0)=G_\nu(a-)=\int_{(0,a)}z\,\nu(\d z)=0$. 

\emph{Subcase \ref{s<0}2: $M_\sss(0)=
[0,a]\times[0,1]$.} Then $\sss(a,0)=0=\sss(a,1)$. So, $a\notin D$ and hence $\nu_2(\{a\})=0$ and $\nu(\{a\})=0$. Therefore, $0\le\tG_\nu(x,u)=\tG_\nu(a,u)
\le G_\nu(a)=\int_{(0,a]}z\,\nu(\d z)=0$, since $\supp\nu\subseteq[a,b]$ and $\nu(\{a\})=0$. 

Thus, in both subcases one has  $\tG_\nu(x,u)=0$. This completes the verification of \eqref{s<0}, and thus the entire proof of implication (II)$\implies$(I) in Case~1. 

\emph{Case 2: $a=b=\infty$.} Then, by \ref{prop:s(x,u)}(II)(\ref{s<0}'), $M_\sss(0)=[0,\infty]\times[0,1]$ and $M_\sss(-\infty)=\emptyset$. So, the function $\sss$ is the constant $0$. Then, letting $\nu=\de_0$, one sees that conditions \ref{prop:s,nu}(II)\eqref{nu-mass}--\eqref{nu,s cont2}, \eqref{mu<1} are all trivial, with $m_\nu=0$ and $\tG_\nu\equiv0$, whence conditions \ref{prop:s,nu}(II)(\ref{fin},\ref{s<0}) are also trivial. 

\emph{Case 3: $0<a=b<\infty$.} Then \ref{prop:s(x,u)}(II)(\ref{fin}') and \ref{prop:s(x,u)}(II)(\ref{s<0}') admit only the following possibility: $M_\sss(0)=
\big([0,a)\times[0,1]\big)\cup\{(a,0)\}$ and $M_\sss(-\infty)=
\{(a,1)\}\cup\big((a,\infty]\times[0,1]\big)$. Let now $\nu:=c\,\de_a$, where $c\in(0,\infty)$. Then conditions \ref{prop:s,nu}(II)(\ref{nu-mass},\ref{m<infty}) are trivial, with $m_\nu=ca$. By \ref{prop:s(x,u)}(II)(\ref{mu<1}'), one can satisfy \eqref{mu<1} by taking a small enough $c$.  

It remains to check conditions \ref{prop:s,nu}(II)\eqref{nu,s cont1}, \eqref{nu,s cont2}, \eqref{fin}, \eqref{s<0} -- in Case~3. 

\fbox{Checking \ref{prop:s,nu}(II)\eqref{nu,s cont1}} Assume here that $\nu(\{x\})=0$ and  $x\in(0,\infty]$. Then either $x\in(0,a)$ or $x\in(a,\infty]$. 
If $x\in(0,a)$ then $\sss(x,1)=0=\sss(x,0)$. 
If $x\in(a,\infty]$ then $\sss(x,1)=-\infty=\sss(x,0)$. This verifies \ref{prop:s,nu}(II)\eqref{nu,s cont1}. 

\fbox{Checking \ref{prop:s,nu}(II)\eqref{nu,s cont2}} Assume here that 
$\nu\big((x,y)\big)=0$ and $0\le x<y\le\infty$. 
Then either $0\le x<y\le a$ or $a\le x<y\le\infty$. 
If $0\le x<y\le a$ then $\sss(x,1)=0=\sss(y,0)$, since both points $(x,1)$ and $(y,0)$ are in $M_\sss(0)=
\big([0,a)\times[0,1]\big)\cup\{(a,0)\}$. 
Similarly, if $a\le x<y\le\infty$ then $\sss(x,1)=-\infty=\sss(y,0)$, since both points $(x,1)$ and $(y,0)$ are in $M_\sss(-\infty)=
\{(a,1)\}\cup\big((a,\infty]\times[0,1]\big)$. 
This verifies \ref{prop:s,nu}(II)\eqref{nu,s cont2}. 

The verification of conditions \ref{prop:s,nu}(II)\eqref{fin},  \eqref{s<0} in Case~3 is similar to that in Case~1, but simpler. 

Thus, the proof of implication (II)$\implies$(I) and thereby the entire proof of Proposition~\ref{prop:s(x,u)} is complete. 
\end{proof}

\begin{proof}[Proof of Proposition~\ref{prop:s(x,u)_uniq}]\ 

\fbox{Checking (I)$\implies$(II).} Assume that condition (I) holds. By Proposition~\ref{prop:s(x,u)}, here it suffices to check that condition \ref{prop:s(x,u)_uniq}(II)(u) holds. To obtain a contradiction, assume that \ref{prop:s(x,u)_uniq}(II)(u) is false. Then, by property \ref{prop:s,nu}(II)\eqref{nonincr}, there exist some $x$ and $y$ such that 
\begin{equation*}
\text{$0\le x<y\le\infty$ but $-\infty<\sss(x,1)=\sss(y,0)<0$. }	
\end{equation*}
Hence, by \eqref{eq:a,b}, $a\le x<y\le b$; in particular, $a<b$. Take now any $\tx$ and $\ty$ such that $x<\tx<\ty<y$. Then, again by property \ref{prop:s,nu}(II)\eqref{nonincr}, 
\begin{equation}\label{eq:s,x,tx,ty,y}
	\sss(x,1)\ge\sss(\tx,1)\ge\sss(\ty,0)\ge\sss(y,0)=\sss(x,1), 
\end{equation}
whence $\sss(\tx,1)=\sss(\ty,0)$. 
Also, $\sss(x,1)\ge\sss(\tx,0)\ge\sss(\tx,1)\ge\sss(y,0)=\sss(x,1)$, and so, 
$\sss(\tx,1)=\sss(\tx,0)$ for all $\tx\in(x,y)$; that is, $(x,y)\cap D=\emptyset$, where $D$ is the set defined by \eqref{eq:set D}. So, there exist $\tx$ and $\ty$ such that 
\begin{equation}\label{eq:cap D}
0\le a<\tx<\ty<b\le\infty,\quad
-\infty<\sss(\tx,1)=\sss(\ty,0)<0,\quad
[\tx,\ty]\cap D=\emptyset. 	
\end{equation}
Now, fix any such $\tx$ and $\ty$ and -- along with the measures $\nu_1$, $\nu_2$, and $\nu$ constructed in the proof of Proposition~\ref{prop:s(x,u)} -- consider the measures $\tnu_1$ and $\tnu$ defined by the formulas 
\begin{equation*}
	\tnu_1(A):=\nu_1\big(A\setminus(\tx,\ty)\big)\quad\text{for all $A\in\B(\R)$\quad and}\quad \tnu:=c\,(\tnu_1+\nu_2),
\end{equation*}
where $c$ again is a finite positive constant. 

Then one can verify that properties \ref{prop:s,nu}(II)\eqref{nu-mass}--\eqref{mu<1} hold with $\tnu$ in place of $\nu$. This verification is quite similar to that done when checking implication (II)$\implies$(I) of Proposition~\ref{prop:s(x,u)}, with the only (if any) substantial difference being in the verification of \eqref{nu,s cont2}. There, another case when the conjuncture of conditions $0\le x<y\le\infty$ and $\nu\big((x,y)\big)=0$ occurs is $\tx\le x<y\le\ty$. But then too (cf.\ \eqref{eq:s,x,tx,ty,y}), it easily follows that $\sss(x,1)=\sss(y,0)$. 

So, by Proposition~\ref{prop:s,nu}, there exist zero-mean probability measures $\mu$ and $\tmu$ such that for $\r:=\r_\mu$ and  $\trr:=\r_{\tmu}$ one has $\mu_+=\nu$, $\tmu_+=\tnu$, and $\r_+=\trr_+=\sss$. Moreover, by the last condition in \eqref{eq:cap D}, 
\begin{equation}\label{eq:tnu=0}
	\tnu\big((\tx,\ty)\big)=0.
\end{equation}

Fix now any $x\in(\tx,\ty)$ and let, for brevity, $\hat x:=\hat x_\mu (x,1)$. If $\hat x<x$ then $\mu\big((\hat x,x)\big)=\nu\big((\hat x,x)\big)\ge c\,\nu_1\big((\hat x,x)\big)>0$, since $c>0$ and $x\in(a,b)=\operatorname{interior}(\supp\nu_1)$; this contradicts properties (i) and (ii)(c) listed in Proposition~\ref{lem:hat}. So, by Proposition~\ref{lem:hat=rr},  
$x=\hat x_\mu (x,1)=\r\big(\r(x,1),v\big)$
for some $v\in[0,1]$. But, by condition (I) of Proposition~\ref{prop:s(x,u)_uniq}, $\r=\trr$, and so, $x=\trr\big(\trr(x,1),v\big)$. 
Therefore, by \eqref{eq:r}, 
\begin{equation}\label{eq:x=x+(h)}
x=x_{+,\tmu}(h)=x_{+,\tnu}(h)	\quad\text{for some $h\in[0,m_{\tnu}]$.}
\end{equation} 

Next, introduce $h_0:=G_{\tnu}(x)$. Then $G_{\tnu}(\tx)=h_0$ \big(because, by \eqref{eq:tnu=0}, $\tnu\big((\tx,x]\big)=0$\big). So, for each $h\in(h_0,m_{\tnu}]$, one has $G_{\tnu}(x)<h$, whence, by \eqref{eq:L1}, $x<x_{+,\tnu}(h)$. 
On the other hand, again by \eqref{eq:L1} and the mentioned relation $G_{\tnu}(\tx)=h_0$, for each $h\in[0,h_0]$ one has $G_{\tnu}(\tx)\ge h$, whence, again by \eqref{eq:L1}, 
$\tx\ge x_{+,\tnu}(h)$, and thus $x>x_{+,\tnu}(h)$. So, $x<x_{+,\tnu}(h)$ for each $h\in(h_0,m_{\tnu}]$, and $x>x_{+,\tnu}(h)$ for each $h\in[0,h_0]$. 
This is a contradiction with \eqref{eq:x=x+(h)}. 
Thus, condition \ref{prop:s(x,u)_uniq}(II)(u) follows, so that implication (I)$\implies$(II) of Proposition~\ref{prop:s(x,u)_uniq} is verified. 


\fbox{Checking (II)$\implies$(I)} Assume that condition (II) holds. We have to show that there exists a \emph{unique} function $\r$ such that $\r_+=\sss$ and $\r$ coincides with the reciprocating function $\r_\mu$ of some zero-mean probability measure $\mu$ on $\B(\R)$. The existence here follows immediately from Proposition~\ref{prop:s(x,u)}. To verify implication (II)$\implies$(I), 
it remains to prove the uniqueness of $\r$. We shall do it in steps. 

\fbox{\emph{Step~1}.} Here we shall prove that the values of $\r(z,v)$ are uniquely determined for all $(z,v)\in(-\infty,0)\times(0,1]$ such that $z$ is in the image, say $S$, of the set $(0,\infty]\times[0,1]$ under the mapping $\sss$. 
More specifically, we shall prove at this step that, if $z=\sss(x,u)\in(-\infty,0)$ for some $(x,u)\in(0,\infty]\times[0,1]$, then $\r(z,v)=x$ for all $v\in(0,1]$. 
Toward that end, fix any any $v\in(0,1]$ and any $(x,u)\in(0,\infty]\times[0,1]$ such that 
\begin{equation}\label{eq:z}
	z:=\sss(x,u)\in(-\infty,0);
\end{equation}
then one also has $\r(x,u)=\sss(x,u)=z$. 
Next, introduce
$$y:=\r(z,v).$$
Then $y\in[0,\infty]$ and, by Proposition~\ref{lem:hat=rr}, there exists some $w\in[0,1]$ such that $\r(y,w)=\r\big(\r(z,v),w\big)=\hat x(z,v)$. On the other hand, $z=\r(x,u)=x_-(h)$ for $h:=\tG(x,u)$. So, 
by property (iv)(f) of Proposition~\ref{lem:hat} and because $v\ne0$, one has $\hat x(z,v)=z$. So, recalling that $\r(y,w)=\hat x(z,v)$, one has
\begin{equation}\label{eq:z=s(y,w)}
  z=\r(y,w)=\sss(y,w).
\end{equation}

Next, consider two cases: $y>x$ and $x>y$, to show that either one effects a contradiction.

\emph{Case 1: $y>x$, } 
so that $0\le x<y\le\infty$. Then, by condition~\ref{prop:s(x,u)_uniq}(II)(u), one of the following two  subcases must take place:
	$\sss(x,1)>\sss(y,0)$ or 
	$\sss(x,1)=\sss(y,0)\in\{-\infty,0\}$. 

\emph{Subcase 1.1: $\sss(x,1)>\sss(y,0)$.} Then, by property \ref{prop:s,nu}(II)\eqref{nonincr} of $\sss$, \eqref{eq:z}, and \eqref{eq:z=s(y,w)}, one has
\begin{equation}\label{eq:chain-z}
 \sss(x,1)\le\sss(x,u)=z=\sss(y,w)\le\sss(y,0)<\sss(x,1),
\end{equation}
which is a contradiction. 

\emph{Subcase 1.2: $\sss(x,1)=\sss(y,0)\in\{-\infty,0\}$.} Then the ``non-strict'' version of \eqref{eq:chain-z}, with the sign $<$ replaced by $\le$, still holds, whence $z\in\{-\infty,0\}$, which contradicts the assumption $(z,v)\in(-\infty,0)\times(0,1]$ made above for Step~1. 


\emph{Case 2: $x>y$.} This case is treated similarly to Case~1, with the roles of the pairs $(x,u)$ and $(y,w)$ interchanged. 

From this consideration of Cases~1 and 2, it follows that $y=x$, that is,  $\r(z,v)=x$. 
This completes Step~1.

\fbox{\emph{Step~2}.} Here we shall prove that the values of $\r(z,v)$ are uniquely determined for all $(z,v)\in(-\infty,0)\times(0,1]$, whether or not $z$ is in the image, $S$, of the set $(0,\infty]\times[0,1]$ under the mapping $\sss$. Toward that end, fix any $(z,v)\in(-\infty,0)\times(0,1]$ and introduce 
$$L_z:=\{x\in[0,\infty]\colon\sss(x,0)\ge z\}\quad\text{and}\quad
x_z:=\sup L_z.$$
Note that $0\in L_z$, so that $L_z\ne\emptyset$ and $x_z\in[0,\infty]$. By the monotonicity and left-continuity properties \ref{prop:s,nu}(II)(\ref{nonincr},\ref{x-l.c}), $L_z=[0,x_z]$, so that 
$$\sss(x,0)\ge z\ \forall x\in[0,x_z]\quad\text{and}\quad
\sss(x,0)<z\ \forall x\in(x_z,\infty].$$
Let next 
$$u_z:=\sup\{u\in[0,1]\colon\sss(x_z,u)\ge z\}.$$
Then $u_z\in[0,1]$ (because $\sss(x_z,0)\ge z$) and, by the monotonicity and left-continuity properties \ref{prop:s,nu}(II)(\ref{nonincr},\ref{u-l.c}), 
$$\sss(x_z,u)\ge z\ \forall u\in[0,u_z]\quad\text{and}\quad
\sss(x_z,u)<z\ \forall u\in(u_z,1].$$
Thus, in terms of the lexicographic order, 
\begin{equation}\label{eq:lex}
	\sss(x,u)
	\begin{cases}
\ge z &\text{ if } (0,0)\preccurlyeq(x,u)\preccurlyeq(x_z,u_z),\\
< z &\text{ if } (x,u)\succ(x_z,u_z),
	\end{cases}
\end{equation}
where, as usual, $(x,u)\preccurlyeq(y,v)$ means that either $(x,u)\prec(y,v)$ or $(x,u)=(y,v)$, and $(x,u)\succ(y,v)$ means that $(y,v)\prec(x,u)$. 

In particular, $\sss(x_z,u_z)\ge z$. 
So, by Step~1, w.l.o.g.\ we may, and shall, assume that 
\begin{equation}\label{eq:check z}
	\check z:=\sss(x_z,u_z)>z. 
\end{equation} 
Note that the pair $(x_z,u_z)$ is uniquely determined by $z$ and the function $\sss$, and hence so is $\check z$. 
Note also that \eqref{eq:check z} implies that $-\infty<\check z\le0$. 

Therefore, to complete Step~2, it suffices to verify that 
$\r(z,v)=\r(\check z,1)$	for all $v\in[0,1]$. 
Indeed, then one will have $\r(z,v)=0$ if $\check z=0$ and (by Step~1 and inequalities $-\infty<\check z\le0$) $\r(z,v)=x_z$ if $\check z\ne0$. 
 
So, to obtain a contradiction, assume that $\r(z,v)\ne\r(\check z,1)$	for some $v\in[0,1]$. Then (cf.\ the monotonicity property \ref{prop:s,nu}(II)\eqref{nonincr}), by \eqref{eq:check z}, $\r(z,v)>\r(\check z,1)$, and so, 
\begin{equation}\label{eq:r>r check}
 \r(z,1)>\r(\check z,1). 
\end{equation}
Let us now consider separately the only two possible cases: $\check z=0$ and $-\infty<\check z<0$. 

\emph{Case~1: $\check z=0$.} Then $0=\check z=\sss(x_z,u_z)$, so that, by \eqref{eq:lex}, \ref{prop:s,nu}(II)\eqref{nonincr}, 
and \eqref{eq:M_s}, $M_\sss(0)=
\{(x,u)\colon(0,0)\preccurlyeq(x,u)\preccurlyeq(x_z,u_z)\}=
[0,x_z)\times[0,1]\cup\{x_z\}\times[0,u_z]$. Comparing this with property \ref{prop:s(x,u)}(II)(\ref{s<0}'), one concludes that either $u_z=0$ or $u_z=1$. Let us show that either of these subcases effects a contradiction. 

\emph{Subcase~1.1: $u_z=0$.} Then, in view of \eqref{eq:lex} and equalities $\sss(x_z,u_z)=\check z=0$, one has 
\begin{equation}\label{eq:sss...}
\sss
	\begin{cases}
=0 &\text{ on }[0,x_z)\times[0,1]\cup\{(x_z,0)\}, \\
<z &\text{ on }\{x_z\}\times(0,1]\cup(x_z,\infty]\times[0,1].
	\end{cases}
\end{equation} 
In particular, $0=\sss(x_z,0)=x_-\big(G(x_z-)\big)$, where the functions $x_-$ and $G$ pertain to any given zero-mean probability measure $\mu$ such that $(\r_\mu)_+=\sss$. 
So, by the strict positivity property \eqref{x+-pos} listed in Proposition~\ref{lem:left-cont}, $G(x_z-)=0$. Take now any $u\in(0,1]$; then, by \eqref{eq:sss...}, $z>\sss(x_z,u)=x_-\big(\tG(x_z,u)\big)$ and hence, by \eqref{eq:L1-}, $G(z)<\tG(x_z,u)$. So, $0\le G(z)\le\lim_{u\downarrow0}\tG(x_z,u)=G(x_z-)=0$, whence $G(z)=0$ and $\r(z,1)=x_+\big(G(z)\big)=x_+(0)=0$, which contradicts inequality \eqref{eq:r>r check}, since $\r(\check z,1)=\r(0,1)=0$. 

\emph{Subcase~1.2: $u_z=1$.} This subcase is similar to Subcase~1.1. Indeed, here 
\begin{equation}\label{eq:sss...'}
\sss
	\begin{cases}
=0 &\text{ on }[0,x_z]\times[0,1], \\
<z &\text{ on }(x_z,\infty]\times[0,1].
	\end{cases}
\end{equation}  
In particular, $0=\sss(x_z,1)=x_-\big(G(x_z)\big)$, whence $G(x_z)=0$. Also, for all $(x,u)\in(x_z,\infty]\times[0,1]$, \eqref{eq:sss...'} (or even \eqref{eq:lex}) yields $z>\sss(x,u)=x_-\big(\tG(x,u)\big)$ and hence, by \eqref{eq:L1-},  $G(z)<\tG(x,u)\le G(x)$. So, 
$0\le G(z)\le\lim_{x\downarrow x_z}G(x)=G(x_z)=0$, whence $G(z)=0$ and $\r(z,1)=0$, which again contradicts inequality \eqref{eq:r>r check}. 
 
\emph{Case~2: $-\infty<\check z<0$.} Then, by Step~1, $\r(\check z,1)=x_z$, so that, in view of \eqref{eq:r>r check}, $\r(z,1)>\r(\check z,1)=x_z$. So, by \eqref{eq:r}, $x_+\big(G(z)\big)>
x_z$. Hence, by 
\eqref{eq:L1}, $
G(x_z)<G(z)$. 
On the other hand, just in the consideration of Subcase~1.2, one can see that $G(z)\le G(x_z)$, which contradicts the just established inequality $
G(x_z)<G(z)$.  

All these contradictions demonstrate that indeed $\r(z,v)=\r(\check z,1)$	for all $v\in[0,1]$, which completes Step~2 in the proof of implication (II)$\implies$(I) of Proposition~\ref{prop:s(x,u)_uniq}. 

\fbox{\emph{Step~3}.} Here we shall conclude the proof of implication (II)$\implies$(I). 
By Step~2, the values of $\r(z,v)$ are uniquely determined for all $(z,v)\in(-\infty,0)\times(0,1]$. Let now $\mu_1$ and $\mu_2$ be any two zero-mean probability measures such that $(\r_{\mu_1})_+=\sss=(\r_{\mu_2})_+$, and let $A_1$ and $A_2$ be, respectively, the sets of all the atoms of $\mu_1$ and $\mu_2$. Then the set $A:=A_1\cup A_2$ is countable. Also, for all $z\in(-\infty,0)\setminus A$ one has $\r_{\mu_1}(z,0)=\r_{\mu_1}(z,1)=r_{\mu_2}(z,1)=r_{\mu_2}(z,0)$. On the other hand (cf.\ condition~\ref{prop:s,nu}(II)\eqref{x-l.c}), for any reciprocating function $\r$, the function $z\mapsto\r(z,0)$ is right-continuous on $[-\infty,0)$. Since the set $(-\infty,0)\setminus A$ is dense in a right neighborhood of any given point 
in $[-\infty,0)$, 
it follows that $\r_{\mu_1}(z,0)=r_{\mu_2}(z,0)$ for all $z\in[-\infty,0)$, which in turn also implies that for all $u\in[0,1]$ one has $\r_{\mu_1}(-\infty,u)=r_{\mu_2}(-\infty,u)$, since $\tG_\mu(-\infty,u)=G_\mu\big((-\infty)+\big)=\tG_\mu(-\infty,0)$ for any 
probability measure $\mu$ on $\B(\R)$. 

Thus, $\r_{\mu_1}(z,u)=r_{\mu_2}(z,u)$ for all $u\in[0,1]$ and all  $z\in[-\infty,0)$. The same trivially holds for all $z\in[0,\infty)$ \big(since $(\r_{\mu_1})_+=(\r_{\mu_2})_+=\sss$\big). 

This completes the proof of implication (II)$\implies$(I) of Proposition~\ref{prop:s(x,u)_uniq}. 

It remains to prove the statement about $\supp(\mu_+)$. Assume that indeed either one of the two mutually equivalent conditions, (I) or (II), holds. Let $\mu$ be 
any zero-mean probability measure $\mu$ on $\B(\R)$ such that $(\r_\mu)_+=\sss$. For brevity, let us again omit subscript ${}_\mu$ everywhere and write $a$ and $b$ for $a_\sss$ and $b_\sss$. Then for all $(x,u)\in(b,\infty)\times[0,1]$ one has $\sss(x,u)=-\infty$, whence, by \ref{prop:s,nu}(II)\eqref{fin}, $G(x)=m$. So, $G(b)=G(b+)=m$ if $b<\infty$. On the other hand, if $b=\infty$, then $G(b)=G(\infty)=m$. Therefore, in all cases $G(b)=m=G(\infty)$, which yields $\mu\big((b,\infty)\big)=0$. 

Similarly, for all $(x,u)\in[0,a)\times[0,1]$ one has $\sss(x,u)=0$, whence, by \ref{prop:s,nu}(II)\eqref{s<0}, $G(x)=0$. So, $G(a-)=0=G(0)$ if $a>0$. On the other hand, if $a=0$, then $G(a-)=G(0)=0$. Therefore, in all cases $G(a-)=0=G(0)$, which yields $\mu\big((0,a)\big)=0$. 

Thus, $\supp(\mu_+)\subseteq[a,b]$. To obtain a contradiction, assume that $\supp(\mu_+)\ne[a,b]$. Then there exist $x$ and $y$ such that $a<x<y<b$ and $\mu\big((x,y)\big)=0$, whence $\nu\big((x,y)\big)=0$. So, by \ref{prop:s,nu}(II)\eqref{nu,s cont2}, $\sss(x,1)=\sss(y,0)$. Consequently, by \ref{prop:s(x,u)_uniq}(II)(u), $\sss(x,1)$ is either $-\infty$ or $0$. But this contradicts condition \eqref{eq:a,b}, since $a<x<b$.  
This contradiction shows that indeed  $\supp(\mu_+)=\R\cap[a_\sss,b_\sss]$. 
The proof of Proposition~\ref{prop:s(x,u)_uniq} is now complete.
\end{proof}

\begin{proof}[Proof of Proposition~\ref{prop:mu-,uniq}]\ Note that, in the proof of implication (I)$\implies$(II) in Proposition~\ref{prop:s(x,u)_uniq}, condition \ref{prop:mu-,uniq}(II)(u) was used only in Step~1, where it was proved that, if $z=\sss(x,u)\in(-\infty,0)$ for some $(x,u)\in[0,\infty]\times[0,1]$, then $\r(z,v)=x$ for all $v\in(0,1]$. 

So, here it suffices to verify that, if $\mu$ is any zero-mean probability measure on $\B(\R)$ such that $\mu_-$ is non-atomic and $(\r_\mu)_+=\sss$, then for all $(x,u)\in[0,\infty]\times[0,1]$
\begin{equation}\label{eq:r(s)}
	\r\big(\sss(x,u)\big)=\inf\{y\in[0,x]\colon\sss(y,u)=\sss(x,u)\};
\end{equation}
here $\r:=\r_\mu$, and we write $\r(z)$ instead of $\r(z,v)$ for any $(z,v)\in[-\infty,0]\times[0,1]$, which is correct since $\mu_-=\mu|_{\B((-\infty,0))}$ is non-atomic and $\r(0,v)=0$ for all $v\in[0,1]$. 
To prove \eqref{eq:r(s)}, take any $(x,u)\in[0,\infty]\times[0,1]$ and introduce
\begin{equation}\label{eq:z,...}
	z:=\sss(x,u),\quad h:=\tG(x,u),\quad E:=\{y\in[0,x]\colon\sss(y,u)=z\};
\end{equation}
again, the subscript ${}_\mu$ is omitted everywhere here. 
Then $z\in[-\infty,0]$ and $x\in E$. Also, by \eqref{eq:r} and $\r_+=\sss$, 
\begin{equation}\label{eq:z=x_-(h)}
	z=x_-(h),
\end{equation}
whence, by \eqref{eq:bet-}, $G(z+)\le h\le G(z)$, and so, $G(z)=h$ \big(since $z\le0$ and $\mu_-$ is non-atomic, so that $G$ is continuous on $[-\infty,0]$\big). Therefore,
	$\r(z)=x_+\big(G(z)\big)=x_+(h)$. 
So, by \eqref{eq:z,...}, to prove \eqref{eq:r(s)} it suffices to check that $x_+(h)=\inf E$.

Take now any $y\in E$. Then, by \eqref{eq:r}, $x_-\big(\tG(y,u)\big)=\r(y,u)=\sss(y,u)=z$. Hence, by \eqref{eq:z=x_-(h)} and \eqref{eq:nonat-}, $\tG(y,u)=h$ \big(because $\mu_-$ is non-atomic and hence $z\,\mu(\{z\})=0$\big). So, by \eqref{eq:H}, $G(y)\ge h$ and, then by \eqref{eq:L1}, $x_+(h)\le y$ -- for any $y\in E$. It follows that $x_+(h)\le\inf E$. So, 
it remains to show that $x_+(h)\ge\inf E$. Assume the contrary: $x_+(h)<\inf E$. Then also $x_+(h)<x$, since $x\in E$. So, there exists some $y_1$ such that $x_+(h)<y_1<x$ and $y_1\notin E$, so that, by \eqref{eq:z,...}, $\sss(y_1,u)\ne z=\sss(x,u)$. 
But $\sss(y_1,u)\ge\sss(x,u)$, because $0<y_1<x$ and the function $\sss=\r_+$ is non-increasing, by \eqref{prop:s,nu}(II)\eqref{nonincr}. 
It follows that
\begin{equation}\label{eq:s(y)>z}
	\sss(y_1,u)>z.
\end{equation}

On the other hand, by \eqref{eq:bet+}, $G\big(x_+(h)\big)\ge h\ge G\big(x_+(h)-\big)$. So, there exists some $v\in[0,1]$ such that $h=\tG\big(x_+(h),v\big)$, whence $\sss\big(x_+(h),v\big)=x_-(h)=z$. 
Also, again by \eqref{prop:s,nu}(II)\eqref{nonincr}, the condition  $x_+(h)<y_1$ implies that $\sss(y_1,u)\le\sss\big(x_+(h),v\big)=z$, which contradicts \eqref{eq:s(y)>z}. 
\end{proof}

\begin{proof}[Proof of Proposition~\ref{prop:s(x),nu}]\ 

\fbox{Checking (I)$\implies$(II).} Here it is assumed that condition (I) of Proposition~\ref{prop:s(x),nu} takes place. By Proposition~\ref{prop:s,nu}, 
conditions \ref{prop:s,nu}(II)\eqref{0}--\eqref{x-l.c}, \eqref{nu-mass}, \eqref{m<infty}, \eqref{fin}, and \eqref{s<0} \big(with $\sss(x)$ and $G(x)$ in place of $\sss(x,u)$ and $\tG(x,u)$\big) will then hold. So, to complete the proof of implication (I)$\implies$(II), it remains to check conditions  (\ref{nu,s cont2}') and (\ref{mu<1}''). 

\fbox{Checking (\ref{nu,s cont2}').} By \ref{prop:s,nu}(II)\eqref{nu,s cont2}, only implication $\Longleftarrow$ in place of $\iff$ needs to be proved here. 
To obtain a contradiction, suppose that 
$0\le x<y\le\infty$, $\nu\big((x,y)\big)>0$, and $\sss(x)=\sss(y)$, so that  $G(x)<G(y)$ and $x_-\big(G(x)\big)=x_-\big(G(y)\big)=:z\le0$ (for brevity, here we omit the subscript ${}_\mu$). Then $G(z)\ge G(y)>G(x)\ge G(z+0)$, by \eqref{eq:L1} and \eqref{eq:bet-}. So, $\mu(\{z\})>0$, which contradicts the condition that measure $\mu$ is non-atomic. 

\fbox{Checking (\ref{mu<1}'').} The verification of this is the same as that of condition \ref{prop:s,nu}(II)\eqref{mu<1}, taking also into account in \eqref{eq:mu<1} that 
$\mu\big(\R\setminus\{0\}\big)=1$, since $\mu$ is non-atomic. 

\fbox{Checking (II)$\implies$(I).} Here it is assumed that condition (II) of Proposition~\ref{prop:s(x),nu} takes place. Then, by Proposition~\ref{prop:s,nu}, there exists a unique zero-mean probability measure $\mu$ on $\B(\R)$ whose  
reciprocating function $\r=\r_\mu$ satisfies the conditions $\mu_+=\nu$ and $\r_+=\sss$. It remains to show that $\mu$ is non-atomic. 
To obtain a contradiction, suppose that $\mu(\{z\})>0$ for some $z\in\R$. Then $z\in(-\infty,0)$, since measure $\mu_+=\nu$ is non-atomic. Introduce $h_1:=G(z+)$ and $h_2:=G(z)$. (Here as well, we omit the subscript ${}_\mu$.) Then $\mu(\{z\})>0$ implies that $0\le h_1<h_2\le m$. Take any $h\in(h_1,h_2)$ and let $x:=x_+(h)$  and $y:=x_+(h_2)$. Then $0\le x\le y\le\infty$. Also, by \eqref{eq:bet+}, $G(x)=h$ and $G(y)=h_2$, since $\mu_+=\nu$ is non-atomic. So, $G(x)<G(y)$, whence $\nu\big((x,y)\big)>0$ and $x<y$. So, by (\ref{nu,s cont2}'), $\sss(y)<\sss(x)$. On the other hand, $\sss(x)=x_-\big(G(x)\big)=x_-(h)$ and $\sss(y)=x_-\big(G(y)\big)=x_-(h_2)\ge z$, by \eqref{eq:L1-}. 
Also, taking any $w\in(z,0)$, one has $G(w)\le G(z+)=h_1<h$ and hence, again by \eqref{eq:L1-}, $x_-(h)<w$. This implies that $\sss(x)=x_-(h)\le z$. Thus, $z\le\sss(y)<\sss(x)\le z$, a contradiction. 
\end{proof}

\begin{proof}[Proof of Proposition~\ref{prop:s(x)}]\ 

\fbox{Checking (I)$\implies$(II).} Here it is assumed that condition (I) of Proposition~\ref{prop:s(x)} takes place. By Proposition~\ref{prop:s(x),nu}, 
conditions \ref{prop:s,nu}(II)\eqref{0}--\eqref{x-l.c} \big(with $\sss(x)$ in place of $\sss(x,u)$\big) will then hold. So, to complete the proof of implication (I)$\implies$(II), it remains to check condition (\ref{nu,s cont2}''). 
In view of the monotonicity condition \ref{prop:s,nu}(II)\eqref{nonincr}, it is enough to show that the conjuncture of conditions  
$0\le x<y\le\infty$, $\sss(x+)<\sss(x)$, and 
$\sss(y)=\sss(x+)$ effects a contradiction. Now, conditions $\sss(y)=\sss(x+)$ and \ref{prop:s,nu}(II)\eqref{nonincr} imply that for all $z\in(x,y)$ one has $\sss(z)=\sss(y)$ and hence, by
\ref{prop:s(x),nu}(II)(\ref{nu,s cont2}'), $\nu\big((z,y)\big)=0$, for $\nu:=\mu_+$. 
So, $\nu\big((x,y)\big)=\lim_{z\downarrow x}\nu\big((z,y)\big)=0$. Using \ref{prop:s(x),nu}(II)(\ref{nu,s cont2}') again, one has $\sss(x)=\sss(y)$. This contradicts the assumptions $\sss(x+)<\sss(x)$ and 
$\sss(y)=\sss(x+)$. 

\fbox{Checking (II)$\implies$(I).} Here it is assumed that condition (II) of Proposition~\ref{prop:s(x)} takes place. 
Let $\vpi(z):=\frac z{1+z}$ for $z\in[0,\infty)$ and $\vpi(\infty):=1$. 
Then, in view of conditions \ref{prop:s,nu}(II)\eqref{0}--\eqref{x-l.c}, the formulas $\si\big((-\infty,0)\big):=0$ and 
\begin{equation}\label{eq:si,vpi}
	\si\big([0,x)\big):=\vpi\big(-\sss(x)\big)
\end{equation}
for all $x\in[0,\infty]$ uniquely determine a finite nonnegative measure $\si$ on $\B(\R)$. 

Observe that $\supp\si$ does not contain isolated points. Indeed, suppose that there exists an isolated point $x\in\supp\si$. Then there exists an open set $O\subseteq\R$ such that $x\in O$ but $(O\setminus\{x\})\cap\supp\si=\emptyset$. It follows that $\si(\{x\})=\si(O)>0$, and so, $\vpi\big(-\sss(x+)\big)-\vpi\big(-\sss(x)\big)=\si(\{x\})>0$, whence $\sss(x+)<\sss(x)$. Now, for any $y\in(x,\infty]$,  (\ref{nu,s cont2}'') yields $\sss(y)<\sss(x+)$, so that $\si\big((x,y)\big)>0$ and $(x,y)\cap\supp\si\ne\emptyset$. 
This contradicts the assumption that $x$ is an isolated point of $\supp\si$. 

So, by a well-known fact (see e.g.\ \cite[Corollary~6.2]{varad}, or \cite[Theorem~5.3]
{abram1} together with \cite[Problem~11.4.5(b)]{abram2}), there exists a non-atomic probability measure, say $\nu_0$, on $\B(\R)$ with 
\begin{equation}\label{eq:supp}
	\supp\nu_0=\supp\si.
\end{equation}

Now one can see that condition \ref{prop:s(x),nu}(II)(\ref{nu,s cont2}') is satisfied for $\nu_0$ in place of $\nu$. Indeed, take any $x$ and $y$ such that $0\le x<y\le\infty$. Then \eqref{eq:supp}, \eqref{eq:si,vpi}, and (\ref{nu,s cont2}'') imply
\begin{equation*}
	\nu_0\big((x,y)\big)=0\iff\si\big((x,y)\big)=0
	\iff\sss(x+)=\sss(y)\implies\sss(x+)=\sss(x).
\end{equation*}
So, $\nu_0\big((x,y)\big)=0$ implies $\sss(x)=\sss(y)$. 
Vice versa, in view of the monotonicity, $\sss(x)=\sss(y)$ implies that $\sss$ is constant on $[x,y]$, whence $\sss(x+)=\sss(y)$, $\si\big((x,y)\big)=0$, and thus $\nu_0\big((x,y)\big)=0$. 
This verifies condition \ref{prop:s(x),nu}(II)(\ref{nu,s cont2}') for $\nu_0$ in place of $\nu$. 

Next, by 
\eqref{eq:M_s} and \eqref{eq:a,b}, $M_\sss(0)=[0,a]\times[0,1]$. Also, $\si\big([0,a)\big)=\vpi\big(-\sss(a)\big)=0$, whence, by \eqref{eq:supp},  $\nu_0\big([0,a)\big)=0$. Since $\nu_0$ is non-atomic, $\nu_0\big([0,a]\big)=0$. 
Also, $\nu_0\big((-\infty,0)\big)=0$, since $\supp\nu_0=\supp\si\subseteq[0,\infty)$. 
So, $\sss<0$ on $\supp\nu_0$. 
Thus, for any $c\in(0,\infty)$, the formula
\begin{equation*}
	\nu(A):=c\int_A\frac{\nu_0(\d z)}{1+z-\frac z{\sss(z)}}\quad
	\text{for all $A\in\B(\R)$}
\end{equation*}
\big($\frac z{\sss(z)}:=0$ if $\sss(z)=-\infty$\big) 
correctly defines a finite non-atomic measure on $\B(\R)$, and at that the measures $\nu$ and $\nu_0$ are absolutely continuous relative each other. 

Hence, $\nu$ satisfies condition \ref{prop:s(x),nu}(II)(\ref{nu,s cont2}'). Also, conditions \ref{prop:s,nu}(II)(\ref{nu-mass},\ref{m<infty})
obviously hold, as well as condition \ref{prop:s(x),nu}(II)(\ref{mu<1}'') -- provided that $c$ is chosen appropriately. 

To complete the proof of implication (II)$\implies$(I), 
it remains to check conditions \ref{prop:s,nu}(II)(\ref{fin},\ref{s<0}) -- for $\sss(x)$ in place of $\sss(x,u)$.  

Checking \ref{prop:s,nu}(II)\eqref{fin}. Assume that $x\in[0,\infty)$ and $\sss(x)=-\infty$. Then, by \ref{prop:s,nu}(II)\eqref{nonincr}, $\sss(\infty)=-\infty=\sss(x)$. So, by the already verified condition \ref{prop:s(x),nu}(II)(\ref{nu,s cont2}'), $\nu\big((x,\infty)\big)=0$. 
 Hence, $G_\nu(x)=G_\nu(\infty)-\int_{(x,\infty)}z\,\nu(\d z)=G_\nu(\infty)=m_\nu$. 

Checking \ref{prop:s,nu}(II)\eqref{s<0}. Assume that $x\in(0,\infty]$ and $\sss(x)=0$. Then, by \ref{prop:s,nu}(II)\eqref{0}, $\sss(0)=0=\sss(x)$. 
So, by \ref{prop:s(x),nu}(II)(\ref{nu,s cont2}'), $\nu\big((0,x)\big)=0$.
Hence and because $\nu$ is non-atomic, $\nu\big((0,x]\big)=0$. It follows that $G_\nu(x)=\int_{(0,x]}z\,\nu(\d z)=0$. 

Thus, implication (II)$\implies$(I) is proved. It remains to note that the uniqueness of $\r$ follows by Proposition~\ref{prop:mu-,uniq}. 
\end{proof}

\begin{proof}[Proof of Proposition~\ref{prop:s(x),nu conn}]\ 

\fbox{Checking (I)$\implies$(II):} Assume that condition (I) takes place. Then conditions $\sss(0)=0$, $\nu\big((-\infty,0)\big)=0$, $m_\nu=G_\nu(\infty)<\infty$, and \ref{prop:s(x),nu}(II)(\ref{mu<1}'') 
hold by Proposition~
\ref{prop:s(x),nu}. 
Condition (\ref{nu,s cont2}'') is obvious, since $\nu=\mu_+$ and $\supp\mu=I$. 
So, to complete the proof of implication (I)$\implies$(II), it remains to check conditions (\ref{nonincr}'), (\ref{x-l.c}'), and (\ref{fin}''). At this point we shall do more: check conditions 
\ref{prop:r cont charact}(II)(\ref{nonincr}'',\ref{x-l.c}'',\ref{fin}''',$\r\circ\r$).

Note that the function $G$ is continuous on $[-\infty,\infty]$, since $\mu$ is non-atomic (again, we omit the subscript ${}_\mu$ everywhere here). Because $\supp\mu=I=\R\cap[a_-,a_+]$, the restriction (say $G_+$) of $G$ to the interval $[0,a_+]$ is strictly increasing and maps $[0,a_+]$ onto $[0,m]$. 
Similarly, the restriction (say $G_-$) of $G$ to $[a_-,0]$ is strictly decreasing and maps $[a_-,0]$ onto $[0,m]$. 
Hence, the function $x_+\colon[0,m]\to\R$ is continuous and strictly increasing, as the inverse to the continuous and strictly increasing function $G_+$, and $x_+$ maps $[0,m]$ onto $[0,a_+]$. Similarly, the function $x_-\colon[0,m]\to\R$ is continuous and strictly decreasing, and it maps $[0,m]$ onto $[a_-,0]$. Now conditions \ref{prop:r cont charact}(II)(\ref{nonincr}'',\ref{x-l.c}'') follow by \eqref{eq:r}, since $\r(x)=x_+\big(G(x)\big)$ for $x\in[0,\infty]$ and $\r(x)=x_-\big(G(x)\big)$ for $x\in[-\infty,0]$. 

Take now any $x\in[a_-,0]$ and let $y:=\r(x)=x_+\big(G(x)\big)=G_+^{-1}\big(G_-(x)\big)\in[0,a_+]$. Then $G_+(y)=G_-(x)$ and 
$\r\big(\r(x)\big)=\r(y)=x_-\big(G(y)\big)=G_-^{-1}\big(G_+(y)\big)=G_-^{-1}\big(G_-(x)\big)=x$. 
Similarly, $\r\big(\r(x)\big)=x$ for all $x\in[0,a_+]$. This proves condition \ref{prop:r cont charact}(II)($\r\circ\r$) as well. 

Further, take any $x\in[-\infty,a_-]$. Then $G(x)=G(a_-)=m$, since $\supp\mu=I=\R\cap[a_-,a_+]$ does not intersect with $[-\infty,a_-)$. Hence, $\r(x)=x_+\big(G(x)\big)=G_+^{-1}(m)=a_+$. Similarly, $\r(x)=a_-$ for all $x\in[a_+,\infty]$. This proves condition \ref{prop:r cont charact}(II)(\ref{fin}''') and thus completes the entire proof of implication (I)$\implies$(II). Moreover, we have shown that (I) implies conditions \ref{prop:r cont charact}(\ref{nonincr}'',\ref{x-l.c}'',\ref{fin}''',$\r\circ\r$). 

\fbox{Checking (II)$\implies$(I):} Assume that condition (II) takes place. Then, by Proposition~\ref{prop:s(x),nu}, it suffices to check conditions  \ref{prop:s,nu}(II)(\ref{nonincr},\ref{x-l.c},\ref{fin},\ref{s<0}), \ref{prop:s(x),nu}(II)(\ref{nu,s cont2}'), and $\supp\mu=I$. Conditions~\ref{prop:s,nu}(II)(\ref{nonincr}) and \ref{prop:s,nu}(II)(\ref{x-l.c}) follow immediately from \ref{prop:s(x),nu conn}(II)(\ref{nonincr}',\ref{fin}'') and, respectively, \ref{prop:s(x),nu conn}(II)(\ref{x-l.c}',\ref{fin}'').

\fbox{Checking \ref{prop:s,nu}(II)(\ref{fin}):} Take any $x\in[0,\infty)$ such that $G_\nu(x)<m_\nu$. Then $x\in[0,a_+)$; indeed, in view of \ref{prop:s(x),nu conn}(II)(\ref{nu,s cont2}'') -- and because $\nu$ is non-atomic and hence $\nu(\{a_+\})=0$, one has $G_\nu(y)=G_\nu(\infty)=m_\nu$ for all $y\in[a_+,\infty]$. So, by \ref{prop:s(x),nu conn}(II)(\ref{nonincr}'), $\sss(x)>\sss(a_+)\ge-\infty$, whence $\sss(x)>-\infty$. 

\fbox{Checking \ref{prop:s,nu}(II)(\ref{s<0}):} Take any $x\in(0,\infty]$. Then, by \ref{prop:s(x),nu conn}(II)(\ref{fin}'',\ref{nonincr}'), $\sss(x)=\sss(x\wedge a_+)<\sss(0)=0$, whence $\sss(x)<0$. 

\fbox{Checking \ref{prop:s(x),nu}(II)(\ref{nu,s cont2}'):} Take any $x$ and $y$ such that $0\le x<y\le\infty$. We have to check the equivalence $\nu\big((x,y)\big)=0\ \iff \sss(x)=\sss(y)$. 

Assume here first that $\nu\big((x,y)\big)=0$. Then $(x,y)\cap[0,a_+]=\emptyset$, by \ref{prop:s(x),nu conn}(II)(\ref{nu,s cont2}''). So, $a_+\le x<y\le\infty$, whence, by by \ref{prop:s(x),nu conn}(II)(\ref{fin}''), $\sss(x)=a_-=\sss(y)$. 

Vice versa, assume that $\sss(x)=\sss(y)$. Then $a_+\le x<y\le\infty$ \big(otherwise, one would have $0\le x<a_+$, and so, by \ref{prop:s(x),nu conn}(II)(\ref{nonincr}',\ref{fin}''), $\sss(x)>\sss(y\wedge a_+)\ge\sss(y)$, a contradiction\big). Hence, $(x,y)\subseteq(a_+,\infty)\subseteq\R\setminus\supp\nu$, by \ref{prop:s(x),nu conn}(II)(\ref{nu,s cont2}''). So, $\nu\big((x,y)\big)=0$.

\fbox{Checking $\supp\mu=I$:} First, by \ref{prop:s(x),nu conn}(II)(\ref{nu,s cont2}''), $G(a_+)=G(\infty)=m$. So, by \ref{prop:s(x),nu conn}(II)(\ref{fin}''), $a_-=\sss(a_+)=x_-\big(G(a_+)\big)=x_-(m)$. 
Hence, $m=G(-\infty)\ge G(a_-)\ge m$, 
by \eqref{eq:L1-}. So, $G(-\infty)=m=G(a_-)$, whence $\mu\big((-\infty,a_-)\big)=0$. 
Therefore, $(-\infty,a_-)\cap\supp\mu=\emptyset$. 

Next, take any $z_1$ and $z_2$ such that $a_-<z_1<z_2<0$. By conditions $\sss(0)=0$ and \ref{prop:s(x),nu conn}(II)(\ref{nonincr}',\ref{x-l.c}',\ref{fin}''), the function $\sss$ continuously decreases on the interval $[0,a_+]$ from $0$ to $a_-$. So, there exist $x_1$ and $x_2$ such that $z_1=\sss(x_1)$, $z_2=\sss(x_2)$, and $0<x_2<x_1<a_+$. Then, by \eqref{eq:r(s)}, $\r(z_1)=\r\big(\sss(x_1)\big)=x_1$ and $\r(z_2)=\r\big(\sss(x_2)\big)=x_2$, whence $\r(z_1)\ne\r(z_2)$. Now, by the natural left ``mirror'' analogue of \ref{prop:s(x),nu}(II)(\ref{nu,s cont2}'), $\mu\big((z_1,z_2)\big)>0$. 
So, $\supp\mu\supseteq\R\cap[a_-,0]$. Recalling now that $(-\infty,a_-)\cap\supp\mu=\emptyset$, one has $(-\infty,0]\cap\supp\mu=\R\cap[a_-,0]$. 
Also, $[0,\infty)\cap\supp\mu=\supp\nu=\R\cap[0,a_+]$. Thus, 
indeed $\supp\mu=\R\cap[a_-,a_+]=I$. 

This completes the proof of implication (II)$\implies$(I). It remains to note that the uniqueness of $\mu$ follows immediately by Proposition~\ref{prop:s,nu}. 
\end{proof}

\begin{proof}[Proof of Proposition~\ref{prop:s(x) conn}]
The implication (I)$\implies$(II) follows immediately from Proposition~\ref{prop:s(x),nu conn}. 

The uniqueness of $\r$ follows immediately 
from Proposition~\ref{prop:mu-,uniq} or Proposition~\ref{prop:s(x)}. 

Finally, implication (II)$\implies$(I) follows by Proposition~\ref{prop:s(x)}, since conditions $\sss(0)=0$ and \ref{prop:s(x),nu conn}(II)(\ref{nonincr}',\ref{x-l.c}',\ref{fin}'') imply \ref{prop:s,nu}(II)(\ref{0}--\ref{x-l.c}) and \ref{prop:s(x)}(II)(\ref{nu,s cont2}''). 
\end{proof}

\begin{proof}[Proof of Proposition~\ref{prop:r cont charact}]
Implication (I)$\implies$(II) of Proposition~\ref{prop:r cont charact} was already proved in the proof of implication (I)$\implies$(II) of Proposition~\ref{prop:s(x),nu conn}. 

Implication (II)$\implies$(I) of Proposition~\ref{prop:r cont charact} follows immediately by Proposition~\ref{prop:s(x) conn}. 
\end{proof}

\begin{proof}[Proof of Proposition~\ref{prop:loc-symm}]
 \fbox{Checking (I)$\implies$(II):} Assume that condition (I) takes place. Then, by Proposition~\ref{prop:r cont charact}, condition \ref{prop:r cont charact}(II) holds. Since $\supp\mu=I$ and in an open neighborhood (say $O$) of $0$ measure $\mu$ has a continuous strictly positive density (say $f$), function $G$ is continuously differentiable in $O$, with $G'(x)=x\,f(x)$ for all $x\in O$. 
Also, by the continuity of $\r$, one has $\r(x)\in O$ for all $x$ in some other open neighborhood (say $O_1$) of $0$ such that $O_1\subseteq O$. 
Next, by \eqref{eq:r} and (say) \eqref{eq:bet+}, \eqref{eq:bet-}, one has 
\begin{equation}\label{eq:G(r)}
 G\big(\r(x)\big)=G(x)\quad\text{for all $x\in[-\infty,\infty]$; }
\end{equation}
once again, the subscript ${}_\mu$ is omitted. So, by the inverse function theorem, $\r$ is differentiable in $O_1\setminus\{0\}$ and, for all $x\in O_1\setminus\{0\}$, 
\begin{equation}\label{eq:r'}
 \r'(x)=\frac{G'(x)}{G'\big(\r(x)\big)}
=\frac{x\,f(x)}{\r(x)\,f\big(\r(x)\big)}
\sim\frac x{\r(x)}
\end{equation}
as $x\to0$. In particular, it follows by \eqref{eq:r'} 
that $\r'$ is continuous in $O_1\setminus\{0\}$. 
It also follows that $\r(x)^2=\int_0^x 2\r(z)\,\r'(z)\,\d z\sim
\int_0^x 2z\,\d z=x^2$, $\r(x)\sim-x$, $\r'(0)=-1$, and, again by \eqref{eq:r'}, $\r'(x)\to-1$ as $x\to0$, so that $\r'$ is continuous at $0$ as well. 
This completes the proof of implication (I)$\implies$(II). 

\fbox{Checking (II)$\implies$(I):} Assume that condition (II) takes place. Let $\sss:=\r_+$. Then one can easily construct a measure $\nu$ on $\B(\R)$ with a density $g:=\frac{\d\nu}{\d x}$ that is continuous and strictly positive on $[0,a_+)$ and 
such that condition \ref{prop:s(x),nu conn}(II) holds. Then condition \ref{prop:s(x),nu conn}(I) holds as well, so that there exists a non-atomic zero-mean probability measure $\mu$ on $\B(\R)$ such that $\supp\mu=I$, $\mu_+=\nu$, and  
the reciprocating function $(\r_\mu)_+=\sss$. Moreover, by the uniqueness part of Proposition~\ref{prop:s(x) conn}, $\r_\mu=\r$. Therefore, identity \eqref{eq:G(r)} holds with $G=G_\mu$. So, for all $x$ in a left neighborhood (l.n.) of $0$ there exists the derivative
\begin{equation}\label{eq:G'}
 G'(x)=G'\big(\r(x)\big)\,\r'(x)
=\r(x)\,g\big(\r(x)\big)\,\r'(x)
\sim x\,g(0)
\end{equation}
as $x\uparrow0$, 
because $\r'(0-)=\r'(0)=-1$ and $\r(x)\sim-x$ as $x\uparrow0$. 
On the other hand, \eqref{eq:G'} implies that $G'$ is strictly negative and continuous in a l.n.\ (say $O_-$) of $0$. So, for all $A\in\B(O_-)$, 
\begin{equation*}
 \mu(A)=\int_A\frac{\d G(x)}x=\int_A\frac{G'(x)\,\d x}x
=\int_A f(x)\,\d x,
\end{equation*}
where $f(x):=\frac{G'(x)}x=\frac{\r(x)}x\,g\big(\r(x)\big)\,\r'(x)$ is continuous on $O_-$ and, by \eqref{eq:G'}, $f(0-)=g(0)$. Gluing the functions $f$ and $g$ together, one sees that indeed the probability measure $\mu$ has a continuous strictly positive density a neighborhood of $0$.  
\end{proof}

\begin{proof}[Proof of Proposition~\ref{prop:a}]\ 

\fbox{Checking (I):} By conditions $\r(0)=0$ and \ref{prop:r cont charact}(II)(\ref{nonincr}'',\ref{x-l.c}'',\ref{fin}'''), for $x$ in the interval $[0,a_+)$ 
the width $\w(x)=|x-\r(x)|$ equals $x-\r(x)$ and hence continuously and strictly increases from $0$ to $a_+-a_-$; 
therefore, the restriction of function $\w$ to the interval $[0,a_+)$
has a unique inverse, say $\w_+^{-1}$, which continuously increases on $[0,a_+-a_-)$ from $0$ to $a_+$.
Similarly, for $y$ in the interval $(a_-,0]$,  
$\w(y)$ equals $\r(y)-y$ and hence continuously and strictly decreases from $a_+-a_-$ to $0$, so that 
the restriction of function $\w$ to the interval $(a_-,0]$
has a unique inverse, say $\w_-^{-1}$, which continuously decreases on $[0,a_+-a_-)$ from $0$ to $a_-$. 
Thus, condition \eqref{eq:a} will hold for all $x\in[0,a_+)$ iff $\aaa(w)=\aaa_+(w):=\big(x+\r(x)\big)|_{x=\w_+^{-1}(w)}$ for all $w\in[0,a_+-a_-)$; 
similarly, \eqref{eq:a} will hold for all $y\in(a_-,0]$ iff $\aaa(w)=\aaa_-(w):=\big(y+\r(y)\big)|_{y=\w_-^{-1}(w)}$ for all $w\in[0,a_+-a_-)$. 
So, to prove the existence and uniqueness of a function $\aaa$ satisfying condition \eqref{eq:a} for all $w\in[0,a_+-a_-)$, it suffices 
to show that $\aaa_+(w)=\aaa_-(w)$ for all $w\in[0,a_+-a_-)$, that is, to verify the implication
\begin{equation}\label{eq:1-to-1}
	x-\r(x)=\r(y)-y\implies y+\r(y)=x+\r(x)
\end{equation}
whenever $a_-<y\le0\le x<a_+$. 
Fix any such $x$ and $y$. 
Again by conditions $\r(0)=0$ and \ref{prop:r cont charact}(II)(\ref{nonincr}'',\ref{x-l.c}'',\ref{fin}'''), $\r$ maps interval $(a_-,0]$ onto $[0,a_+)$. Hence, there exists some $\ty\in(a_-,0]$ such that $x=\r(\ty)$, so that, by condition \ref{prop:r cont charact}(II)($\r\circ\r$), $\r(x)=\ty$. It follows that $\r(y)-y=x-\r(x)=\r(\ty)-\ty$, and so, $y=\ty$, since $\r(y)-y$ strictly decreases in $y\in\R$. Therefore, $\r(y)+y=\r(\ty)+\ty=x+\r(x)$, so that implication \eqref{eq:1-to-1} is verified. 

Next, let us check the strict Lip(1) condition, which is easy to see to be equivalent to the condition that the functions $\xi$ and $\rho$ (defined by \eqref{eq:xi,rho}) are strictly increasing on $[0,a_+-a_-)$. 

For each $w\in[0,a_+-a_-)$,  $x:=\w_+^{-1}(w)$, and $y:=\w_-^{-1}(w)$, one has $w=\w(x)=\w(y)$, 
$x\in[0,a_+)$, and $y\in(a_-,0]$, and so, by \eqref{eq:xi,rho} and  \eqref{eq:a},
\begin{align}
\xi(w)&=\tfrac12(w+\aaa(w))=\tfrac12(x-\r(x)+x+\r(x))=\w_+^{-1}(w)
\quad\text{and}\label{eq:xi=+}\\
\xi(w)&=\tfrac12(w+\aaa(w))=\tfrac12(\r(y)-y+y+\r(y))=\r\big(\w_-^{-1}(w)\big);
\label{eq:xi=-}
\end{align}	
either of these two lines shows that $\xi$ is continuously and strictly increasing on $[0,a_+-a_-)$, from $0$ to $a_+$;
similarly, 
\begin{align}
\rho(w)&=\tfrac12(w-\aaa(w))=\tfrac12(x-\r(x)-x-\r(x))=-\r\big(\w_+^{-1}(w)\big)
\quad\text{and} \label{eq:rho=+}\\
\rho(w)&=\tfrac12(w-\aaa(w))=\tfrac12(\r(y)-y-y-\r(y))=-\w_-^{-1}(w); \label{eq:rho=-}
\end{align}	
either of the last two lines shows that $\rho$ is continuously and strictly increasing on $[0,a_+-a_-)$, from $0$ to $-a_-$. 
It also follows that $\aaa(w)=\xi(w)-\rho(w)\to a_++a_-$ as $w\uparrow a_+-a_-$.
Thus, statement (I) is verified. 

\fbox{Checking (II):} 
As noted in the above proof of statement (I), 
the strict Lip(1) condition on $\aaa$ is equivalent to the condition that the functions $\xi$ and $\rho$ be strictly increasing on $[0,a_+-a_-)$; these functions are also continuous, in view of definition \eqref{eq:xi,rho}, since the function $\aaa$ is Lipschitz and hence continuous. Therefore, the functions $\xi$ and $\rho$ are strictly and continuously increasing on $[0,a_+-a_-)$, from $0$ to $a_+$ and $-a_-$, respectively. So, the functions $\xi$ and $\rho$ have strictly and continuously increasing inverses $\xi^{-1}$ and $\rho^{-1}$, which map  $[0,a_+)$ and $[0,-a_-)$, respectively, onto $[0,a_+-a_-)$. 

Thus, one can use formula \eqref{eq:r,xi,rho} \emph{to define a} function $\r$ on the interval $(a_-,a_+)$. Let us then extend this definition to the entire interval $[-\infty,\infty]$ by imposing condition \ref{prop:r cont charact}(II)(\ref{fin}'''). 
Then one can see that the function $\r$ satisfies condition \ref{prop:r cont charact}(II). 
So, Proposition~\ref{prop:r cont charact} implies that $\r$ is the reciprocating function of a nonatomic zero-mean probability measure $\mu$ on $\B(\R)$ with $\supp\mu=I=\R\cap[a_-,a_+]$. 

Let us now verify \eqref{eq:a}. Take any $x\in[0,a_+)$. Then $x=\xi(w)$ for some $w\in[0,a_+-a_-)$, whence $\xi^{-1}(x)=w$ and, in view of \eqref{eq:r,xi,rho},  
\begin{gather*}
 |x-\r(x)|=x-\r(x)=x+\rho\big(\xi^{-1}(x)\big)
=\xi(w)+\rho(w)=w,\\
x+\r(x)=x-\rho\big(\xi^{-1}(x)\big)
=\xi(w)-\rho(w)=a(w),
\end{gather*}
so that \eqref{eq:a} holds for $x\in[0,a_+)$. Similarly, \eqref{eq:a} holds for $x\in(a_-,0]$. So, \eqref{eq:a} is verified. 

To complete the proof of (II), it remains to check the uniqueness of $\r$ given $\aaa$ and \eqref{eq:a}. That is, we have to show that the value of $\r(x)$ is uniquely determined for each $x\in(a_-,a_+)$. 
In fact, we shall show that, moreover, relations \eqref{eq:r,xi,rho} must necessarily hold. 
Toward that end, observe that, as shown in the above proof of statement (I), condition \eqref{eq:a} implies \eqref{eq:xi=+}--\eqref{eq:rho=-}. 

Now, take any $x\in[0,a_+)$. Then $x=\w_+^{-1}(w)$ for some $w\in[0,a_+-a_-)$, whence, by \eqref{eq:xi=+}, $\xi(w)=x$ and hence $\xi^{-1}(x)=w$; so, by \eqref{eq:rho=+}, $-\rho\big(\xi^{-1}(x)\big)=-\rho(w)=\r\big(\w_+^{-1}(w)\big)=\r(x)$, which proves the first case in \eqref{eq:r,xi,rho}. 

Similarly, take any $y\in(a_-,0]$. Then $y=\w_-^{-1}(w)$ for some $w\in[0,a_+-a_-)$, whence, by \eqref{eq:rho=-}, $\rho(w)=-y$ and hence 
$\rho^{-1}(-y)=w$. So, by \eqref{eq:xi=-}, $\xi\big(\rho^{-1}(-y)\big)=\xi(w)=\r\big(\w_-^{-1}(w)\big)=\r(y)$, which proves the second case in \eqref{eq:r,xi,rho}. 

\fbox{Checking (III), the ``if'' part:} Here, assume that the  asymmetry pattern function $\aaa$ is continuously differentiable in an open r.n.\ of $0$ and $\aaa'(0+)=0$. Then in such a r.n.\ the functions $\xi$ and $\rho$ are continuously differentiable, $\xi'=\frac12(1+\aaa')$, $\xi'(0+)=\frac12$, $\rho'=\frac12(1-\aaa')$, $\rho'(0+)=\frac12$. 
Recall also that the functions $\xi$ and $\rho$ are continuously increasing on $[0,a_+-a_-)$, 
so that the inverse functions $\xi^{-1}$ and $\rho^{-1}$ are continuously differentiable in a r.n.\ of $0$. 
Hence, by \eqref{eq:r,xi,rho}, there is some $\vp>0$ such that
\begin{align*}
x\in(0,\vp) & \implies \r'(x)
=\frac{-2\rho'\big(\xi^{-1}(x)\big)}{1+\aaa'\big(\xi^{-1}(x)\big)}
\underset{x\downarrow0}\longrightarrow-1,\\
x\in(-\vp,0) & \implies \r'(x)
=\frac{-2\xi'\big(\rho^{-1}(-x)\big)}{1-\aaa'\big(\rho^{-1}(-x)\big)}
\underset{x\uparrow0}\longrightarrow-1,
\end{align*}
which shows, in view of the mean value theorem, that indeed the corresponding reciprocating function $\r$ 
is continuously differentiable in a neighborhood of $0$. 

\fbox{Checking (III), the ``only if'' part:} Here, assume that a reciprocating function $\r$ such as in Proposition~\ref{prop:r cont charact}  
is continuously differentiable in a neighborhood of $0$. Then, by Proposition~\ref{prop:loc-symm}, $\r'(0)=-1$. By \eqref{eq:a},  $\aaa\circ\w_+=\al$, where $\w_+\colon[0,a_+)\to[0,a_+-a_-)$, $\al\colon[0,a_+)\to\R$, 
$\w_+(x)\equiv x-\r(x)$ and $\al(x)\equiv x+\r(x)$. At that, the function $\w_+$ is continuously differentiable in some r.n., say $(0,\vp)$, of $0$, with $\w_+'(x)=1-\r'(x)$ for all $x\in(0,\vp)$, so that $\w_+'(0+)=2\ne0$. So, for all $w$ in some open r.n.\ of $0$, one has $\aaa(w)=\al\big(\w_+^{-1}(w)\big)$ and hence 
\begin{equation*}
	\aaa'(w)=\frac{1+\r'\big(\w_+^{-1}(w)\big)} {\w_+'\big(\w_+^{-1}(w)\big)}
	\underset{w\downarrow0}\longrightarrow0, 
\end{equation*}
so that indeed the function $\aaa$ is continuously differentiable in an open right neighborhood of $0$ and $\aaa'(0+)=0$. 
\end{proof}

\subsection{Proofs of the main results} \label{proofs:main} 

\begin{proof}[Proof of Theorem~\ref{th:main}]
This theorem is a special case of Proposition~\ref{prop:main}. 
\end{proof}

\begin{proof}[Proof of Theorem~\ref{th:F}]
For $j=1,\dots,n+1$, introduce
\begin{equation}\label{eq:Phi}
g_j(p_1,\dots,p_n) 
:=\E g(p_1,\dots,p_{j-1},X_{j;p_j},R_{j;p_j},\dots,X_{n;p_n}R_{n;p_n})
\end{equation}
and
$$
\I_j
:= \int_{(\R\times[0,1])^n}
g_j(p_1,\dots,p_n)\,
\d p_1\dots\d p_n.
$$
Then, for all $j=1,\dots,n$,
$$
\I_{j+1}
= \int_{(\R\times[0,1])^{n-1}}\EE_j\,\d p_1\dots\d p_{j-1}\,
\d p_{j+1}\dots\d p_n,
$$
where
\begin{align*} 
\EE_j
&:= \E g_{j+1}(p_1,\dots,p_{j-1},X_j,R_j,p_{j+1},\dots,p_n) \\
&=\int_{\R\times[0,1]} 
\E g_{j+1}(p_1,\dots,p_{j-1},X_{j;p_j},R_{j;p_j},p_{j+1},\dots,p_n)\,\P(X_j\in\d x_j)\,\d u_j 
\\
&=\int_{\R\times\R\times[0,1]}  g_{j+1}(p_1,\dots,p_{j-1},x_{j;p_j},r_{j;p_j},p_{j+1},\dots,p_n)\\
&\qquad\qquad\qquad\qquad\qquad\times
\P\big((X_{j;p_j},R_{j;p_j})
\in\d x_{j;p_j}\times\d r_{j;p_j}\big)\,
\P(X_j\in\d x_j)\d u_j 
\\
&=\int_{\R\times\R\times[0,1]}  \E g(p_1,\dots,p_{j-1},x_{j;p_j},r_{j;p_j},
X_{j+1;p_{j+1}},R_{j+1;p_{j+1}},\dots,X_{n;p_n},R_{n;p_n})\\
&\qquad\qquad\qquad\qquad\qquad\times
\P\big((X_{j;p_j},R_{j;p_j})
\in\d x_{j;p_j}\times\d r_{j;p_j}\big)\,
\P(X_j\in\d x_j)\d u_j 
\\
&=\int_{\R\times[0,1]}  \E g(p_1,\dots,p_{j-1},X_{j;p_j},R_{j;p_j},\dots,X_{n;p_n},R_{n;p_n})\,\P(X_j\in\d x_j)\,\d u_j 
\\
&=\int_{\R\times[0,1]}  g_j(p_1,\dots,p_n)\,\P(X_j\in\d x_j)\,\d u_j
=\int_{\R\times[0,1]}  g_j(p_1,\dots,p_n)\,\d p_j;
\end{align*}
the second of these 7 equalities follows by \eqref{eq:main}, and the fourth and sixth ones by \eqref{eq:Phi}.  

Now it follows that  
$\I_{j+1}=\I_j$,
for all $j=1,\dots,n$. 
This finally implies $\I_{n+1}=\I_1$, so that
\begin{equation*} 
\begin{split}
\E g( & X_1,R_1,\dots,X_n,R_n)\\
&=\I_{n+1}
=\I_1=
\int_{(\R\times[0,1])^n}
\E g(X_{1;p_1}R_{1;p_1},\dots,X_{n;p_n},R_{n;p_n})\,
\d p_1\dots\d p_n.
\end{split}
\end{equation*}
\end{proof}

\begin{proof}[Proof of Corollary~\ref{cor:student-normal}]
Take any function $f\in\H5$ and, for any $x_1,r_1\dots,$ $x_n,r_n$ in $\R$, let  
$$g_f(x_1,r_1\dots,x_n,r_n):=
\begin{cases}
f\bigg(\dfrac{x_1+\dots+x_n}{\frac12\sqrt{w_1^2+\dots+w_n^2}}\bigg) & \text{ if } w_1^2+\dots+w_n^2\ne0, \\
f(0) & \text{ otherwise, }
\end{cases}
$$
where $w_i:=|x_i-r_i|$. Then, by Theorem~\ref{th:F} and \cite[Theorem~2.1]{normal}, 
\begin{align*}
\E f(S_W)&=\E g_f(X_1,R_1,\dots,X_n,R_n) \\
&\le\sup\big\{\E g_f(X_{x_1,r_1},R_{x_1,r_1},\dots,X_{x_n,r_n},R_{x_n,r_n})\colon(x_1,r_1\dots,x_n,r_n)\in\R^{2n},\\ 
& \quad\quad\quad\quad\quad\quad\quad\quad\quad\quad\quad\quad\quad\quad\quad\quad\quad\quad\quad\quad\quad\quad\quad\quad\quad\quad x_jr_j\le0\ \forall j\big\} \\
&=\sup\bigg\{\E
f\bigg(\dfrac{X_{x_1,r_1}+\dots+X_{x_n,r_n}}
{\frac12\sqrt{(x_1-r_1)^2+\dots+(x_n-r_n)^2}}\bigg) \colon(x_1,r_1\dots,x_n,r_n)\in\R^{2n},\\ 
& \quad\quad\quad\quad\quad\quad\quad\quad\quad\quad\quad\quad\quad\quad\quad\quad\quad\quad\quad\quad\quad\quad\quad\quad\quad\quad x_jr_j\le0\ \forall j\bigg\} \\
&\le\E f(Z), 
\end{align*}
which proves \eqref{eq:stud-f(Z)}. 
Now \eqref{eq:stud-P(Z>x)} follows by \cite[Corollary~2.2]{normal}. 
\end{proof}

\begin{proof}[Proof of Corollary~\ref{cor:stud-asymm}]
This proof is similar to that of Corollary~\ref{cor:student-normal}, using 
\cite[Theorem~4 and Corollary~3]{asymm} 
instead of \cite[Theorem~2.1 and Corollary~2.2]{normal}. 
Here we only would like to provide some details concerning condition~\eqref{eq:bounded-asymm}, in its relation with condition \cite[(20)]{asymm}. Namely, we shall show that \eqref{eq:bounded-asymm} implies that
\begin{equation}\label{eq:left}
	\frac{\r_i(y,u)}{|y|}\le\ga:=\frac{1-p}p\quad\text{for all}\ y<0,\ u\in[0,1], 
\end{equation}
and $i\in\{1,\dots,n\}$, which will allow one to immediately apply the mentioned results of \cite{asymm}. 

Fix any $i\in\{1,\dots,n\}$ and write, for brevity, $X$ 
and $\r$ for $X_i$ 
and $\r_i$, respectively. Then, by \eqref{eq:bounded-asymm} and Fubini's theorem, one has $\P(X\in A)=1$, where $A:=\{x>0\colon\mes(B_x)=1\}$, $\mes$ denotes the Lebesgue measure, and $B_x:=\{u\in[0,1]\colon\frac x{|\r(x,u)|}\le\ga\}$. Since $B_x$ is a closed interval, one has $B_x=[0,1]$ and hence $\frac x{|\r(x,0)|}\le\ga$ for all $x\in A$. 

Take now any $y<0$ and let $x_y:=\r(y,1)$ and $h:=\tG(y,1)=G(y)$. Then $x_y=x_+(h)$, 
$G(x_y-)\le h$ \big(by \eqref{eq:bet+}\big), and $0\ge\r(x_y,0)=x_-(G(x_y-))\ge x_-(h)\ge y$ \big(by property~\eqref{x+-non-decr} of Proposition~\ref{lem:left-cont} and \eqref{eq:L1-}\big), so that $|\r(x_y,0)|\le|y|$ and 
\begin{equation}\label{eq:less}
\frac{\r(y,u)}{|y|}\le\frac{x_y}{|y|}\le\frac{x_y}{|\r(x_y,0)|}\quad\text{for all }u\in[0,1]. 	
\end{equation}
Moreover, w.l.o.g.\ $x_y>0$ \big(otherwise, $\r(y,1)=x_y=0$ and hence $\r(y,u)=0$ for all $u\in[0,1]$, so that \eqref{eq:left} is trivial\big). 
On the other hand, for each $x\in[0,x_y)$, by \eqref{eq:less+} and \eqref{eq:L1} one has $G(x)<h\le G(x_y)$ and hence $\P(X\in(x,x_y])>0$, so that $(x,x_y]\cap A\ne\emptyset$ (since $\P(X\in A)=1$). Therefore, there exists a non-decreasing sequence $(x_n)$ in $A$ such that $x_n\uparrow x_y$. So, $\ga_n:=\frac{x_n}{|\r(x_n,0)|}\le\ga$ for all $n$ and, in view of property~\eqref{x+-left-cont} of  Proposition~\ref{lem:left-cont}, $\frac{x_y}{|\r(x_y,0)|}=\lim_n\ga_n\le\ga$, since $\r(x,0)=x_-(G(x-))$ for all $x>0$ and the function $(0,\infty)\ni x\mapsto G(x-)$ is non-decreasing and left-continuous. 

Thus, $\frac{x_y}{|\r(x_y,0)|}\le\ga$, whence, by \eqref{eq:less}, inequality\eqref{eq:left} follows. 
\end{proof}

\end{document}